\documentclass[11pt]{amsart}
\usepackage{amsthm,amssymb,mathrsfs,mathtools,empheq}
\usepackage[shortlabels]{enumitem}
\usepackage[T1]{fontenc}

\usepackage[notref,notcite,final]{showkeys}
\mathtoolsset{showonlyrefs}
\usepackage{color}
\usepackage{geometry}
\allowdisplaybreaks[3]

\newtheorem{theorem}{Theorem}[section]
\newtheorem*{theorem*}{Theorem}
\newtheorem{corollary}[theorem]{Corollary}

\newtheorem{lemma}[theorem]{Lemma}
\newtheorem{proposition}[theorem]{Proposition}
\newtheorem*{proposition*}{Proposition}

\newtheorem*{conjecture*}{Conjecture}
\theoremstyle{definition}
\newtheorem{problem}{Problem}
\newtheorem*{problem*}{Problem}

\newtheorem{remark}[theorem]{Remark}

\numberwithin{equation}{section}

\def\bR {\mathbb{R}}

\def\bZ {\mathbb{Z}}

\def\cM {\mathcal{M}}

\def\cY {\mathcal{Y}}
\def\cZ {\mathcal{Z}}

\def\grad {{\nabla}}

\def\la {\langle}
\def\ra {\rangle}

\newcommand{\tx}[1]{\mathrm{#1}}

\newcommand{\wt}[1]{\widetilde{#1}}
\newcommand{\bs}[1]{\boldsymbol{#1}}

\newcommand{\sh}[1]{#1^\sharp}

\newcommand{\spn}{\operatorname{span}}

\renewcommand{\ker}{\operatorname{ker}}

\newcommand{\eee}{\mathrm e}

\newcommand{\ud}{\mathrm{\,d}}
\newcommand{\vd}{\mathrm{d}}

\newcommand{\vD}{\mathrm{D}}
\newcommand{\dd}[1]{{\frac{\vd}{\vd{#1}}}}

\title[Dynamics of strongly interacting unstable two-solitons]{Dynamics of strongly interacting \\ unstable two-solitons for generalized \\ Korteweg-de Vries equations}
\author{Jacek Jendrej}
\address{CNRS and Universit\'e Sorbonne Paris Nord, LAGA, UMR 7539, 99 av J.-B.~Cl\'ement, 93430 Villetaneuse, France}
\email{jendrej@math.univ-paris13.fr}

\begin{document}

\begin{abstract}
We consider the generalized Korteweg-de Vries equation $\partial_t u = -\partial_x(\partial_x^2 u + f(u))$,
where $f$ is an odd function of class $C^3$.
Under some assumptions on $f$, this equation admits \emph{solitary waves}, that is solutions of the form $u(t, x) = Q_v(x - vt - x_0)$,
for $v$ in some range $(0, v_*)$. We study pure two-solitons in the case of the same limit speed, in other words global solutions $u(t)$
such that
\begin{equation}
\label{eq:abstract}
\tag{$\ast$}
\lim_{t\to\infty}\|u(t) - (Q_v(\cdot - x_1(t)) \pm Q_v(\cdot - x_2(t)))\|_{H^1} = 0, \qquad \text{with}\quad\lim_{t \to \infty}x_2(t) - x_1(t) = \infty.
\end{equation}
Existence of such solutions is known for $f(u) = |u|^{p-1}u$ with $p \in \bZ \setminus \{5\}$ and $p > 2$.
We describe the~dynamical behavior of any solution satisfying \eqref{eq:abstract} under the assumption that $Q_v$ is linearly unstable
(which corresponds to $p > 5$ for power nonlinearities).
We prove that in this case the sign in \eqref{eq:abstract} is necessarily ``$+$'', which corresponds to an attractive interaction.
We also prove that the~distance $x_2(t) - x_1(t)$ between the solitons equals $\frac{2}{\sqrt v}\log(\kappa t) + o(1)$
for some $\kappa = \kappa(v) > 0$.
\end{abstract}

\maketitle
\section{Introduction}
\label{sec:intro}
\subsection{Setting of the problem}
\label{ssec:setting}
We consider the generalized Korteweg-de Vries equation
\begin{equation}
\label{eq:kdv}
\tag{gKdV}
\Bigg\{ \begin{aligned}
& \partial_t u(t, x) = - \partial_x\big(\partial_x^2 u(t, x) + f(u(t, x))\big), \\
& u(0, x) = u_0(x),\qquad u_0 \in H^1(\bR).
\end{aligned}
\end{equation}
For $f(u) = u^2$ we obtain the classical KdV equation and for $f(u) = u^3$ the~mKdV equation.
Both equations are completely integrable. Thus, for these two models, at least in principle,
the dynamical behavior of solutions can be fully understood, see for instance \cite{EckSch83}.
We are interested in describing some aspects of the dynamical behavior of solutions for other nonlinearities $f$.

In this paper, we assume that $f$ is a non-trivial odd function of class $C^3$
such that $f(0) = f'(0) = 0$ and $f(u)$ is convex for $u \geq 0$.
Local well-posedness in $H^1(\bR)$ of the Cauchy problem was established by Kenig, Ponce and Vega~\cite{KPV89, KPV93}.
Moreover, if the final time of existence is finite, then the solution is unbounded in $H^1$.

For $u_0 \in H^1(\bR)$ we define the following quantities:
\begin{align}
M(u_0) &:= \int_\bR u_0(x)^2 \ud x&\text{(momentum)}, \\
E(u_0) &:= \int_\bR \Big[\frac 12 (\partial_x u_0(x))^2 - F(u_0(x))\Big]\ud x&\text{(energy)},
\end{align}
where $F(u) := \int_0^u f(u')\ud u'$.
We say that $H^1(\bR)$ is the \emph{energy space}, because it is the largest functional space
whose elements have finite energy and finite momentum.
The functionals $M$ and $E$ are conservation laws: if $u(t)$ solves \eqref{eq:kdv},
then $M(u(t)) = M(u_0)$ and $E(u(t)) = E(u_0)$ for all $t$ belonging to the maximal time interval of existence.

It is known, see Proposition~\ref{prop:Q-asym} below, that for $v > 0$ the equation
\begin{equation}
\label{eq:Qv}
-\partial_x^2 w(x) - f(w(x)) + v w(x) = 0, \qquad w \in H^1(\bR)
\end{equation}
has a unique positive even solution $w(x) = Q_v(x)$ if and only if $v < v_* := \lim_{u \to \infty}f(u) / u$.
It is easy to see that for any $v \in (0, v_*) $ and $x_0 \in \bR$ the function $u(t, x) = Q_v(x -vt- x_0)$
is a solution of \eqref{eq:kdv}. These solutions are called \emph{solitons} or \emph{travelling waves}.
We call $v$ the \emph{velocity} of the soliton.

We denote
\begin{equation}
\label{eq:wtQv}
\wt Q_v := \partial_v Q_v, \qquad v \in (0, v_*)
\end{equation}
(we will justify later, in Lemma~\ref{lem:Qv}, that this derivative is well defined).
By classical results, $Q_v$ is orbitally (with respect to translations)
stable if and only if $\int_\bR Q_v \wt Q_v\ud x > 0$, see \cite{CaLi82, Weinstein86, GSS87, BSS87}.
Pego and Weinstein~\cite{PeWe92} proved that if $\int_\bR Q_v \wt Q_v\ud x < 0$,
then $Q_v$ is linearly unstable.
The corresponding unstable manifold was constructed by Combet~\cite{Combet10},
giving another proof of instability in this case.
For power nonlinearities $f(u) = |u|^{p-1}u$,
we have $\int_\bR Q_v \wt Q_v\ud x > 0$ if and only if $p < 5$ ($L^2$-subcritical case)
and $\int_\bR Q_v \wt Q_v\ud x < 0$ if and only if $p > 5$ ($L^2$-supercritical case).

Martel and Merle \cite{MaMe01-ARMA, MaMe08} proved that solitons are asymptotically stable in a suitable sense.

We say that $u(t)$ is a \emph{multi-soliton} as $t \to \infty$ if there exist
$K \in \bZ_{>0}$, $\sigma_k\in \{{-}1, 1\}$, $v_k \in (0, v_*)$
and continuous functions $x_k(t)$ for $k \in \{1, \ldots, K\}$
such that $v_{k+1} \geq v_k$, $\lim_{t\to\infty}x_{k+1}(t) - x_k(t) = \infty$ and
\begin{equation}
u(t) \simeq \sum_{k=1}^K \sigma_k Q_{v_k}(\cdot - x_k(t))\qquad\text{as }t \to \infty,
\end{equation}
where the meaning of ``$\simeq$'' can depend on the context.
We say that $u(t)$ is a \emph{pure multi-soliton} as $t \to \infty$ if
\begin{equation}
\lim_{t \to\infty}\Big\|u(t) - \sum_{k=1}^K \sigma_k Q_{v_k}(\cdot - x_k(t))\Big\|_{H^1} = 0.
\end{equation}
In the case $v_1 < v_2 < \ldots < v_K$, stability and asymptotic stability of multi-solitons
was proved by Martel, Merle and Tsai~\cite{MMT02}. Also for $v_1 < \ldots < v_K$, pure multi-solitons
were completely classified by Martel~\cite{Martel05} and Combet~\cite{Combet11}.

\subsection{Formal prediction of multi-soliton dynamics}
\label{ssec:formal}
Consider a solution which is close to a superposition of a finite number of solitons:
\begin{equation}
\label{eq:formal-1}
u(t) \simeq \sum_{k=1}^K \sigma_k Q_{v_k(t)}(\cdot - x_k(t)),\quad x_{k+1}(t) - x_k(t) \gg 1,\ \sigma_k \in \{{-}1, 1\}.
\end{equation}
One natural way to predict the dynamical behavior of the parameters $x_k(t)$ and $v_k(t)$
is to consider the \emph{motion with constraints}, see \cite[Chapter 1.5]{AKN-Celestial}.

We equip the space of real-valued functions on $\bR$ with the symplectic form $\omega(v, w) := \int_{\bR}v\partial_x^{-1}w\ud x$, where
$\partial_x^{-1}w(x) := \frac 12\big(\int_{-\infty}^x w(y)\ud y - \int_x^\infty w(y)\ud y\big)$.
The Hamiltonian vector field corresponding to the energy functional $E$
is given by $X_E(u) = {-}\partial_x(\partial_x^2 u + f(u))$,
which is the right hand side of \eqref{eq:kdv}.
We now restrict our Hamiltonian system to the $2K$-dimensional manifold
\begin{equation}
\label{eq:manifold}
\cM := \bigg\{ \sum_{k=1}^K \sigma_k Q_{v_k}(\cdot - x_k)\,:\,x_{k+1} - x_k \gg 1,\ v_k \in (0, v_*)\bigg\}.
\end{equation}
Let us stress that in general $\cM$ is not invariant under the flow.
Denote $\bs x := (x_1, \ldots, x_K)$ and $\bs v := (v_1, \ldots, v_K)$.
Then $(\bs x, \bs v)$ is a natural system of coordinates on $\cM$.
The basis of the tangent space $T_{(\bs x, \bs v)}\cM$ induced by these coordinates is given by
$\partial_{x_k} = -\sigma_k\partial_x Q_{v_k}(\cdot - x_k)$ and
$\partial_{v_k} = \sigma_k \wt Q_{v_k}(\cdot - x_k)$.
Let
\begin{equation}
\begin{pmatrix}A(\bs x, \bs v) & C(\bs x, \bs v) \\
-C(\bs x, \bs v) & B(\bs x, \bs v)\end{pmatrix} =
\begin{pmatrix}(a_{jk})_{j, k=1}^{K} & (c_{jk})_{j, k=1}^{K} \\
({-}c_{jk})_{j, k=1}^{K} & (b_{jk})_{j, k=1}^{K}\end{pmatrix}\end{equation}
be the matrix of the symplectic form in this basis,
in other words for $j, k \in \{1, \ldots, K\}$ we have
\begin{align}
a_{jk} &= \omega(\partial_{x_j}, \partial_{x_k}) = \sigma_j\sigma_k\int_\bR \partial_x Q_{v_j}(x - x_j)Q_{v_k}(x - x_k)\ud x,\\
c_{jk} &= \omega(\partial_{x_j}, \partial_{v_k}) = \sigma_j\sigma_k\int_\bR Q_{v_j}(x - x_j)
\wt Q_{v_k}(x - x_k)\ud x,\\
b_{jk} &= \omega(\partial_{v_j}, \partial_{v_k}) = \sigma_j\sigma_k\int_\bR \wt Q_{v_j}(x -x_j)\partial_x^{-1}\wt Q_{v_k}(x - x_k)\ud x.
\end{align}
Note that if $\int_\bR Q_{v_k}\wt Q_{v_k}\ud x = 0$ for some $k \in \{1, \ldots, K\}$,
then $C(\bs x, \bs v)$ and $\begin{pmatrix}A(\bs x, \bs v) & C(\bs x, \bs v) \\
-C(\bs x, \bs v) & B(\bs x, \bs v)\end{pmatrix}$  become singular as the separation between the solitons tends to infinity.
This corresponds to the delicate $L^2$-critical regime
studied for instance in \cite{MMR14-1, MMR15-2, MMR15-3}, which will not be considered in this paper.
We denote $V_\tx{crit} \subset (0, v_*)$ the set of $v \in (0, v_*)$ such that $\int_\bR Q_{v}\wt Q_{v}\ud x = 0$.
This is a closed set.

The Hamiltonian is the restriction $E\vert_{\cM}$. Slightly abusing the notation, we write
\begin{equation}
\label{eq:E-on-M}
E(\bs x, \bs v) := E\bigg(\sum_{k=1}^K \sigma_k Q_{v_k}(\cdot - x_k)\bigg).
\end{equation}
The function $E(\bs x, \bs v)$ is sometimes called the \emph{reduced Hamiltonian}.
The motion with constraints is given by the equation
\begin{equation}
\label{eq:constrained}
\begin{pmatrix}\bs x' \\ \bs v'\end{pmatrix} = \begin{pmatrix}X(\bs x, \bs v) \\ V(\bs x, \bs v)\end{pmatrix} := \begin{pmatrix}A(\bs x, \bs v) & C(\bs x, \bs v) \\
-C(\bs x, \bs v) & B(\bs x, \bs v)\end{pmatrix}^{-1}\begin{pmatrix}\partial_{\bs x}E(\bs x, \bs v) \\ \partial_{\bs v}E(\bs x, \bs v)\end{pmatrix}.
\end{equation}
Two problems seem natural.
\begin{problem}
\label{prob:constrained}
Study the solutions of the reduced equation \eqref{eq:constrained}.
\end{problem}
\begin{problem}
\label{prob:full}
Is the dynamical behavior of (pure) multi-soliton solutions to \eqref{eq:kdv}
correctly described by equation~\eqref{eq:constrained}?
\end{problem}
It turns out that if $0 < v_1^\infty < \ldots < v_K^\infty < v_*$ and $v_k^\infty \notin V_\tx{crit}$
for $k \in \{1, \ldots, K\}$,
then one can easily classify all the solutions to \eqref{eq:constrained} such that $v_k(t) \to v_k^\infty$
and $x_{k+1}(t) - x_k(t) \to \infty$ as $t \to \infty$, see Proposition~\ref{prop:constrained-distinct}.
Also for distinct limit velocities, Problem~\ref{prob:full} was solved in the works
\cite{MMT02, Martel05, Combet11} mentioned above.

Without the assumption that the limit velocities are distinct, the situation is only partially understood.
Existence of 2-solitons and 3-solitons with asymptotically equal velocities
was first observed for the mKdV equation by Wadati and Ohkuma~\cite{WaOhk82}.
For any power nonlinearity except for the critical case, such 2-solitons were constructed by Nguyen~\cite{Vinh17}.
We call this case the \emph{strong interaction regime} because, as we will see, interactions
between the solitons play an essential role in determining the asymptotic behavior of the solution.

We are not aware of any systematic treatment of equation~\eqref{eq:constrained},
we note though that the equation itself appears for instance in \cite{GoOs81}.
Providing a full answer to Problem~\ref{prob:constrained} might be of independent interest,
in view of the fact that an analogous formal reduction argument can be carried out
in the study of multi-solitons for various other models, most likely leading to a reduced system
similar to~\eqref{eq:constrained}.

\subsection{Statements of the results}
\label{ssec:results}
In this paper, we only consider the case $K = 2$. We hope to treat the general case in the future.
Concerning Problem~\ref{prob:constrained}, we have the following result.
\begin{proposition}
\label{prop:constrained}
Let $\sigma_1, \sigma_2 \in \{{-}1, 1\}$ and let $v^\infty \in (0, v_*) \setminus V_\tx{crit}$.
If
\begin{equation}
\label{eq:sigma1sigma2}
\sigma_1 \sigma_2 = \begin{cases}
-1\quad&\text{in the case }\int_\bR Q_{v^\infty}\wt Q_{v^\infty}\ud x > 0, \\
1\quad&\text{in the case }\int_\bR Q_{v^\infty}\wt Q_{v^\infty}\ud x < 0,
\end{cases}
\end{equation}
then the equation \eqref{eq:constrained} has a solution $(\bs x(t), \bs v(t)) = (x_1(t), x_2(t), v_1(t), v_2(t))$ such that
\begin{gather}
\lim_{t\to\infty} t^\beta\big(\big|x_1(t) - v^\infty t + (v^\infty)^{-\frac 12}\log(\kappa t)\big| + \big|x_2(t) - v^\infty t - (v^\infty)^{-\frac 12}\log(\kappa t)\big|\big) = 0, \\
\lim_{t \to \infty} t^{\beta+1}\big(\big|v_1(t) - v^\infty + (v^\infty)^{-\frac 12}t^{-1}\big| + \big|v_2(t) - v^\infty
- (v^\infty)^{-\frac 12}t^{-1}\big|\big) = 0,
\end{gather}
for some $\beta > 0$, where $\kappa = \kappa(v^\infty) > 0$ is an explicit constant.
In particular,
\begin{equation}
\label{eq:conv-vitesse}
\begin{gathered}
\lim_{t \to \infty}v_1(t) = \lim_{t\to\infty}v_2(t) = v^\infty, \\
\lim_{t \to \infty} x_2(t) - x_1(t) = +\infty.
\end{gathered}
\end{equation}

If $(\sh{\bs x}, \sh{\bs v})$ is any solution of \eqref{eq:constrained} satisfying \eqref{eq:conv-vitesse},
then there exist unique $t^\infty, x^\infty \in \bR$ such that $(\sh{\bs x}(t), \sh{\bs v}(t)) = (\bs x(t - t^\infty) + x^\infty, \bs v(t - t^\infty))$.

Finally, if \eqref{eq:sigma1sigma2} is not satisfied, then there are no solutions of \eqref{eq:constrained} satisfying \eqref{eq:conv-vitesse}.
\end{proposition}
We give a proof (skipping the most routine computations) in Appendix~\ref{sec:reduced}.

Our main result is a partial positive answer to Problem~\ref{prob:full}
in the case of the same limit velocity $v_1^\infty = v_2^\infty =: v^\infty$.
\begin{theorem}
\label{thm:classif}
Let $v^\infty \in (0, v_*)$ be such that
\begin{equation}
\label{eq:supercrit}
\int_\bR Q_{v^\infty}(x)\wt Q_{v^\infty}(x)\ud x < 0
\end{equation}
and let
$u: [T_0, \infty) \to H^1(\bR)$ be a solution of \eqref{eq:kdv} satisfying
\begin{equation}
\label{eq:u-conv}
\lim_{t \to \infty}\big\|u(t) - Q_{v^\infty}\big(\cdot -\, x_1(t)\big) - \sigma Q_{v^\infty}\big(\cdot -\, x_2(t)\big)\big\|_{H^1} = 0,
\end{equation}
where $\sigma \in \{{-1}, 1\}$ and
$x_1, x_2: [T_0, \infty) \to \bR$ are continuous functions such that
\begin{equation}
\label{eq:q-conv}
\lim_{t\to\infty} x_2(t) - x_1(t) = +\infty.
\end{equation}
 Then $\sigma = 1$ and
\begin{equation}
\label{eq:classif}
\lim_{t\to \infty}x_2(t) - x_1(t) -\frac{2}{\sqrt{v^\infty}}\log(\kappa t) = 0,
\end{equation}
where $\kappa = \kappa(v^\infty) > 0$.
\end{theorem}
\begin{remark}
We were unable to treat the stable case $\int_\bR Q_{v^\infty}(x)\wt Q_{v^\infty}(x)\ud x > 0$.
The distance between the solitons for the solutions constructed in~\cite{Vinh17}
is given by~\eqref{eq:classif} both in the unstable and stable case.
However, in the stable case it remains an open problem to prove that this is the only possible separation.
\end{remark}
\begin{remark}
One natural refinement of Theorem~\ref{thm:classif} would be to obtain a complete classification
of all the solutions satisfying \eqref{eq:u-conv}.
The set of solutions obtained in Proposition~\ref{prop:constrained} is a two-dimensional manifold,
hence, taking into account that the linearisation around each of the two solitons has one stable direction,
we conjecture that all the solutions satisfying \eqref{eq:u-conv} form a four-dimensional manifold.
For kink-antikink pairs, a uniqueness result of this kind was obtained in \cite{JKL1}.
\end{remark}
\begin{remark}
The assumption $f \in C^3$ is mainly to ensure local well-posedness.
We expect that $f \in C^{1, \gamma}$ for some $\gamma > 0$ would suffice to justify our computations.
\end{remark}
\begin{remark}
The problem considered here is quite similar to the work of Gustafson and Sigal~\cite{GuSi06}
on multi-vortices in the Higgs model.
Our proof, based mainly on exploiting the Hamiltonian structure of the equation
combined with the modulation method (see below), also bears some resemblance to the approach
adopted in~\cite{GuSi06}. One important difference is that while
we consider \emph{pure} multi-solitons and control them for all positive times,
in~\cite{GuSi06} non-pure multi-vortices are controlled on a large but finite time interval.
\end{remark}
\begin{remark}
Constructions of strongly interacting pure two-solitons or two-bubbles for models which are not completely integrable can be found in
\cite{MaRa15p, moi-two-bub, moi17, Vinh17, Vinh19}, see also \cite{GLPR18} for a construction of a slightly different type.

\end{remark}
\subsection{Main elements of the proof}
\label{ssec:outline}
The key ingredient of the proof is the so-called \emph{modulation method}.
We study solutions which are close (in the energy space) to a superposition of two translated
copies of the soliton $Q$. Hence, it is natural to decompose
\begin{equation}
\label{eq:idea-decomp}
u(t) = \sigma_1 Q(\cdot - x_1(t)) + \sigma_2 Q(\cdot - x_2(t)) + \eta(t),
\end{equation}
where $x_1(t)$ and $x_2(t)$ are the centers of the two solitons (we address below the question
of how exactly $x_1(t)$ and $x_2(t)$ are chosen)
and $ \eta(t)$ is the \emph{error term}.
The only \emph{a priori} information is that $x_2(t) - x_1(t) \to \infty$
and $\| \eta(t)\|_{H^1} \to 0$ as $t \to \infty$.
The idea of the modulation method is to derive some differential inequalities on
the \emph{modulation parameters} $x_1(t)$ and $x_2(t)$.
These inequalities are traditionally called \emph{modulation equations}.
Since it is hard to obtain any precise information about $\eta(t)$,
preferably $ \eta(t)$ should not appear in the formulas.

Guided by the intuition explained in Section~\ref{ssec:formal}, we expect that we should define
two auxiliary parameters playing the role of the momenta, in order to obtain a system close to \eqref{eq:constrained}.
These parameters $p_k(t)$, defined by formula \eqref{eq:pk} below, are related to the projections of the error term
on null directions of the adjoint of the linearization of the flow around our two-soliton.
This choice makes linear terms disappear when we compute the time derivative of these momenta.
It turns out that one can define parameters, which we call $q_k(t)$,
related to the positions of the two solitons, whose time derivatives are essentially the momenta $p_k(t)$, see \eqref{eq:dtqk}.

The only way of estimating the error term we could think of is to use coercivity properties of the conservation laws.
It seems to us that this can only be achieved in the case $\big\la Q_{v^\infty}, \wt Q_{v^\infty}\big\ra < 0$,
which is precisely the obstacle preventing us from treating the case $\big\la Q_{v^\infty}, \wt Q_{v^\infty}\big\ra > 0$.
In the favorable case, we obtain essentially that $\|\eta\|_{H^1}^2$ is bounded by the size of the interaction between the solitons.
Thus, in order to have useful bounds on derivatives of the momenta, we have to absorb somehow the main quadratic terms,
which is why $p_k(t)$ contains a \emph{correction term}, quadratic in $\eta$.
A similar idea was used in \cite{moi-andy-two-bub} and (in a different context of minimal-mass blow-up)
in the earlier work of Rapha\"el and Szeftel~\cite{RaSz11}.

Note that an alternative way, perhaps more natural in view of Section~\ref{ssec:formal}, would be to decompose
\begin{equation}
\label{eq:idea-decomp-alt}
u(t) = \sigma_1 Q_{v_1(t)}(\cdot - x_1(t)) + \sigma_2 Q_{v_2(t)}(\cdot - x_2(t)) + \eta(t),
\end{equation}
with $\eta(t)$ satisfying four orthogonality conditions.
We have not tried to carry out the computation following this approach.

The paper is organized as follows. In Section~\ref{sec:spectral}, we study the stationary equation~\eqref{eq:Qv}.
In Section~\ref{sec:variational}, we study variational properties of the conserved quantities in a neighborhood of a two-soliton.
In Section~\ref{sec:modulation-two}, we define the modulation parameters and derive bounds on their derivatives.
In Section~\ref{sec:dynamics}, we finish the proof of Theorem~\ref{thm:classif}.
Appendix~\ref{sec:reduced}, independent of the main text, is devoted to the ODE \eqref{eq:constrained}.
\subsection{Acknowledgments}
\label{ssec:merci}
I would like to thank Yvan Martel for helpful discussions
and Carlos Kenig for encouragement.

I am indebted to the anonymous referees for a careful reading of the manuscript,
noticing~an error in the proof of Proposition~\ref{prop:constrained},
and many valuable corrections and suggestions.

I was supported by the ANR-18-CE40-0028 project ESSED.
Part of this work was completed when I was employed by the University of Chicago.

\subsection{Notation}
\label{ssec:notation}
We denote $L^2 := L^2(\bR)$, $H^1 := H^1(\bR)$, $L^\infty := L^\infty(\bR)$, etc.
All the functions are real-valued.
The $L^2$ scalar product is denoted $\la w_1, w_2\ra := \int_\bR w_1(x)w_2(x)\ud x$.
We use the same notation for the distributional pairing.
In the integrals, we often omit the variable and write $\int_\bR w\ud x$ instead of $\int_\bR w(x)\ud x$, etc.

For a nonlinear functional $\Phi: H^1 \to \bR$ we denote $\vD\Phi : H^1 \to H^{-1}$
its Fr\'echet derivative. If $\vD \Phi(w) = 0$, we denote $\vD^2 \Phi(w) \in H^{-1} \otimes H^{-1}$
the second derivative (Hessian).

Even if $w(x)$ is a function of one variable $x$, we often write $\partial_x w(x)$
instead of $w'(x)$ to denote the derivative. The prime notation is only used
for the time derivative of a function of one variable $t$
and for the derivative of the nonlinearity $f$.
For $w \in L^1(\bR)$ we denote $\partial_x^{-1}w(x) := \frac 12\big(\int_{-\infty}^x w\ud y - \int_x^\infty w\ud y\big)$.
It follows that if $\int_\bR w\ud x = 0$, then $\partial_x^{-1}w(x) = \int_{-\infty}^x w\ud y$.

For two functions $a$ and $b$, we write $a \lesssim b$ if $ a \leq Cb$ for some constant $C > 0$,
$a \gtrsim b$ if $a \geq cb$ for some constant $c > 0$,
and $a \sim b$ if $a / b$ converges to $1$.
We use the symbol $a \simeq b$ to denote equality ``up to negligible terms''.
It will be specified in each case which terms are considered as negligible.
\section{Analysis of the stationary equation}
\label{sec:spectral}
In this preliminary section, we study equation \eqref{eq:Qv}.
Many arguments are well-known and included mainly for the convenience of the reader.
\subsection{Existence and asymptotic behavior of $Q_v$ and $\wt Q_v$}
\label{ssec:asympt-Qv}
\begin{proposition}
\label{prop:Q-asym}
Let $v_* := \lim_{u \to \infty}f(u) / u$. For $v > 0$ the following conditions are equivalent.
\begin{enumerate}[(a)]
\item \label{it:vstar} $v \in (0, v_*)$.
\item \label{it:exists}
Equation~\eqref{eq:Qv} has a nontrivial solution $w \in H^1$.
\item \label{it:Qv-unique} Equation~\eqref{eq:Qv} has a unique positive even solution $w = Q_v$
and $Q_v(0)$ is the unique positive zero of $s \mapsto \frac v2 s^2 - F(s)$.
All the other solutions are obtained from $Q_v$ by translations and sign change.
\end{enumerate}
\end{proposition}
\begin{remark}
The limit defining $v_*$ exists because our assumptions on $f$ imply that the function
$u \mapsto f(u) / u$ is increasing.
\end{remark}
\begin{proof}
This is an easy consequence of \cite[Section 6]{BeLi83}, where it was proved that \ref{it:exists} and \ref{it:Qv-unique}
are both equivalent to the following condition: there exists $s_0$ the smallest positive zero of $s \mapsto \frac v2 s^2 - F(s)$ and $s_0$ satisfies $vs_0 - f(s_0) < 0$.

Our assumptions on $f$ imply $f(u) < \frac us f(s)$ for all $0 < u < s$.
Integrating for $u \in (0, s)$ we obtain
\begin{equation}
\label{eq:f-F-cond}
sf(s) - 2F(s) > 0,\qquad\text{for all }s > 0.
\end{equation}
Consider the function
\begin{equation}
\label{eq:Ftil}
\wt F(s) := \frac{2F(s)}{s^2}, \quad \wt F(0) := 0.
\end{equation}
From \eqref{eq:f-F-cond} we get $\wt F'(s) > 0$,
so $\wt F: [0, \infty) \to [0, \infty)$ is an increasing continuous function.
Clearly $\lim_{s \to \infty}\wt F(s) = v_*$.
This shows that $s \mapsto \frac v2 s^2 - F(s)$ has a positive zero $s_0$ if and only if $v \in (0, v_*)$
and that $s_0$ is unique. The condition $vs_0 - f(s_0) < 0$ is automatic, as is seen from \eqref{eq:f-F-cond}.
\end{proof}
For $v \in (0, v_*)$, we denote $L_v := -\partial_x^2 - f'(Q_v) + v$
the linearization of the left hand side of \eqref{eq:Qv} around $w = Q_v$.
\begin{lemma}
\label{lem:Q-asym}
For all $v \in (0, v_*)$, $Q_v \in C^5$ and $L_v \partial_x Q_v = 0$.
Moreover, there exists $k_0 = k_0(v) > 0$ such that for $j \in \{0, 1, 2, \ldots\}$ the function $Q_v$ satisfies
$Q_v^{(j)}(x) \sim (-\sqrt v)^j k_0 \eee^{-\sqrt vx}$
as $x \to +\infty$ and $Q_v^{(j)}(x) \sim k_0\sqrt v^j \eee^{\sqrt vx}$ as $x \to -\infty$.
\end{lemma}
\begin{remark}
\label{rem:kappa}
The constant $\kappa$ in Theorem~\ref{thm:classif} turns out to be $\kappa := k_0(v^\infty)\sqrt{\frac{2(v^\infty)^{3/2}}{-\la Q_{v^\infty}, \wt Q_{v^\infty}\ra}}$.
\end{remark}
\begin{proof}[Proof of Lemma~\ref{lem:Q-asym}]
Let $v \in (0, v_*)$ and let $s_0 > 0$ be such that $\frac v2 s_0^2 - F(s_0) = 0$.
Regularity of $Q_v$ follows from the equation. Differentiating
${-}\partial_x^2 Q_v(\cdot - x_0) - f(Q_v(\cdot - x_0)) + vQ_v(\cdot - x_0) = 0$
with respect to $x_0$ we get $L_v \partial_x Q_v = 0$.

In order to determine the asymptotic behavior of $Q_v$, we recall how to solve \eqref{eq:Qv}.
We observe that \eqref{eq:Qv} implies $\frac 12 (\partial_x Q_v)^2 + F(Q_v) - \frac v2 Q_v^2 = 0$, hence
\begin{equation}
\label{eq:bogomolny}
\partial_x Q_v(x) = \pm\sqrt{vQ_v^2 - 2F(Q_v)},\qquad \text{for all }x \in \bR.
\end{equation}
Since $Q_v'(0) = 0$ and $Q_v''(0) = vs_0 - f(s_0) < 0$, for small positive $x$ the sign in \eqref{eq:bogomolny} is ``$-$''.
We also have $vs^2 - 2F(s) > 0$ for all $0 < s < s_0$, hence by a straightforward continuity argument,
$Q_v$ is decreasing for $x > 0$ and $\partial_x Q_v(x) = -\sqrt{vQ_v^2 - 2F(Q_v)}$ for all $x \geq 0$.
After separation of variables, for all $0 < x_1 < x < \infty$ we obtain
\begin{equation}
\int_{Q_v(x)}^{Q_v(x_1)}\bigg(\Big(s^2-\frac{2F(s)}{v}\Big)^{-\frac 12}-\frac 1s\bigg)\ud s = \sqrt v(x - x_1) + \log\frac{Q_v(x)}{Q_v(x_1)}.
\end{equation}
Taking the limit $x_1 \to 0$, we have
\begin{equation}
\eee^{\sqrt v x}Q_v(x) = s_0\exp\bigg(\int_{Q_v(x)}^{s_0}\bigg(\Big(s^2-\frac{2F(s)}{v}\Big)^{-\frac 12}-\frac 1s\bigg)\ud s\bigg).
\end{equation}
Since $Q_v \in H^1(\bR)$ implies $\lim_{x \to \infty}Q_v(x) = 0$, taking the limit $x \to \infty$ yields
\begin{equation}
\label{eq:k0-formula}
\lim_{x\to\infty} \eee^{\sqrt v x}Q_v(x) = s_0\exp\bigg(\int_0^{s_0}\bigg(\Big(s^2-\frac{2F(s)}{v}\Big)^{-\frac 12}-\frac 1s\bigg)\ud s\bigg) \in (0, \infty).
\end{equation}
Integrability near $s = s_0$ follows from $vs_0 - f(s_0) < 0$
and integrability near $s = 0$ from $F(s) \lesssim s^3$.

Once the asymptotic behavior of $Q_v$ is known, the estimates for the derivatives follow from
the differential equation.
\end{proof}
\begin{remark}
\label{rem:unif-bound}
We see from \eqref{eq:bogomolny} that
$\partial_x Q_v(x) = -\sqrt{vQ_v^2 - 2F(Q_v)} > -\sqrt{v}Q_v(x)$ for all $x > 0$,
so in fact we have
\begin{equation}
Q_v(x) \leq k_0\eee^{-\sqrt v x},\qquad\text{for all }x > 0.
\end{equation}
Formula \eqref{eq:k0-formula} implies (by standard arguments) that $k_0(v)$ is continuous with respect to $v$.
Thus we can conclude that for any $0 < v_1 < v_2 < v_*$ there exists $C_0 > 0$ such that
\begin{equation}
Q_v(x) \leq C_0\eee^{-\sqrt{v_1} x},\qquad\text{for all }x > 0, v \in [v_1, v_2],
\end{equation}
and similarly for $\partial_x^j Q_v$.

Again from \eqref{eq:bogomolny}, we have $|\partial_x Q_v + \sqrt v Q_v| \lesssim \eee^{-2\sqrt v x}$ for $x > 0$,
which implies $|Q_v(x) - k_0\eee^{-\sqrt v x}| \lesssim \eee^{-2\sqrt v x}$, and similarly for derivatives.
This estimate can be also made uniform in $v$, as above.
\end{remark}
\begin{lemma}
\label{lem:Qv}
For all $v\in (0, v_*)$, the function $\wt Q_v := \partial_v Q_v$ is well-defined as a classical partial derivative.
Moreover, $\wt Q_v \in C^4$, $|\wt Q_v(x)| \lesssim |x|\eee^{-\sqrt v |x|}$ as $|x| \to \infty$
and $L_v\wt Q_v = -Q_v$.
\end{lemma}
\begin{proof}
The function $\wt F(s)$ defined in \eqref{eq:Ftil} is $C^4$ for $s > 0$.
Thus $Q_v(0) = s_0 = \wt F^{-1}(v)$ is $C^4$ on $(0, v_*)$.
By smooth dependence on initial conditions of solutions of ordinary equations,
$Q_v(x)$ is of class $C^3$ as a function of $2$ variables $(x, v)$.
Differentiating $-\partial_x^2 Q_v - f(Q_v) + vQ_v = 0$ with respect to $v$ we obtain $L_v\partial_v Q_v = -Q_v$.

Denote $\phi_v(x) := \partial_x Q_v(x)$ and $\psi_v(x)$ the solution of
$L_v\psi_v = 0$ with initial conditions $\psi_v(0) = \frac{1}{f(s_0)-vs_0}$
and $\frac{\vd\psi_v}{\vd x}(0) = 0$. Then $(\phi_v, \psi_v)$ is a fundamental system of solutions for the operator
$L_v$, with the Wronskian equal to $1$. We set
\begin{equation}
\label{eq:Qtil-def}
\wt Q_v(x) := {-}\phi_v(x)\int_0^x \psi_v(y)Q_v(y)\ud y + \frac 12 \psi_v(x)Q_v(x)^2.
\end{equation}
For the moment, it is not clear that $\wt Q_v = \partial_v Q_v$, but we will prove that this is indeed the case.

By standard ODE theory, $\psi_v(x)$ is continuous in both variables, of class $C^4$ in $x$,
and $|\psi_v(x)| \lesssim \eee^{\sqrt v|x|}$.
Thus \eqref{eq:Qtil-def} yields $\wt Q_v \in C^4$ and, using Lemma~\ref{lem:Q-asym},
$|\wt Q_v(x)| \lesssim |x|\eee^{-\sqrt v|x|}$. The fact that $L_v \wt Q_v = -Q_v$
is a routine computation. Hence $\partial_v Q_v$ and $\wt Q_v$ satisfy the same differential equation.
Moreover, $\wt Q_v(0) = \frac{s_0^2}{2(f(s_0) - vs_0)} = \partial_v Q_v(0)$,
where the last equality follows by differentiating $\frac v2 (Q_v(0))^2 - F(Q_v(0)) = 0$ with respect to $v$.
Since both $\wt Q_v$ and $\partial_v Q_v$ are even functions, we obtain $\wt Q_v = \partial_v Q_v$.
\end{proof}
\begin{remark}
We note that one can obtain ``semi-explicit'' formulas for $\|Q_v\|_{L^2}^2$ and $\la Q_v, \wt Q_v\ra$.
Using the fact that $\partial_x Q_v(x) = -\sqrt{v Q_v(x)^2 - 2F(Q_v(x))}$ for $x > 0$ and changing the variable
we find
\begin{equation}
\|Q_v\|_{L^2}^2 = 2\int_0^{s_0}\frac{s^2\ud s}{\sqrt{v s^2 - 2F(s)}}.
\end{equation}
One can find a similar (but more complicated) formula for $\la Q_v, \wt Q_v\ra$
by carefully differentiating the formula above (taking into account the singular behavior near $s = s_0$).
Alternatively, one can use \eqref{eq:Qtil-def} and then change the variable to $s = Q_v(x)$.
\end{remark}

\subsection{Spectral properties of $L_v$}
All the results contained in this section are proved in \cite{CMM11}
in the case $f(u) = u^{2m+1}$. Since the specific form of the nonlinearity is used
in proofs given there, for reader's convenience we provide alternative proofs,
but of course some steps are the same as in \cite{CMM11}.

Without loss of generality, we take $v = 1$ (the general case follows by rescaling).
We denote $Q := Q_1$, $\wt Q := \wt Q_1$ and $L := L_1$.
From Lemmas~\ref{lem:Q-asym} and \ref{lem:Qv} we have
\begin{equation}
\label{eq:L-ker}
L (\partial_x Q) = 0, \qquad L(\wt Q) = -Q.
\end{equation}
We assume $\la Q, \wt Q\ra < 0$.

\begin{proposition}
\label{prop:L-index}
The operator $L$ is self-adjoint with domain $H^2(\bR)$, has one simple negative
eigenvalue and $\ker L = \spn(\partial_x Q)$.
\end{proposition}
\begin{proof}
This is a standard consequence of the Sturm-Liouville theory and the fact
that $\partial_x Q$ has exactly one zero.
\end{proof}
\begin{proposition}
\label{prop:eigen}
There exist exponentially decaying $C^4$ functions $\cY^-, \cY^+$ and $\nu > 0$ such that
\begin{gather}
\partial_x(L \cY^-) = -\nu\cY^-, \qquad \partial_x(L \cY^+) = \nu\cY^+, \label{eq:dxLY} \\
\cY^+(x) = \cY^-(-x), \label{eq:Yrefl} \\
\|\cY^-\|_{L^2} = \|\cY^+\|_{L^2} = 1, \label{eq:Y-norm} \\
\int_\bR \cY^- = \int_\bR \cY^+ = 0, \label{eq:intY} \\
\la \cY^-, L\cY^-\ra = \la \cY^+, L\cY^+\ra = 0,  \label{eq:YLY1} \\
\la \cY^-, L\cY^+\ra = \la \cY^+, L\cY^-\ra \neq 0, \label{eq:YLY2} \\
\cY^-, \cY^+, \partial_x Q \text{ are linearly independent.} \label{eq:YQ'indep}
\end{gather}
\end{proposition}
\begin{proof}
Existence of $\cY^-$ satisfying \eqref{eq:dxLY} is proved in \cite{PeWe92}.
It is easy to check that if $\partial_x(L\cY^-) = -\nu\cY^-$,
then $\cY^+(x) := \cY^-(-x)$ satisfies $\partial_x(L\cY^+) = \nu\cY^+$.
We obtain \eqref{eq:Y-norm} by normalizing. Integrating \eqref{eq:dxLY} we get \eqref{eq:intY}.
Using again \eqref{eq:dxLY} we get
\begin{equation}
\la \cY^-, L\cY^-\ra = -\frac{1}{\nu}\la \partial_x L\cY^-, L\cY^-\ra = 0,
\end{equation}
and similarly $\la \cY^+, L\cY^+\ra = 0$.

In order to prove \eqref{eq:YQ'indep}, we first check that $\cY^-$ and $\cY^+$
are linearly independent. Indeed, suppose that $a^-\cY^- + a^+\cY^+ = 0$.
Applying the operator $\partial_x L$ to both sides and using \eqref{eq:dxLY} we get
$a^-\cY^- - a^+\cY^+ = 0$, which implies $a^- = a^+ = 0$.

Now suppose that $\partial_x Q = a^-\cY^- + a^+\cY^+$. Again, applying $\partial_x L$
to both sides and using the fact that $L\partial_x Q = 0$, we obtain $a^-\cY^- - a^+\cY^+ = 0$.
Since $\cY^-$ and $\cY^+$ are linearly independent, it follows that $a^- = a^+ = 0$.

It remains to prove \eqref{eq:YLY2}.
Suppose that $\la \cY^-, L\cY^+\ra = 0$.
Let $v = a^-\cY^- + a^+\cY^+ + b\partial_xQ \in \spn(\cY^-, \cY^+, \partial_xQ)$.
We have
\begin{equation}
\begin{aligned}
\la v, Lv\ra &= \la a^-\cY^- + a^+\cY^+ + b\partial_xQ, a^-L\cY^- + a^+L\cY^+\ra \\
&= (a^-)^2\la \cY^-, L\cY^-\ra + (a^+)^2\la \cY^+, L\cY^+\ra\\
&+ 2a^-a^+\la \cY^-, L\cY^+\ra
+b\la L\partial_xQ, a^-\cY^- + a^+\cY^+\ra = 0.
\end{aligned}
\end{equation}
Since $\cY^-$, $\cY^+$ and $\partial_xQ$ are linearly independent, this is in contradiction
with Proposition~\ref{prop:L-index} (by the~min-max theorem for self-adjoint operators).
\end{proof}
\begin{proposition}
\label{prop:alpha}
There exist exponentially decaying $C^5$ functions $\alpha^-, \alpha^+$ such that
\begin{gather}
L(\partial_x \alpha^-) = \nu\alpha^-, \qquad
L(\partial_x \alpha^+) = -\nu\alpha^+, \label{eq:Ldxal} \\
\la \alpha^-, \cY^-\ra = \la \alpha^+, \cY^+\ra = 1, \label{eq:alY1} \\
\la \alpha^-, \cY^+\ra = \la \alpha^+, \cY^-\ra = 0, \label{eq:alY2} \\
\la \alpha^-, \partial_xQ\ra = \la \alpha^+, \partial_xQ\ra = 0. \label{eq:alQ'}
\end{gather}
\end{proposition}
\begin{proof}
Set $\wt \alpha^- := \int_{-\infty}^x \cY^+$ and $\wt \alpha^+ := \int_{-\infty}^x\cY^-$.
Using \eqref{eq:dxLY} we have
\begin{equation}
\label{eq:Ldxal-1}
L(\partial_x \wt\alpha^-) = L\cY^+ = \nu \int_{-\infty}^x \cY^+ = \nu\wt\alpha^-,
\end{equation}
and analogously $L(\partial_x \wt\alpha^+) = -\nu\wt\alpha^+$.

Next, we compute
\begin{equation}
\la \wt\alpha^-, \cY^+\ra = \frac{1}{\nu}\la \wt\alpha^-, \partial_x L\cY^+\ra =
-\frac{1}{\nu}\la \cY^+, L\cY^+\ra = 0,
\end{equation}
where in the last step we use \eqref{eq:YLY1}.
Similarly, using \eqref{eq:YLY1} and \eqref{eq:YLY2} we obtain
$\la \wt\alpha^+, \cY^-\ra = 0$, $\la \wt\alpha^-, \cY^-\ra \neq 0$
and $\la \wt\alpha^+, \cY^+\ra \neq 0$.
We set $\alpha^- := \la \wt\alpha^-, \cY^-\ra^{-1}\wt\alpha^-$
and $\alpha^+ := \la \wt\alpha^+, \cY^+\ra^{-1}\wt\alpha^+$.

Finally, \eqref{eq:alQ'} follows from \eqref{eq:Ldxal} and $L\partial_xQ = 0$.
\end{proof}
\begin{lemma}
\label{lem:coer}
If $\la \alpha^-, v\ra = \la \alpha^+, v\ra = 0$ and $\la v, Lv\ra \leq 0$, then $v \in \spn(\partial_xQ)$.
\end{lemma}
\begin{proof}
Let $v$ be such that
\begin{equation}
\label{eq:v-ortho}
\la \alpha^-, v\ra = \la \alpha^+, v\ra = 0,\quad v\notin\spn(\partial_xQ)\quad\text{and}\quad \la v, Lv\ra \leq 0.
\end{equation}
Consider the space $\Sigma := \spn(\cY^+, \partial_xQ, v)$.
First, we prove that $\dim(\Sigma) = 3$.
Indeed, if $v = a\cY^+ + b\partial_xQ$, then
\begin{equation}
\label{eq:v-lin-dep}
0 = \la \alpha^+, v\ra = a\la \alpha^+, \cY^+\ra + b\la \alpha^+, \partial_xQ\ra = a,
\end{equation}
which contradicts the assumption $v \notin \spn(\partial_xQ)$.

Let $w = a\cY^+ + b\partial_xQ + cv \in \Sigma$. We have
\begin{equation}
\la w, Lw\ra = \la a\cY^+ + b\partial_xQ + cv, aL\cY^+ + cLv\ra = 2ac\la L\cY^+, v\ra + c^2\la v, Lv\ra.
\end{equation}
We see from \eqref{eq:Ldxal-1} that $L\cY^+ \in \spn(\alpha^-)$, so that $\la L\cY^+, v\ra = 0$.
Thus $\la w, Lw\ra = c^2 \la v, Lv\ra \leq 0$.
Since $\dim(\Sigma) = 3$, this contradicts Proposition~\ref{prop:L-index} and finishes the proof.
\end{proof}
\begin{remark}
Note that from the second part of the proof above we can obtain the following fact:
if $\la \alpha^-, v\ra = 0$ and $\la v, Lv\ra \leq 0$, then $v \in \spn(\cY^+, \partial_xQ)$.
Similarly, if $\la \alpha^+, v\ra = 0$ and $\la v, Lv\ra \leq 0$, then $v \in \spn(\cY^-, \partial_xQ)$.
In particular, either of the conditions $\la \alpha^-, v\ra = 0$, $\la \alpha^+, v\ra = 0$
implies $\la v, Lv\ra \geq 0$.
\end{remark}
\begin{proposition}
\label{prop:vLv-coer}
There exists $\lambda_0 > 0$ such that for all $v \in H^1$
\begin{equation}
\label{eq:vLv-coer}
\la v, Lv\ra \geq \lambda_0 \|v\|_{H^1}^2 - \frac{1}{\lambda_0}\big(\la \alpha^-, v\ra^2
+ \la \alpha^+, v\ra^2 + \la \partial_x Q, v\ra^2\big).
\end{equation}
\end{proposition}
\begin{proof}
By the definition of $L$ we have
\begin{equation}
\label{eq:vLv-1}
\la v, Lv\ra = \|v\|_{H^1}^2 - \int_\bR f'(Q)v^2\ud x,
\end{equation}
so we can rewrite \eqref{eq:vLv-coer} as
\begin{equation}
\label{eq:vLv-2}
\la v, Lv\ra \geq \frac{\lambda_0}{1-\lambda_0}\int_\bR f'(Q)v^2\ud x - \frac{1}{\lambda_0(1-\lambda_0)}\big(\la \alpha^-, v\ra^2
+ \la \alpha^+, v\ra^2 + \la \partial_x Q, v\ra^2\big).
\end{equation}
If \eqref{eq:vLv-2} does not hold for any $\lambda_0 > 0$,
then there exists a sequence $(v_n) \in H^1$ such that
\begin{gather}
\int_\bR f'(Q)v_n^2\ud x = 1, \label{eq:vn-norm} \\
\la v_n, Lv_n\ra \leq \frac 1n - n \big(\la \alpha^-, v_n\ra^2
+ \la \alpha^+, v_n\ra^2 + \la \partial_x Q, v_n\ra^2\big). \label{eq:vnLvn}
\end{gather}
We see from \eqref{eq:vLv-1} that $(v_n)$ is bounded in $H^1$,
hence it has a subsequence weakly converging to $v \in H^1$.
By standard arguments, we obtain $\int_\bR f'(Q)v^2\ud x = 1$, $\la v, Lv\ra \leq 0$
and $\la \alpha^-, v\ra = \la \alpha^+, v\ra = \la \partial_x Q, v\ra = 0$.
In particular, Lemma~\ref{lem:coer} yields $v = 0$, which is impossible.
\end{proof}
We also need a localized version of the last coercivity result.
\begin{lemma}
\label{lem:vLv-loc-coer}
There exists $\lambda_0 > 0$ such that the following is true.
For any $c > 0$ there exists $\rho > 0$ such that for all $v \in H^1$
\begin{equation}
\label{eq:vLv-loc-coer}
\begin{aligned}
&(1-\lambda_0)\int_{-\rho}^\rho\big((\partial_x v)^2 + v^2\big)\ud x - \int_\bR f'(Q)v^2\ud x \\
&\qquad\geq -c \|v\|_{H^1}^2 - \frac{1}{\lambda_0}\big(\la \alpha^-, v\ra^2
+ \la \alpha^+, v\ra^2 + \la \partial_x Q, v\ra^2\big).
\end{aligned}
\end{equation}
\end{lemma}
\begin{proof}
Let $\chi$ be a cut-off function supported in $[{-}1, 1]$, $\chi(x) = 1$ for $x \in \big[{-}\frac 12, \frac 12\big]$
and let $\rho \gg 1$. Let $\wt v := \chi\big(\frac{\cdot}{\rho}\big)v$.
By the Chain Rule,
\begin{equation}
\partial_x \wt v = \frac{1}{\rho}\partial_x\chi\big(\frac{\cdot}{\rho}\big)v + \chi\big(\frac{\cdot}{\rho}\big)\partial_x v,
\end{equation}
which implies
\begin{equation}
\Big|\|\partial_x \wt v\|_{L^2} - \big\|\chi\big(\frac{\cdot}{\rho}\big)\partial_x v\big\|_{L^2}\Big|\lesssim \frac 1\rho\|v\|_{L^2} \ll \|v\|_{H^1}.
\end{equation}
Let $\gamma := (1-\lambda_0)\chi\big(\frac{\cdot}{\rho}\big)^2$.
Applying Proposition~\ref{prop:vLv-coer} to the function
$\wt v$ we obtain
\begin{equation}
\label{eq:vLv-loc-1}
\begin{aligned}
&\int_\bR \gamma\big((\partial_x v)^2 + v^2\big)\ud x \geq (1-\lambda_0)\int_\bR \big((\partial_x \wt v)^2
+ \wt v^2\big)\ud x - \frac c4 \|v\|_{H^1}^2 \\
&\qquad \geq \int_\bR f'(Q)\wt v^2\ud x - \frac{1}{\lambda_0}\big(\la \alpha^-, \wt v\ra^2
+ \la \alpha^+, \wt v\ra^2 + \la \partial_x Q, \wt v\ra^2\big) - \frac c4 \|v\|_{H^1}^2.
\end{aligned}
\end{equation}
We have
\begin{equation}
\begin{aligned}
|\la \alpha^-, v\ra^2 - \la \alpha^-, \wt v\ra^2| &= |\la \alpha^-, (1-\chi(\cdot/\rho))v\ra
\la \alpha^-, (1+\chi(\cdot/\rho))v\ra|  \\
&\lesssim \| v\|_{L^2}|\la (1-\chi(\cdot/\rho))\alpha^-, v\ra|.
\end{aligned}
\end{equation}
But $\|(1-\chi(\cdot/\rho))\alpha^-\|_{L^2}$ can be made arbitrarily
small by taking $\rho$ large enough, so we can ensure that
\begin{equation}
|\la \alpha^-, v\ra^2 - \la \alpha^-, \wt v\ra^2| \leq \frac{c\lambda_0}{6}\| v\|_{L^2}^2 \leq \frac{c\lambda_0}{6}\| v\|_{H^1}^2,
\end{equation}
and analogously for similar terms involving $ \alpha^+$ and $\partial_x Q$. Thus \eqref{eq:vLv-loc-1} implies
\begin{multline}
\label{eq:vLv-loc-2}
\int_\bR \gamma\big((\partial_x v)^2 + v^2\big)\ud x 
\geq \int_\bR f'(Q)\wt v^2\ud x \\
- \frac{1}{\lambda_0}\big(\la \alpha^-, v\ra^2
+ \la \alpha^+, v\ra^2 + \la \partial_x Q, v\ra^2\big) - \frac{3c}{4} \|v\|_{H^1}^2.
\end{multline}
Finally, we have
\begin{equation}
\begin{aligned}
\Big|\int_\bR f'(Q) v^2\ud x - \int_\bR f'(Q)\wt v^2\ud x\Big| &= \int_\bR f'(Q)(1-\chi(\cdot/\rho)^2)v^2\ud x \\
&\leq \|f'(Q)(1-\chi(\cdot/\rho)^2)\|_{L^\infty}\|v\|_{L^2}^2.
\end{aligned}
\end{equation}
By taking $\rho$ large enough, we can ensure that the last term is $\leq \frac c4 \|v\|_{H^1}^2$,
so that \eqref{eq:vLv-loc-2} yields the conclusion.
\end{proof}
\section{Coercivity near a two-soliton}
\label{sec:variational}
Following Weinstein \cite{Weinstein85}, we will make an extensive use of the following functional:
\begin{equation}
H(u) := E(u) + \frac 12 M(u),\qquad\text{for }u \in H^1(\bR).
\end{equation}
We are interested in coercivity properties of $H(u)$, for $u$ close to a sum of
two translated copies of~$Q$.

In the next lemma, we gather some easy facts which will be frequently used
to bound various interaction terms. We skip the standard proof.
\begin{lemma}
\label{lem:Taylor}
Fix $M > 0$. For all $u, v \in \bR$ such that $|u| + |v| \leq M$
the following inequalities hold:
\begin{align}
|f(u + v) - f(u)| &\lesssim |v|, \label{eq:f-Lip} \\
|f(u + v) - f(u) - f(v)| &\lesssim |uv|, \label{eq:f-cross} \\
|f(u + v) - f(u) - f'(u)v| &\lesssim v^2, \label{eq:f-1stTay} \\
|F(u + v) - F(u) - f(u)v| &\lesssim v^2, \label{eq:F-1stTay} \\
\Big|F(u + v) - F(u) - f(u)v-\frac 12 f'(u)v^2\Big| &\lesssim |v|^3, \label{eq:F-2ndTay}
\end{align}
with constants depending on $M$.
\end{lemma}
\begin{lemma}
\label{lem:D2H}
Fix $\sigma \in \{{-}1, 1\}$.
There exist constants $\delta, \lambda_0, L_0 > 0$ such that if
$y_2 - y_1 \geq L_0$ and $\|U - (Q(\cdot - y_1) + \sigma Q(\cdot - y_2))\|_{L^\infty} \leq \delta$,
then for all $ \varepsilon \in H^1$
\begin{equation}
\label{eq:D2H-coer}
\begin{aligned}
\la \varepsilon, \vD^2 H(U) \varepsilon\ra \geq \lambda_0 \| \varepsilon\|_{H^1}^2
- &\frac{1}{\lambda_0}\big( \la \alpha^-(\cdot - y_1), \varepsilon\ra^2 +
\la\alpha^+(\cdot - y_1), \varepsilon\ra^2 + \la\partial_x Q(\cdot - y_1), \varepsilon\ra^2 \\
&+\la\alpha^-(\cdot - y_2), \varepsilon\ra^2 + 
\la\alpha^+(\cdot - y_2), \varepsilon\ra^2 + \la\partial_x Q(\cdot - y_2), \varepsilon\ra^2\big).
\end{aligned}
\end{equation}
\end{lemma}
\begin{proof}
Without loss of generality we can assume that $y_1 = 0$ and $y_2 = y \geq L_0$.
Consider the operator $T_y$ defined by the formula
\begin{equation}
\label{eq:Tq-def}
T_y := -\partial_x^2 - f'(Q) - f'(Q(\cdot - y)) + 1.
\end{equation}
We have $\vD^2 H(U) = -\partial_x^2 - f'(U) + 1$, hence
\begin{equation}
\label{eq:D2H-Tq}
\la \varepsilon, \vD^2 H(U) \varepsilon\ra - \la \varepsilon, T_y \varepsilon\ra
= -\int_\bR \big(f'(U) - f'(Q) - f'(Q(\cdot - y))\big)\varepsilon^2\ud x.
\end{equation}
Let $c > 0$. Since $f'$ is locally Lipschitz, we have
\begin{equation}
\label{eq:f'U-1}
\|f'(U) - f'(Q + \sigma Q(\cdot - y))\|_{L^\infty} \leq \frac c2,
\end{equation}
provided that we take $\delta$ small enough.
Considering separately the regions $x \leq \frac y2$ and $x \geq \frac y2$
one can check that
\begin{equation}
\label{eq:f'QQq}
\|f'(Q + \sigma Q(\cdot - y)) - f'(Q) - f'(Q(\cdot - y))\|_{L^\infty} \leq \frac c2,
\end{equation}
if $L_0$ is sufficiently large.
From \eqref{eq:D2H-Tq}, \eqref{eq:f'U-1} and \eqref{eq:f'QQq} we obtain
\begin{equation}
\label{eq:close-Tq}
|\la \varepsilon, \vD^2 H(U) \varepsilon\ra - \la \varepsilon, T_y \varepsilon\ra | \leq c\| \varepsilon\|_{L^2}^2,\qquad \forall \varepsilon \in H^1.
\end{equation}
Since $c$ is arbitrary, it suffices to prove \eqref{eq:D2H-coer} with $\vD^2 H(U)$
replaced by $T_y$.

From Lemma~\ref{lem:vLv-loc-coer} we have
\begin{equation}
\label{eq:D2H-coer-1}
\begin{aligned}
&(1-\lambda_0)\int_{-\rho}^\rho\big((\partial_x  \varepsilon)^2 + \varepsilon^2\big)\ud x - \int_\bR f'(Q)\varepsilon^2\ud x \\
&\quad\geq -c \| \varepsilon\|_{H^1}^2 - \frac{1}{\lambda_0}\big(\la \alpha^-, \varepsilon\ra^2
+ \la \alpha^+, \varepsilon\ra^2 + \la \partial_x Q, \varepsilon\ra^2\big)
\end{aligned}
\end{equation}
and
\begin{equation}
\label{eq:D2H-coer-2}
\begin{aligned}
&(1-\lambda_0)\int_{y-\rho}^{y+\rho}\big((\partial_x \varepsilon)^2 + \varepsilon^2\big)\ud x - \int_\bR f'(Q(\cdot -y))\varepsilon^2\ud x \\
&\quad\geq -c \| \varepsilon\|_{H^1}^2 - \frac{1}{\lambda_0}\big(\la \alpha^-(\cdot -y), \varepsilon\ra^2
+ \la \alpha^+(\cdot -y), \varepsilon\ra^2 + \la \partial_x Q(\cdot -y), \varepsilon\ra^2\big).
\end{aligned}
\end{equation}
Now, if $y \geq 2\rho$, then it suffices to take the sum of \eqref{eq:D2H-coer-1} and \eqref{eq:D2H-coer-2},
and add $\lambda_0\| \varepsilon\|_{H^1}^2$ to both sides.
\end{proof}

\begin{lemma}
\label{lem:H-deux-solit}
Let $k_0 = k_0(1)$ be the constant from Lemma~\ref{lem:Q-asym}. Then
\begin{equation}
\label{eq:H-deux-solit}
H(Q(\cdot - y_1) + \sigma Q(\cdot - y_2)) = 2H(Q) - \sigma(2k_0^2 + o(1)) \eee^{-(y_2 - y_1)},
\end{equation}
where $o(1)$ tends to $0$ as $y_2 - y_1 \to +\infty$.
\end{lemma}
\begin{proof}
We introduce the following notation, which we will often use later:
\begin{equation}
R_1(x) := Q(x - y_1),\qquad R_2(x) := Q(x - y_2).
\end{equation}
We also denote
\begin{equation}
m := \frac{y_1 + y_2}{2}, \quad m_1 := \frac{y_1 + m}{2} = \frac{3y_1 + y_2}{4}, \quad m_2 := \frac{m + y_2}{2}
= \frac{y_1 + 3y_2}{4}.
\end{equation}
We have
\begin{equation}
\label{eq:energy-comp-0}
\begin{aligned}
H(R_1 + \sigma R_2) = \int_\bR \Big(\frac 12 (\partial_x R_1 + \sigma\partial_x R_2)^2
+ \frac 12 (R_1 + \sigma R_2)^2 - F(R_1 + \sigma R_2)\Big)\ud x \\
= 2H(Q) + \int_\bR \big(\sigma \partial_x R_1 \partial_x R_2 + \sigma R_1R_2 - \big(F(R_1 + \sigma R_2)-F(R_1) - F(R_2)\big)\big)\ud x.
\end{aligned}
\end{equation}
In order to compute the main term of the last integral, we consider separately $x \leq m$ and $x \geq m$.
Integrating by parts and using \eqref{eq:Qv}, we get
\begin{equation}
\label{eq:energy-comp-1}
\int_{-\infty}^m \big(\sigma \partial_x R_1 \partial_x R_2 + \sigma R_1R_2)\ud x
= \sigma \partial_x R_1(m)R_2(m) + \int_{-\infty}^m \sigma f(R_1)R_2\ud x.
\end{equation}
Since $|F(u)| \ll |u|^2$ for $|u|$ small, Lemma~\ref{lem:Q-asym} easily implies
$\big|\int_{-\infty}^m F(R_2)\ud x\big| \ll \eee^{-(y_2 - y_1)}$.
Together with \eqref{eq:energy-comp-1}, this yields
\begin{equation}
\label{eq:energy-comp-2}
\begin{aligned}
\int_{-\infty}^m \big(\sigma \partial_x R_1 \partial_x R_2 + \sigma R_1R_2 - \big(F(R_1 + \sigma R_2)-F(R_1) - F(R_2)\big)\big)\ud x \\
\simeq \sigma \partial_x R_1(m)R_2(m) -\int_{-\infty}^m \big(F(R_1 + \sigma R_2) - F(R_1) - \sigma f(R_1) R_2\big)\ud x
\end{aligned}
\end{equation}
(where in this proof ``$\simeq$'' always means ``up to terms of order $\ll \eee^{-(y_2 - y_1)}$'').
From \eqref{eq:F-2ndTay} we obtain
\begin{equation}
\begin{aligned}
\int_{-\infty}^m \big(F(R_1 + \sigma R_2) - F(R_1) - \sigma f(R_1) R_2\big)\ud x \simeq \frac 12 \int_{-\infty}^m
f'(R_1)R_2^2\ud x \\ = \frac 12 \int_{-\infty}^{m_1}f'(R_1)R_2^2\ud x + \frac 12 \int_{m_1}^mf'(R_1)R_2^2\ud x.
\end{aligned}
\end{equation}
By Lemma~\ref{lem:Q-asym}, the first integral is $\lesssim \eee^{-\frac 32(y_2 - y_1)} \ll \eee^{-(y_2 - y_1)}$.
The second integral is also $\ll \eee^{-(y_2 - y_1)}$, because $|f'(R_1)| \ll 1$ for $x \geq m_1$.
Taking this into account, we get from \eqref{eq:energy-comp-2}
\begin{equation}
\label{eq:energy-comp-3}
\begin{aligned}
\int_{-\infty}^m \big(\sigma \partial_x R_1 \partial_x R_2 + \sigma R_1R_2 - \big(F(R_1 + \sigma R_2)-F(R_1) - F(R_2)\big)\big)\ud x \\
\simeq \sigma \partial_x R_1(m)R_2(m) \simeq -\sigma k_0^2 \eee^{-(y_2 - y_1)},
\end{aligned}
\end{equation}
where the last step follows from Lemma~\ref{lem:Q-asym}.

A similar computation yields
\begin{equation}
\label{eq:energy-comp-4}
\begin{aligned}
\int_m^\infty \big(\sigma \partial_x R_1 \partial_x R_2 + \sigma R_1R_2 - \big(F(R_1 + \sigma R_2)-F(R_1) - F(R_2)\big)\big)\ud x \\
\simeq -\sigma \partial_x R_2(m)R_1(m) \simeq -\sigma k_0^2 \eee^{-(y_2 - y_1)}.
\end{aligned}
\end{equation}
The conclusion directly follows from \eqref{eq:energy-comp-0}, \eqref{eq:energy-comp-3}
and \eqref{eq:energy-comp-4}.
\end{proof}

\begin{proposition}
\label{prop:coer}
There exist $\delta, L_0, C_0 > 0$ such that if $H(Q(\cdot - y_1) + \sigma Q(\cdot - y_2) + \varepsilon) = 2H(Q)$,
$\| \varepsilon\|_{H^1} \leq \delta$ and $y_2 - y_1 \geq L_0$, then
\begin{itemize}
\item in the case $\sigma = -1$,
\begin{equation}
\begin{aligned}
\| \varepsilon\|_{H^1}^2 + \eee^{-(y_2 - y_1)} \leq C_0\big(&\la \alpha^-(\cdot - y_1), \varepsilon\ra^2 +
\la\alpha^+(\cdot - y_1), \varepsilon\ra^2 + \la\partial_x Q(\cdot - y_1), \varepsilon\ra^2 \\
+ &\la\alpha^-(\cdot - y_2), \varepsilon\ra^2 + 
\la\alpha^+(\cdot - y_2), \varepsilon\ra^2 + \la\partial_x Q(\cdot - y_2), \varepsilon\ra^2\big).
\end{aligned}
\end{equation}
\item in the case $\sigma = 1$,
\begin{equation}
\begin{aligned}
\| \varepsilon\|_{H^1}^2 \leq C_0\big(\eee^{-(y_2 - y_1)} +  &\la \alpha^-(\cdot - y_1), \varepsilon\ra^2 +
\la\alpha^+(\cdot - y_1), \varepsilon\ra^2 + \la\partial_x Q(\cdot - y_1), \varepsilon\ra^2 \\
+ &\la\alpha^-(\cdot - y_2), \varepsilon\ra^2 + 
\la\alpha^+(\cdot - y_2), \varepsilon\ra^2 + \la\partial_x Q(\cdot - y_2), \varepsilon\ra^2\big).
\end{aligned}
\end{equation}
\end{itemize}
\end{proposition}
\begin{proof}
Denote $R_1 := Q(\cdot - y_1)$, $R_2 := Q(\cdot - y_2)$ and $U := R_1 + \sigma R_2$.
We have the Taylor expansion
\begin{equation}
\label{eq:H-Tay}
H(U + \varepsilon) = H(U) + \la \vD H(U), \varepsilon\ra + \frac 12 \la \varepsilon, \vD^2 H(U) \varepsilon\ra
+ O(\| \varepsilon\|_{H^1}^3).
\end{equation}
Indeed, from the definition of $H$ we obtain
\begin{equation}
\begin{aligned}
H(U + \varepsilon) - \Big( H(U) + \la \vD H(U), \varepsilon\ra + \frac 12 \la \varepsilon, \vD^2 H(U) \varepsilon\ra
\Big) \\
= -\int_\bR \Big(F(U + \varepsilon) - F(U) - f(U) \varepsilon - \frac 12 f'(U) \varepsilon^2\Big)\ud x.
\end{aligned}
\end{equation}
Now \eqref{eq:F-2ndTay} yields
\begin{equation}
\Big|H(U + \varepsilon) - \Big( H(U) + \la \vD H(U), \varepsilon\ra + \frac 12 \la \varepsilon, \vD^2 H(U) \varepsilon\ra
\Big)\Big| \lesssim \int_\bR | \varepsilon|^3 \ud x = \| \varepsilon\|_{L^3}^3 \lesssim \| \varepsilon\|_{H^1}^3.
\end{equation}
Replacing $H(U)$ in \eqref{eq:H-Tay} by the formula given in Lemma~\ref{lem:H-deux-solit} and using the assumption
$H(U+ \varepsilon) = 2H(Q)$, we get
\begin{equation}
\label{eq:H-Tay-1}
- 2\sigma k_0^2 \eee^{-(y_2 - y_1)} + \la \vD H(U), \varepsilon\ra + \frac 12 \la \varepsilon, \vD^2 H(U) \varepsilon\ra = o\big(\eee^{-(y_2 - y_1)} + \| \varepsilon\|_{H^1}^2\big).
\end{equation}
We now show that
\begin{equation}
\label{eq:small-tension-0}
|\la \vD H(U), \varepsilon\ra| \ll \eee^{-(y_2 - y_1)} + \| \varepsilon\|_{H^1}^2.
\end{equation}
By the Cauchy-Schwarz inequality, it suffices to check that
\begin{equation}
\label{eq:small-tension}
\|f(R_1 + \sigma R_2) - f(R_1) - \sigma f(R_2)\|_{L^2} \ll \eee^{-\frac 12(y_2 - y_1)}.
\end{equation}
This proof is similar to the computations in Lemma~\ref{lem:H-deux-solit}.
For $x \leq m$ we have
\begin{equation}
\big|f(R_1 + \sigma R_2) - f(R_1) - \sigma f(R_2) - \sigma f'(R_1)R_2\big| \lesssim R_2^2,
\end{equation}
and $\|R_2^2\|_{L^2(x \leq m)} \lesssim \eee^{-(y_2 - y_1)} \ll \eee^{-\frac 12(y_2 - y_1)}$.
Considering separately $x \leq m_1$ and $m_1 \leq x \leq m$,
it is easy to see that $\|f'(R_1)R_2\|_{L^2(x \leq m)} \ll \eee^{-\frac 12(y_2 - y_1)}$.
Thus
\begin{equation}
\|f(R_1 + \sigma R_2) - f(R_1) - \sigma f(R_2)\|_{L^2(x \leq m)} \ll \eee^{-\frac 12(y_2 - y_1)},
\end{equation}
and a similar argument yields the same estimate for $x \geq m$. This proves \eqref{eq:small-tension}.

From \eqref{eq:H-Tay-1} and \eqref{eq:small-tension-0} we have
\begin{equation}
\frac 12 \la  \varepsilon, \vD^2 H(U) \varepsilon\ra - 2\sigma k_0^2 \eee^{-(y_2 - y_1)} = o\big(\eee^{-(y_2 - y_1)} + \| \varepsilon\|_{H^1}^2\big),
\end{equation}
so \eqref{eq:D2H-coer} yields the conclusion, both for $\sigma = 1$ and $\sigma = -1$.
\end{proof}
\section{Modulation near a two-soliton}
\label{sec:modulation-two}
This section is the heart of our proof. We show here how a good choice of modulation parameters
allows to identify the interaction force in the modulation equations.
\subsection{Definition of the position parameters}
\label{ssec:orth}
We consider a solution which is close to a two-soliton
on some time interval (with velocities of both solitons close to $1$):
\begin{equation}
\label{eq:u-decomp}
u(t, x+t) = Q(x-y_1(t)) + \sigma Q(x - y_2(t)) + \varepsilon(t, x),
\end{equation}
where $y_2(t) - y_1(t) \gg 1$ and $\| \varepsilon(t)\|_{H^1} \ll 1$.
We also set $\sigma_1 = 1$ and $\sigma_2 = \sigma$.
Note the simple relation between $y_k$ and the parameters $x_k$ used in the Introduction: $y_k(t) = x_k(t) - t$.

Given $y_1(t)$ and $y_2(t)$, we denote $R_k(t, x) := Q(x - y_k(t))$ for $k \in \{1, 2\}$.
Note that $\partial_x R_k(t, x) = \partial_xQ(x - y_k(t))$. By the Chain Rule,
we also have $\partial_t R_k(t) = -y_k'(t)\partial_x R_k(t)$ and
$\partial_t \partial_x R_k(t) = -y_k'(t)\partial_x^2 R_k(t)$.
The values of $y_1(t)$ and $y_2(t)$ are chosen as follows.
\begin{lemma}
\label{lem:mod0}
There exist $\delta, L_0, C_0 > 0$ such that if $L \geq L_0$ and
\begin{equation}
\inf_{x_2 - x_1 \geq L} \|u(t) - Q(\cdot - x_1) - \sigma Q(\cdot - x_2)\|_{H^1}
= \wt \delta < \delta,\qquad\text{for all }t \in [T_1, T_2],
\end{equation}
then for $t \in [T_1, T_2]$ there exist unique $y_1(t), y_2(t)$ such that
$\varepsilon(t, x) := u(t, x + t) - R_1(t, x) - \sigma R_2(t, x)$ satisfies
\begin{gather}
y_2(t) - y_1(t) \geq L_0 - 1, \\
\| \varepsilon(t) \|_{H^1} < C_0 \delta, \\
\label{eq:eps-orth}
\la \partial_x R_1(t), \varepsilon(t)\ra = \la \partial_x R_2(t), \varepsilon(t)\ra = 0.
\end{gather}
These functions satisfy
$y_2(t) - y_1(t) \geq L - 1$ and $\| \varepsilon(t) \|_{H^1} \leq C_0 \wt \delta$ for $t \in [T_1, T_2]$.
Moreover, $y_1(t)$ and $y_2(t)$ are of class $C^1$ and
\begin{equation}
\label{eq:dtyk}
|y_1'(t)| + |y_2'(t)| \leq C_0\|\varepsilon(t)\|_{H^1} + c_0\eee^{-\frac 12(y_2(t) - y_1(t))},
\end{equation}
where $c_0 > 0$ can be made arbitrarily small by taking $L_0$ large enough.
\end{lemma}
\begin{proof}
Existence and uniqueness of $y_1(t)$ and $y_2(t)$ is a standard application of the Implicit Function Theorem
and we skip it, see for example \cite[Proposition 3]{GuSi06}.

In order to prove \eqref{eq:dtyk}, we need the evolution equation of $\varepsilon(t)$.
Differentiating \eqref{eq:u-decomp} in time we obtain
\begin{equation}
\label{eq:u-decomp-1}
\partial_t u(t, x+t) + \partial_x u(t, x+t) = -y_1'(t)\partial_x R_1(t, x) - \sigma y_2'(t)\partial_x R_2(t, x) +\partial_t \varepsilon(t, x).
\end{equation}
From \eqref{eq:kdv} we have
\begin{equation}
\partial_t u(t, x+t) + \partial_x u(t, x+t) =
\partial_x\big({-}\partial_x^2 u(t, x+t) - f(u(t, x+t)) + u(t, x+t)\big).
\end{equation}
Using again \eqref{eq:u-decomp}, we obtain that the right hand side is
\begin{equation}
\partial_x\big({-}\partial_x^2 R_1 - \sigma \partial_x^2 R_2 -\partial_x^2 \varepsilon
- f(R_1 + \sigma R_2 + \varepsilon)+R_1 + \sigma R_2 + \varepsilon\big),
\end{equation}
which, using $\partial_x^2Q + f(Q) = Q$, is equal to
\begin{equation}
\partial_x\big({-}\partial_x^2 \varepsilon - f(R_1 + \sigma R_2 + \varepsilon)
+ f(R_1) + \sigma f(R_2)+ \varepsilon\big).
\end{equation}
Combining this with \eqref{eq:u-decomp-1} we get
\begin{equation}
\label{eq:dteps}
\partial_t \varepsilon = y_1'\partial_x R_1 + \sigma y_2' \partial_x R_2
+\partial_x\big({-}\partial_x^2 \varepsilon - f(R_1 + \sigma R_2 + \varepsilon)
+ f(R_1) + \sigma f(R_2)+ \varepsilon\big).
\end{equation}
From \eqref{eq:eps-orth} and \eqref{eq:dteps}, we have, for $k \in \{1, 2\}$,
\begin{equation}
\label{eq:dt-orth-all}
\begin{aligned}
0 &= \dd t\la \partial_x R_k(t), \varepsilon(t)\ra =
-y_k'(t) \la \partial_x^2 R_k(t), \varepsilon(t)\ra + \la \partial_x R_k(t), \partial_t\varepsilon(t)\ra \\
&= -y_k'(t) \la \partial_x^2 R_k(t), \varepsilon(t)\ra \\
&+\big\la \partial_x R_k, y_1'\partial_x R_1 + \sigma y_2' \partial_x R_2
+\partial_x\big({-}\partial_x^2 \varepsilon - f(R_1 + \sigma R_2 + \varepsilon)
+ f(R_1) + \sigma f(R_2)+ \varepsilon\big)\big\ra.
\end{aligned}
\end{equation}
We obtain the following linear system for $y_1'$ and $y_2'$:
\begin{equation}
\label{eq:orth-syst}
\begin{pmatrix} M_{11} & M_{12} \\ M_{21} & M_{22} \end{pmatrix}
\begin{pmatrix} y_1' \\ y_2' \end{pmatrix} = \begin{pmatrix} B_1 \\ B_2 \end{pmatrix},
\end{equation}
where
\begin{align}
M_{11} &= \la \partial_x R_1, \partial_x R_1\ra - \la \partial_x^2 R_1, \varepsilon\ra, \\
M_{12} &= \sigma\la \partial_x R_1, \partial_x R_2\ra, \\
M_{21} &= \la \partial_x R_2, \partial_x R_1\ra, \\
M_{22} &= \sigma\la \partial_x R_2, \partial_x R_2\ra - \la \partial_x^2 R_2, \varepsilon\ra, \\
B_1 &= \big\la \partial_x R_1, \partial_x\big(\partial_x^2 \varepsilon + f(R_1 + \sigma R_2 + \varepsilon)
- f(R_1) - \sigma f(R_2) - \varepsilon\big)\big\ra, \\
B_2 &= \big\la \partial_x R_2, \partial_x\big(\partial_x^2 \varepsilon + f(R_1 + \sigma R_2 + \varepsilon)
- f(R_1) - \sigma f(R_2) - \varepsilon\big)\big\ra.
\end{align}
The diagonal terms are of size $\sim 1$, whereas the off-diagonal terms are small when $L_0$ is large.
Moreover, \eqref{eq:f-Lip} and \eqref{eq:small-tension} imply
$|B_1| + |B_2| \leq  \wt C_0\|\varepsilon\|_{H^1} + \wt c_0\eee^{-\frac 12(y_2 - y_1)}$,
where $\wt c_0$ is small when $L_0$ is large, so we get \eqref{eq:dtyk}.
\end{proof}

The orthogonality condition \eqref{eq:eps-orth} was not chosen very carefully.
In fact, we could just as well use a different one.
For this reason, $y_k'(t)$ is not sufficiently well controlled.
To remedy this, we will now introduce a different parameter $q_k(t)$,
such that $|q_k(t) - y_k(t)| \ll 1$, but $q_k'(t)$ behaves better.

We set
\begin{equation}
\label{eq:Z-def}
\cZ(x) := \chi\Big(\frac{x}{\rho}\Big)\int_0^x \wt Q(y)\ud y,
\end{equation}
where $\chi$ is a cut-off function supported in $[-2, 2]$, $\chi(x) = 1$ for $x \in [-1, 1]$
and $\rho \gg 1$. Note that
\begin{equation}
\int_\bR \partial_xQ(x)\Big(\int_0^x \wt Q(y)\ud y\Big)\ud x = {-}\int_\bR Q(x)\wt Q(x)\ud x > 0,
\end{equation}
where in the last step we use \eqref{eq:supercrit}.

Since $\rho$ is large, the triangle inequality yields
\begin{equation}
\label{eq:Z-Q'-sgn}
\la \cZ, \partial_xQ\ra > 0.
\end{equation}

We write $\cZ_k(t, x) := \cZ(x - y_k(t))$ for $k \in \{1, 2\}$.
Note that $\partial_t \cZ_k(t, x) = -y_k'(t)\partial_x \cZ_k(t, x)$.
\begin{lemma}
\label{lem:LcZ'}
For any $c > 0$ there exists $\rho_0 > 0$ such that if $\rho \geq \rho_0$, then
\begin{equation}
\|L(\partial_x \cZ) + Q\|_{L^2} \leq c.
\end{equation}
\end{lemma}
\begin{proof}
We compute $L(\partial_x \cZ)$ applying the Product Rule:
\begin{align}
\partial_x\cZ(x) &= \frac{1}{\rho}\partial_x\chi\Big(\frac{x}{\rho}\Big)\int_0^x \wt Q(y)\ud y +
\chi\Big(\frac{x}{\rho}\Big)\wt Q(x), \label{eq:cZ'} \\
\partial_x^2\cZ(x) &= \frac{1}{\rho^2}\partial_x^2\chi\Big(\frac{x}{\rho}\Big)\int_0^x \wt Q(y)\ud y +
\frac{2}{\rho}\partial_x\chi\Big(\frac{x}{\rho}\Big)\wt Q(x)
+ \chi\Big(\frac{x}{\rho}\Big)\partial_x\wt Q(x), \label{eq:cZ''} \\
 &\begin{aligned}\mathllap{\partial_x^3\cZ(x)}&= \frac{1}{\rho^3}\partial_x^3\chi\Big(\frac{x}{\rho}\Big)\int_0^x \wt Q(y)\ud y +
\frac{3}{\rho^2}\partial_x^2\chi\Big(\frac{x}{\rho}\Big)\wt Q(x) \\
&+ \frac{3}{\rho}\partial_x\chi\Big(\frac{x}{\rho}\Big)\partial_x\wt Q(x)
+ \chi\Big(\frac{x}{\rho}\Big)\partial_x^2\wt Q(x). \end{aligned}\label{eq:cZ'''}
\end{align}
We claim that
\begin{align}
\label{eq:cZderiv-bound}
\|\partial_x\cZ - \wt Q\|_{L^2} + \|\partial_x^3\cZ - \partial_x^2\wt Q\|_{L^2} \lesssim \sqrt{1/\rho}.
\end{align}
In order to see this, note that the functions
$x \mapsto \int_0^x \wt Q(y)\ud y$, $\wt Q$, $\partial_x\wt Q$, $\partial_x^2\wt Q$, $\partial_x\chi$, $\partial_x^2\chi$ and $\partial_x^3\chi$ are bounded.
Moreover, $\partial_x\chi(\cdot / \rho)$, $\partial_x^2\chi(\cdot / \rho)$ and $\partial_x^3\chi(\cdot / \rho)$
are supported on an interval of length $\lesssim \rho$.
Therefore, in the formulas \eqref{eq:cZ'}, \eqref{eq:cZ''} and \eqref{eq:cZ'''},
all the terms containing derivatives of $\chi$ are functions with $L^\infty$ norms
$\lesssim 1/\rho$ and with supports of measure $\lesssim \rho$.
The~$L^2$~norm of such a function is $\lesssim \sqrt{1/\rho}$.

To finish the proof of \eqref{eq:cZderiv-bound}, it suffices to notice that
\begin{equation}
\big\|\big(1-\chi(\cdot/\rho)\big)\wt Q\big\|_{L^2} + \big\|\big(1-\chi(\cdot/\rho)\big)\partial_x^2\wt Q\big\|_{L^2} \lesssim \sqrt{1/\rho}.
\end{equation}
In fact, by Lemma~\ref{lem:Q-asym}, the right hand side could even be replaced
by an exponentially decaying function.

The function $f'(Q)$ is bounded, so \eqref{eq:cZderiv-bound} implies  $\|f(Q)(\partial_x\cZ - \wt Q)\|_{L^2} \lesssim \sqrt{1/\rho}$. Using \eqref{eq:L-ker}, we obtain
\begin{equation}
\begin{aligned}
\|L(\partial_x\cZ) + Q\|_{L^2} &= \|L(\partial_x\cZ) - L\wt Q\|_{L^2} \\
&= \|{-}(\partial_x^3\cZ - \partial_x^2\wt Q) - f'(Q)(\partial_x\cZ - \wt Q)
+ (\partial_x\cZ - \wt Q)\|_{L^2} \lesssim \sqrt{1/\rho}.
\end{aligned}
\end{equation}
~\vspace{-2em}	~

\end{proof}

For $k \in \{1, 2\}$ we define
\begin{equation}
\label{eq:qk-def}
q_k(t) := y_k(t) + \sigma_k \frac{\la \cZ_k(t), \varepsilon(t)\ra}{\la Q, \wt Q\ra}.
\end{equation}
By the Cauchy-Schwarz inequality, for fixed $\rho > 0$ we have $|q_k(t) - y_k(t)| \to 0$ as $\|\varepsilon(t)\|_{L^2} \to 0$.

Let $\psi \in C^\infty(\bR)$ be a decreasing function such that $\psi(x) = 1$ for $x \leq \frac 13$
and $\psi(x) = 0$ for $x \geq \frac 23$.
We set
\begin{align}
\phi_1(t, x) &:= \psi\Big(\frac{x-y_1(t)}{y_{2}(t)-y_1(t)}\Big), \label{eq:phi1-def} \\
\phi_2(t, x) &:= 1 - \phi_1(t, x). \label{eq:phi2-def}
\end{align}

Finally, for $k \in \{1, 2\}$ we define
\begin{equation}
\label{eq:pk}
p_k(t) := \la \sigma_k R_k(t), \varepsilon(t)\ra + \frac 12\int_\bR \phi_k(t) \varepsilon(t)^2\ud x.
\end{equation}
Again, by the Cauchy-Schwarz inequality, we have $p_k(t) \to 0$ as $\|\varepsilon(t)\|_{L^2} \to 0$.
Note that $p_k(t)$ is related to the momentum localized around each soliton.
As expected, $p_1(t)$ and $p_2(t)$ will play the role of the momentum
in the reduced finite-dimensional dynamical system.

Our first goal is to relate $q_k(t)$ and $p_k(t)$.
\begin{lemma}
\label{lem:mod}
For any $c > 0$ and $\rho > 0$ there exists $\delta > 0$ such that if
\begin{equation}
\label{eq:norm-small-mod}
\eee^{-(y_2(t) - y_1(t))} + \|\varepsilon(t)\|_{H^1}^2 \leq \delta,
\end{equation}
then
\begin{equation}
\label{eq:dtqk}
\big|q_k'(t) -\la Q, \wt Q\ra^{-1} p_k(t)\big| \leq c\sqrt{\eee^{-(y_2(t) - y_1(t))} + \|\varepsilon(t)\|_{H^1}^2}.
\end{equation}
\end{lemma}
\begin{remark}
Condition \eqref{eq:supercrit} implies that, up to the error term, $q_k'$ and $p_k$
have opposite signs.
\end{remark}
\begin{proof}[Proof of Lemma~\ref{lem:mod}]
We only check \eqref{eq:dtqk} for $k = 1$, the proof for $k = 2$ being almost the same.
From \eqref{eq:dteps}, we have
\begin{equation}
\label{eq:dt-Zk}
\begin{aligned}
\dd t\la \cZ_1(t), \varepsilon(t)\ra &=
-y_1'(t) \la \partial_x\cZ_1(t), \varepsilon(t)\ra + \la \cZ_1(t), \partial_t\varepsilon(t)\ra \\
&= -y_1' \la \partial_x\cZ_1, \varepsilon\ra 
+\big\la \cZ_1, y_1'\partial_x R_1 + \sigma y_2' \partial_x R_2 \\
&\qquad\qquad+\partial_x\big({-}\partial_x^2 \varepsilon - f(R_1 + \sigma R_2 + \varepsilon)
+ f(R_1) + \sigma f(R_2)+ \varepsilon\big)\big\ra.
\end{aligned}
\end{equation}
By \eqref{eq:dtyk}, we have $|y_1' \la \partial_x\cZ_1, \varepsilon\ra| \lesssim \| \varepsilon\|_{H^1}^2
+ \eee^{-(y_2(t) - y_1(t))}$,
which is negligible. Consider the second line in the formula \eqref{eq:dt-Zk} above. From \eqref{eq:cZderiv-bound}, it is clear that for $\delta$ small enough we have
\begin{equation}
\big|\la \cZ_1, \partial_x R_1\ra + \la Q, \wt Q\ra \big| = \big|\la Q, \wt Q\ra - \la R_1,  \partial_x \cZ_1\ra \big|
= \big|\la Q, \wt Q - \partial_x\cZ\ra\big|
\leq c,
\end{equation}
so \eqref{eq:dtyk} yields
\begin{equation}
\label{eq:q-bound-1}
\big|\la \cZ_1, y_1'\partial_x R_1\ra + \la Q, \wt Q\ra y_1' \big| \leq c\sqrt{\eee^{-(y_2 - y_1)} + \|\varepsilon\|_{H^1}^2}.
\end{equation}
Similarly, we have $|\la \cZ_1, \partial_x R_2\ra| \leq c$, which yields
\begin{equation}
\label{eq:q-bound-2}
\big|\la \cZ_1, y_2'\partial_x R_2\ra\big| \leq c\sqrt{\eee^{-(y_2 - y_1)} + \|\varepsilon\|_{H^1}^2}.
\end{equation}
Finally, we claim that
\begin{equation}
\label{eq:B1}
\begin{aligned}
\big|\big\la \cZ_1, \partial_x\big({-}\partial_x^2 \varepsilon - f(R_1 + \sigma R_2 + \varepsilon)
+ f(R_1) + \sigma f(R_2)+ \varepsilon\big)\big\ra - \la R_1, \varepsilon\ra \big| \\ \leq c\sqrt{\eee^{-(y_2 - y_1)} + \|\varepsilon\|_{H^1}^2}.
\end{aligned}
\end{equation}

Let $L_1 := -\partial_x^2 - f'(R_1) + 1$, which is obtained by conjugating $L$
with a translation of the variable $x$ by $y_1$.
Integrating by parts, we see that \eqref{eq:B1} is equivalent to
\begin{equation}
\label{eq:B1-1}
\begin{aligned}
\big|\la \partial_x \cZ_1, L_1 \varepsilon - f(R_1 + \sigma R_2 + \varepsilon)
+ f(R_1) + f'(R_1) \varepsilon + \sigma f(R_2)\ra + \la R_1, \varepsilon\ra\big| \\
\leq c\sqrt{\eee^{-(y_2 - y_1)} + \|\varepsilon\|_{H^1}^2}.
\end{aligned}
\end{equation}
By Lemma~\ref{lem:LcZ'} we have
\begin{equation}
\big|\la \partial_x \cZ_1, L_1 \varepsilon \ra + \la R_1, \varepsilon\ra\big|
= \big|\la L_1(\partial_x \cZ_1) + R_1, \varepsilon\ra\big|
\leq \|L_1(\partial_x \cZ_1) + R_1\|_{L^2}\| \varepsilon\|_{L^2} \leq c\| \varepsilon\|_{L^2}.
\end{equation}
In order to finish the proof of \eqref{eq:B1}, we need to check that
\begin{equation}
\label{eq:B1-2}
\big|\la \partial_x \cZ_1, f(R_1 + \sigma R_2 + \varepsilon) - f(R_1)
 - f'(R_1) \varepsilon - \sigma f(R_2)\ra\big| \leq c\sqrt{\eee^{-(y_2 - y_1)} + \|\varepsilon\|_{H^1}^2}.
\end{equation}
We restrict to $x \in [y_1 - 2\rho, y_1 + 2\rho]$, because $\partial_x \cZ_1(x) = 0$
for $x \notin [y_1 - 2\rho, y_1 + 2\rho]$. Thus $R_2 \lesssim \eee^{2\rho}\eee^{-(y_2 - y_1)}$
is small when $\delta$ is small.
By the triangle inequality, \eqref{eq:f-Lip} and \eqref{eq:f-1stTay}, we have
\begin{equation}
\begin{aligned}
&|f(R_1 + \sigma R_2 + \varepsilon) - f(R_1)
 - f'(R_1) \varepsilon - \sigma f(R_2)| \\ &\leq |f(R_1 + \sigma R_2 + \varepsilon) - f(R_1+ \varepsilon)| + |f(R_1+ \varepsilon) - f(R_1) - f'(R_1) \varepsilon| + |f(R_2)|
 \lesssim R_2 + \varepsilon^2.
 \end{aligned}
\end{equation}
Now \eqref{eq:B1-2} follows from the boundedness of $\partial_x \cZ_1$, since
\begin{align}
&\int_{y_1 - 2\rho}^{y_1 + 2\rho} R_2\ud x \lesssim \rho\,\eee^{2\rho}\eee^{-(y_2 - y_1)} \ll
\eee^{-\frac 12(y_2 - y_1)}, \\
&\int_{y_1 - 2\rho}^{y_1 + 2\rho} \varepsilon^2\ud x \lesssim \| \varepsilon\|_{L^2}^2 \ll \| \varepsilon\|_{H^1}.
\end{align}
This finishes the proof of \eqref{eq:B1}.

Putting together \eqref{eq:dt-Zk}, \eqref{eq:q-bound-1}, \eqref{eq:q-bound-2} and \eqref{eq:B1}, we have
\begin{equation}
\Big|\dd t \la \cZ_1, \varepsilon\ra + \la Q, \wt Q\ra y_1' - \la R_1, \varepsilon\ra \Big| \leq
c\sqrt{\eee^{-(y_2 - y_1)} + \|\varepsilon\|_{H^1}^2},
\end{equation}
which, by the definition of $q_1$, yields
\begin{equation}
|q_1' - \la Q, \wt Q\ra^{-1}\la R_1, \varepsilon\ra| \leq c\sqrt{\eee^{-(y_2 - y_1)} + \|\varepsilon\|_{H^1}^2}.
\end{equation}
The definition of $p_k$, see \eqref{eq:pk}, implies
\begin{equation}
|p_1 - \la R_1, \varepsilon\ra | \lesssim \| \varepsilon\|_{L^2}^2 \leq c\sqrt{\eee^{-(y_2 - y_1)} + \|\varepsilon\|_{H^1}^2},
\end{equation}
so the triangle inequality yields \eqref{eq:dtqk}.
\end{proof}

\subsection{Computation of the interaction force}
\label{ssec:inter-force}
Our second goal is to compute $p_k'(t)$ (at least the main term).
We call the second term in the definition of $p_k(t)$ the \emph{correction term}.
In order to treat the derivative of the correction term,
we will need the following easy fact.
\begin{lemma}
\label{lem:virial}
For any $M > 0$ there exists a constant $C > 0$ such that for any functions $\phi$, $R$
and $ \varepsilon$ such that $\|R\|_{H^1} \leq M$, $\|\phi\|_{L^\infty} + \|\partial_x \phi\|_{L^\infty} < +\infty$
and $\| \varepsilon\|_{H^1} \leq 1$ the following inequality is true:
\begin{equation}
\label{eq:virial}
\begin{aligned}
\bigg|&\int_\bR \phi\,\partial_x \varepsilon\big(f(R + \varepsilon) - f(R)\big)\ud x \\
 &+ \int_\bR \phi\,\partial_x R\big(f(R+ \varepsilon) - f(R) - f'(R) \varepsilon\big)\ud x \bigg|
 \leq C\|\partial_x \phi\|_{L^\infty}\| \varepsilon\|_{H^1}^2.
\end{aligned}
\end{equation}
\end{lemma}
\begin{proof}
The assumption $\|\phi\|_{L^\infty} < +\infty$ is only used to ensure that both integrals
on the left hand side are well defined.
By the standard approximation procedure,
we can assume that $ \varepsilon, R \in C_0^\infty(\bR)$.

Rearranging the terms, we obtain
\begin{equation}
\begin{aligned}
&\int_\bR \phi\Big(\partial_x \varepsilon\big(f(R+ \varepsilon) - f(R)\big)
+ \partial_x R\big(f(R+ \varepsilon) - f(R) - f'(R) \varepsilon\big)\Big)\ud x \\
= &\int_\bR \phi\Big(\partial_x (R + \varepsilon)f(R + \varepsilon) - \partial_x R f(R)
 - \big(\partial_x \varepsilon f(R) + \varepsilon\, \partial_x R f'(R)\big)\Big)\ud x \\
 =& \int_\bR \phi\, \partial_x\big(F(R+ \varepsilon) - F(R) - f(R) \varepsilon\big)\ud x.
\end{aligned}
\end{equation}
Integrating by parts and invoking \eqref{eq:F-1stTay} finishes the proof.
\end{proof}

In the next lemma, we compute what will turn out to be the main interaction terms.
\begin{lemma}
\label{lem:inter-asym}
For any $c > 0$ there exists $\delta > 0$ such that if $\eee^{-(y_2 - y_1)} \leq \delta$, then
\begin{align}
\big|\la \partial_x R_1, f(R_1 + \sigma R_2) - f(R_1) - \sigma f(R_2)\ra +\sigma 2k_0^2 e^{-(y_2 - y_1)}\big| \leq c\,e^{-(y_2 - y_1)}, \label{eq:inter-asym-1} \\
\big|\la \partial_x R_2, f(R_1 + \sigma R_2) - f(R_1) - \sigma f(R_2)\ra - 2k_0^2e^{-(y_2 - y_1)}\big| \leq c\,e^{-(y_2 - y_1)}. \label{eq:inter-asym-2}
\end{align}
\end{lemma}
\begin{proof}
We only prove \eqref{eq:inter-asym-1}. Substituting $-x + y_1 + y_2$ for $x$ swaps $R_1$ and $R_2$, so we obtain the identity
\begin{equation}
\la \partial_x R_2, f(R_1 + \sigma R_2) - f(R_1) - \sigma f(R_2)\ra = -\sigma \la \partial_x R_1, f(R_1 + \sigma R_2) - f(R_1) - \sigma f(R_2)\ra,
\end{equation}
and \eqref{eq:inter-asym-2} follows from \eqref{eq:inter-asym-1}.

Let $m := \frac{y_1 + y_2}{2}$.
We claim that
\begin{equation}
\label{eq:inter-asym-3}
\Big| \la \partial_x R_1, f(R_1 + \sigma R_2) - f(R_1) - \sigma f(R_2)\ra - \sigma\int_{-\infty}^m
R_2f'(R_1)\partial_x R_1\ud x\Big| \ll \eee^{-(y_2 - y_1)}.
\end{equation}
Let $m_1 := \frac{y_1 + m}{2} = \frac{3y_1 + y_2}{4}$ and $m_2 := \frac{m+y_2}{2} = \frac{y_1 + 3y_2}{4}$. Bound \eqref{eq:inter-asym-3} will follow from the triangle inequality once we check that
\begin{align}
\label{eq:inter-asym-31}
\int_{-\infty}^{m_1} |\partial_x R_1||f(R_1 + \sigma R_2) - f(R_1) - \sigma f(R_2) - \sigma f'(R_1)R_2|\ud x \ll \eee^{-(y_2 - y_1)}, \\
\label{eq:inter-asym-32}
\int_{m_1}^{m} |\partial_x R_1||f(R_1 + \sigma R_2) - f(R_1) - \sigma f(R_2) - \sigma f'(R_1)R_2|\ud x \ll \eee^{-(y_2 - y_1)}, \\
\label{eq:inter-asym-33}
\int_{m}^{m_2} |\partial_x R_1||f(R_1 + \sigma R_2) - f(R_1) - \sigma f(R_2)|\ud x \ll \eee^{-(y_2 - y_1)}, \\
\label{eq:inter-asym-34}
\int_{m_2}^{+\infty} |\partial_x R_1||f(R_1 + \sigma R_2) - f(R_1) - \sigma f(R_2)|\ud x \ll \eee^{-(y_2 - y_1)}.
\end{align}

For $x \leq m_1$ we have
\begin{equation}
\begin{aligned}
|f(R_1 + \sigma R_2) - f(R_1) - \sigma f(R_2) - \sigma f'(R_1)R_2| &\lesssim |R_2|^2 \lesssim \big(\eee^{-(y_2 - m_1)}\big)^2 \\ &\lesssim \eee^{-\frac 32(y_2 - y_1)} \ll \eee^{-(y_2 - y_1)},
\end{aligned}
\end{equation}
and \eqref{eq:inter-asym-31} follows since $\|\partial_x R_1\|_{L^1} = \|\partial_xQ\|_{L^1} < +\infty$.

For $m_1 \leq x \leq m$ we have, similarly,
\begin{equation}
|f(R_1 + \sigma R_2) - f(R_1) - \sigma f(R_2) - \sigma f'(R_1)R_2| \lesssim |R_2|^2 \lesssim \big(\eee^{-(y_2 - m)}\big)^2 \lesssim \eee^{-(y_2 - y_1)},
\end{equation}
and, by Lemma~\ref{lem:Q-asym}, $|\partial_x R_1| \lesssim \eee^{-(m_1 - y_1)} \lesssim \eee^{-\frac 14(y_2 - y_1)}$. Thus
\begin{equation}
\begin{aligned}
\int_{m_1}^{m} |\partial_x R_1||f(R_1 + \sigma R_2) - f(R_1) - \sigma f(R_2) - \sigma f'(R_1)R_2|\ud x
\lesssim (m - m_1)\eee^{-\frac 54(y_2 - y_1)} \\
 \lesssim (y_2-y_1)\eee^{-\frac 54(y_2 - y_1)} \ll \eee^{-(y_2 - y_1)}.
\end{aligned}
\end{equation}

For $m \leq x \leq m_2$, using \eqref{eq:f-cross} we have
\begin{equation}
\label{eq:inter-asym-35}
|f(R_1 + \sigma R_2) - f(R_1)- \sigma f(R_2)| \lesssim R_1 R_2.
\end{equation}
Lemma~\ref{lem:Q-asym} yields $|\partial_x R_1| +R_1 \lesssim \eee^{-(m-y_1)} \lesssim \eee^{-\frac 12(y_2 - y_1)}$ and $R_2 \lesssim \eee^{-(y_2 - m_2)} \lesssim \eee^{-\frac 14(y_2 - y_1)}$,
thus $|\partial_x R_1|R_1 R_2 \lesssim \eee^{-\frac 54 (y_2 - y_1)}$.
Integrating between $m$ and $m_2$ yields \eqref{eq:inter-asym-33}.

Finally, for $x \geq m_2$, we use again \eqref{eq:inter-asym-35}.
From Lemma~\ref{lem:Q-asym} we have $|\partial_x R_1| +R_1 \lesssim \eee^{-(m_2-y_1)} \lesssim \eee^{-\frac 34(y_2 - y_1)}$, so $|\partial_x R_1|R_1 \lesssim \eee^{-\frac 32(y_2 - y_1)}$.
Since $\|R_2\|_{L^1} = \|Q\|_{L^1} < +\infty$, we obtain
\begin{equation}
\int_{m_2}^{+\infty} |\partial_x R_1|R_1R_2 \ud x \lesssim \eee^{-\frac 32 (y_2 - y_1)} \ll \eee^{-(y_2 - y_1)},
\end{equation}
and \eqref{eq:inter-asym-34} follows.

This finishes the proof of \eqref{eq:inter-asym-3}. It remains to compute
\begin{equation}
\label{eq:inter-asym-4}
\int_{-\infty}^m R_2f'(R_1)\partial_x R_1\ud x = \int_{-\infty}^m \partial_x\big(f(R_1)\big)R_2\ud x
= \int_{-\infty}^m \partial_x\big(R_1 - \partial_x^2 R_1\big)R_2\ud x,
\end{equation}
where in the last step we have used \eqref{eq:Qv}. Integrating by parts, we get
\begin{equation}
\begin{aligned}
\int_{-\infty}^m \partial_x\big(\partial_x^2 R_1\big)R_2\ud x
&= \partial_x^2 R_1(m)R_2(m)-\int_{-\infty}^m \partial_x^2R_1\partial_x R_2\ud x \\
&= \partial_x^2 R_1(m)R_2(m) - \partial_x R_1(m)\partial_x R_2(m) + \int_{-\infty}^m \partial_xR_1\partial_x^2 R_2\ud x.
\end{aligned}
\end{equation}
Thus
\begin{equation}
\int_{-\infty}^m \partial_x\big(R_1 - \partial_x^2 R_1\big)R_2\ud x
= \partial_x R_1(m)\partial_x R_2(m) - \partial_x^2 R_1(m)R_2(m) +
\int_{-\infty}^m f(R_2)\partial_x R_1\ud x,
\end{equation}
where in the last step we use $-\partial_x^2 R_2 + R_2 = f(R_2)$.
Dividing into $x \leq m_1$ and $x \geq m_1$, and using $|f(u)| \lesssim u^2$,
we see that $\Big|\int_{-\infty}^m f(R_2)\partial_x R_1\ud x\Big| \ll \eee^{-(y_2 - y_1)}$, hence the last term is negligible.
From Lemma~\ref{lem:Q-asym}, we have
\begin{align}
\partial_x R_1(m) &= {-}(k_0 + o(1))\eee^{-(m - y_1)} = {-}(k_0 + o(1))\eee^{-\frac{y_2 - y_1}{2}}, \\
\partial_x^2 R_1(m) &= (k_0 + o(1))\eee^{-(m - y_1)} = (k_0 + o(1))\eee^{-\frac{y_2 - y_1}{2}}, \\
R_2(m) &= (k_0 + o(1))\eee^{-(y_2 - m)} = (k_0 + o(1))\eee^{-\frac{y_2 - y_1}{2}}, \\
\partial_x R_2(m) &= (k_0 + o(1))\eee^{-(y_2 - m)} = (k_0 + o(1))\eee^{-\frac{y_2 - y_1}{2}}.
\end{align}
This yields
\begin{equation}
\partial_x R_1(m)\partial_x R_2(m) - \partial_x^2 R_1(m)R_2(m) = {-}(2k_0^2 + o(1))\eee^{-(y_2 - y_1)},
\end{equation}
which finishes the proof of \eqref{eq:inter-asym-1}.
\end{proof}

We are ready to compute $p_1'(t)$ and $p_2'(t)$, where $p_1(t)$ and $p_2(t)$
are defined by \eqref{eq:pk}.
\begin{lemma}
\label{lem:dtpk}
For any $c > 0$ there exists $\delta > 0$ such that if
\begin{equation}
\label{eq:norm-small-dtpk}
\eee^{-(y_2(t) - y_1(t))} + \|\varepsilon(t)\|_{H^1}^2 \leq \delta,
\end{equation}
then
\begin{align}
\big|p_1'(t) + \sigma 2k_0^2\eee^{-(y_2(t) - y_1(t))}\big| &\leq c\big(\eee^{-(y_2(t) - y_1(t))} + \|\varepsilon(t)\|_{H^1}^2\big), \label{eq:dtp1} \\
\big|p_2'(t) - \sigma 2k_0^2\eee^{-(y_2(t) - y_1(t))}\big| &\leq c\big(\eee^{-(y_2(t) - y_1(t))} + \|\varepsilon(t)\|_{H^1}^2\big), \label{eq:dtp2}
\end{align}
where $k_0$ is the constant from Lemma~\ref{lem:Q-asym}.
\end{lemma}
\begin{proof}
We will only prove \eqref{eq:dtp1}, because \eqref{eq:dtp2} is obtained analogously.
We will discard terms which are much smaller (as $\delta \to 0$) than
$\eee^{-(y_2 - y_1)}+\| \varepsilon\|_{H^1}^2$. In the sequel, we call such terms ``negligible''
and the sign $\simeq$ always means ``up to terms $\ll \eee^{-(y_2 - y_1)}+\| \varepsilon\|_{H^1}^2$''.

Without loss of generality we can assume that $ \varepsilon \in C^1(I, H^3)$,
where $I$ is some open interval containing $t$
(by a standard approximation procedure using local well-posedness of the equation).

We differentiate the first term of $p_1(t)$ using \eqref{eq:dteps}:
\begin{equation}
\begin{aligned}
\dd t\la R_1, \varepsilon \ra &= -y_1'\la \partial_x R_1, \varepsilon\ra 
+ y_1'\la R_1, \partial_x R_1\ra + \sigma y_2'\la R_1, \partial_x R_2\ra \\
&+ \big\la R_1, \partial_x\big({-}\partial_x^2 \varepsilon - f(R_1 + \sigma R_2 + \varepsilon)
+ f(R_1) + \sigma f(R_2)+ \varepsilon\big)\big\ra.
\end{aligned}
\end{equation}
Since $\la R_1, \partial_x R_1\ra = 0$ and $L_1(\partial_x R_1) = 0$, we obtain
\begin{equation}
\begin{aligned}
\dd t\la R_1, \varepsilon \ra &= -y_1'\la\partial_x R_1, \varepsilon\ra +\sigma y_2'\la R_1, \partial_x R_2\ra \\ &+ \la \partial_x R_1, f(R_1 + \sigma R_2 + \varepsilon) - f(R_1) - f'(R_1) \varepsilon-\sigma f(R_2)\ra.
\end{aligned}
\end{equation}
We claim that the second term of the right hand side is negligible.
This follows from Lemma~\ref{lem:Q-asym},
\eqref{eq:dtyk} and the elementary inequality:
\begin{equation}
\label{eq:prod-exp}
\int_\bR \eee^{-|x - y_1|}\eee^{-|x - y_2|}\ud x \lesssim (y_2 - y_1)\eee^{-(y_2 - y_1)},
\end{equation}
which can be obtained by computing the left hand side separately for $x \leq y_1$,
$y_1 \leq x \leq y_2$ and $x \geq y_2$. We obtain
\begin{equation}
\label{eq:dtcrochet}
\dd t\la R_1, \varepsilon \ra \simeq -y_1'\la\partial_x R_1, \varepsilon\ra + \la \partial_x R_1, f(R_1 + \sigma R_2 + \varepsilon) - f(R_1) - f'(R_1) \varepsilon-\sigma f(R_2)\ra.
\end{equation}

Now we compute the derivative of the correction term. We have
\begin{equation}
\label{eq:dtvir-1}
\frac 12\dd t\int_\bR \phi_1 \,\varepsilon^2 \ud x = \frac 12 \int_\bR \partial_t \phi_1 \,\varepsilon^2\ud x
+ \int_\bR \phi_1 \,\varepsilon\,\partial_t \varepsilon\ud x.
\end{equation}
For the definition of $\phi_1$, see \eqref{eq:phi1-def}. By the Chain Rule, we have
\begin{equation}
\partial_t \phi_1(t, x) = \frac{-y_1'(t)(y_2(t) - y_1(t)) - (x - y_1(t))(y_2'(t) - y_1'(t))}{(y_2(t) - y_1(t))^2}\psi'\Big(\frac{x - y_1(t)}{y_2(t) - y_1(t)}\Big).
\end{equation}
If $x \leq y_1(t)$ or $x \geq y_2(t)$, then the right hand side equals $0$.
If $y_1(t) \leq x \leq y_2(t)$, then we get
\begin{equation}
|\partial_t \phi_1(t, x)| \lesssim \frac{|y_1'(t)| + |y_2'(t)|}{y_2(t) - y_1(t)},
\end{equation}
so \eqref{eq:dtyk} yields
\begin{equation}
\label{eq:virial-1st-term}
\Big|\int_\bR \partial_t \phi_1 \,\varepsilon^2\ud x\Big| \ll \eee^{-(y_2 - y_1)} + \| \varepsilon\|_{H^1}^2.
\end{equation}
We now consider the second term of the right hand side in \eqref{eq:dtvir-1}.
Using \eqref{eq:dteps} we find
\begin{equation}
\label{eq:dtvir-2}
\begin{aligned}
\int_\bR \phi_1 \,\varepsilon\,\partial_t \varepsilon\ud x &=
y_1'\int_\bR \phi_1\,\varepsilon\,\partial_x R_1\ud x + \sigma y_2' \int_\bR \phi_1\,\varepsilon\,\partial_x R_2\ud x \\
&+\int\phi_1\,\varepsilon\,\partial_x\big({-}\partial_x^2 \varepsilon - f(R_1 + \sigma R_2 + \varepsilon)
+ f(R_1) + \sigma f(R_2)+ \varepsilon\big)\ud x \\
&= (I) + (II) + (III).
\end{aligned}
\end{equation}
Recall that $\phi_1(t, x) = 0$ for $x \geq \frac{y_1(t) + 2y_2(t)}{3}$,
whereas for $x \leq \frac{y_1(t) + 2y_2(t)}{3}$ Lemma~\ref{lem:Q-asym} yields
$|\partial_x R_2| \lesssim \eee^{-(y_2 - x)}$, so
\begin{equation}
\int_\bR \phi_1^2 (\partial_x R_2)^2\ud x \lesssim \int_{-\infty}^{\frac{y_1+2y_2}{3}}\eee^{-2(y_2 - x)}\ud x \lesssim \eee^{-\frac 23(y_2 - y_1)}.
\end{equation}
Applying \eqref{eq:dtyk} and the Cauchy-Schwarz inequality, we obtain the bound
\begin{equation}
\label{eq:virial-II}
|(II)| \lesssim \big(\| \varepsilon\|_{H^1} + \eee^{-(y_2 - y_1)}\big)\eee^{-\frac 13(y_2 - y_1)}\| \varepsilon\|_{L^2} \ll \eee^{-(y_2 - y_1)}+\| \varepsilon\|_{H^1}^2.
\end{equation}
In the same way, one can show that
\begin{equation}
\Big|y_1'\int_\bR \phi_2\,\varepsilon\,\partial_x R_1\ud x \Big| \ll \eee^{-(y_2 - y_1)}+\| \varepsilon\|_{H^1}^2
\end{equation}
or, equivalently,
\begin{equation}
\label{eq:virial-I}
\big| (I) - y_1'\la \partial_x R_1, \varepsilon\ra\big| \ll \eee^{-(y_2 - y_1)}+\|\varepsilon\|_{H^1}^2.
\end{equation}
This implies that $(I)$ cancels
with the first term of the right hand side in \eqref{eq:dtcrochet}.

Finally, we consider $(III)$. We claim that
\begin{align}
\Big|\int_\bR \phi_1\,\varepsilon\,\partial_x \varepsilon\ud x\Big| &\ll \eee^{-(y_2 - y_1)} + \| \varepsilon\|_{H^1}^2, \label{eq:virial-III-1} \\
\Big|\int_\bR \phi_1\,\varepsilon\,\partial_x^3 \varepsilon\ud x\Big| &\ll \eee^{-(y_2 - y_1)} + \| \varepsilon\|_{H^1}^2, \label{eq:virial-III-2} \\
\Big|\int_\bR\phi_1\,\varepsilon\,\partial_x\big(f(R_1 + \sigma R_2) - f(R_1) - \sigma f(R_2)\big) \ud x \Big| &\ll \eee^{-(y_2 - y_1)} + \| \varepsilon\|_{H^1}^2, \label{eq:virial-III-3} 
\end{align}
which will imply
\begin{equation}
\label{eq:virial-III-4}
\Big|(III) + \int_\bR \phi_1\,\varepsilon\,\partial_x \big(f(R_1 + \sigma R_2 + \varepsilon) - f(R_1+\sigma R_2)\big)\ud x\Big| \ll \eee^{-(y_2 - y_1)} + \| \varepsilon\|_{H^1}^2.
\end{equation}
Let us assume \eqref{eq:virial-III-4} and finish the proof.
From \eqref{eq:f-Lip} we have $|f(R_1 + \sigma R_2 + \varepsilon) - f(R_1+\sigma R_2)| \lesssim | \varepsilon|$. Since $\|\partial_x\phi_1\|_{L^\infty} \lesssim \frac{1}{y_2 - y_1} \ll 1$, we have
\begin{equation}
\Big|\int_\bR \partial_x\phi_1\,\varepsilon\,\big(f(R_1 + \sigma R_2 + \varepsilon) - f(R_1+\sigma R_2)\big)\ud x \Big| \ll \| \varepsilon\|_{L^2}^2.
\end{equation}
Bound \eqref{eq:virial-III-4} and integration by parts yield
\begin{equation}
\label{eq:virial-III-5}
\Big|(III) - \int_\bR \phi_1\,\partial_x\varepsilon\,\big(f(R_1 + \sigma R_2 + \varepsilon) - f(R_1+\sigma R_2)\big)\ud x\Big| \ll \eee^{-(y_2 - y_1)} + \| \varepsilon\|_{H^1}^2.
\end{equation}
We apply Lemma~\ref{lem:virial} with $R = R_1 + \sigma R_2$ and $\phi = \phi_1(t, \cdot)$. We obtain
\begin{equation}
\label{eq:virial-III-6}
\begin{aligned}
\Big|(III) + \int_\bR \phi_1\,\partial_x(R_1 + \sigma R_2)\,\big(f(R_1 + \sigma R_2 + \varepsilon) - f(R_1+\sigma R_2) - f'(R_1 + \sigma R_2)\varepsilon \big)\ud x\Big|& \\
 \ll \eee^{-(y_2 - y_1)} + \| \varepsilon\|_{H^1}^2.&
\end{aligned}
\end{equation}
From \eqref{eq:f-1stTay} we have $|f(R_1 + \sigma R_2 + \varepsilon) - f(R_1+\sigma R_2) - f'(R_1 + \sigma R_2)\varepsilon| \lesssim \varepsilon^2$.
From the proof of \eqref{eq:virial-II} we see that $\|\phi_1\,\partial_x R_2\|_{L^\infty} \lesssim \eee^{-\frac 13(y_2 - y_1)} \ll 1$, hence in the integral above we can replace
$\phi_1\,\partial_x (R_1 + \sigma R_2)$ by $\phi_1\,\partial_x R_1$.
Similarly, we have the bound $\|(1 - \phi_1)\,\partial_x R_1\|_{L^\infty}
= \|\phi_2\,\partial_x R_1\|_{L^\infty} \lesssim \eee^{-\frac 13(y_2 - y_1)} \ll 1$,
which allows to replace $\phi_1\,\partial_x R_1$ by $\partial_x R_1$.
After all these operations we get
\begin{equation}
\label{eq:virial-III-7}
\begin{aligned}
\big|(III) + \la\partial_x R_1, f(R_1 + \sigma R_2 + \varepsilon) - f(R_1+\sigma R_2) - f'(R_1 + \sigma R_2)\varepsilon \ra\big| \\
 \ll \eee^{-(y_2 - y_1)} + \| \varepsilon\|_{H^1}^2.
\end{aligned}
\end{equation}
When we combine \eqref{eq:dtvir-1}, \eqref{eq:virial-1st-term}, \eqref{eq:dtvir-2}, \eqref{eq:virial-II}, \eqref{eq:virial-I}
and \eqref{eq:virial-III-7}, we find
\begin{equation}
\label{eq:dtvir-3}
\begin{aligned}
\frac 12\dd t\int_\bR \phi_1 \varepsilon^2\ud x &\simeq y_1'\la \partial_x R_1, \varepsilon\ra
- \la\partial_x R_1, f(R_1 + \sigma R_2 + \varepsilon) \\ &- f(R_1+\sigma R_2) - f'(R_1 + \sigma R_2)\varepsilon \ra.
\end{aligned}
\end{equation}
Together with \eqref{eq:dtcrochet} and the definition of $p_1$, this yields
\begin{equation}
\label{eq:dtp1-1}
p_1' \simeq \big\la \partial_x R_1, f(R_1 + \sigma R_2) - f(R_1) - \sigma f(R_2) + \big(f'(R_1 + \sigma R_2)-f'(R_1)\big)\varepsilon\big\ra.
\end{equation}
Subtracting \eqref{eq:inter-asym-1}, we get
\begin{equation}
\label{eq:dtp1-2}
p_1' + \sigma 2k_0^2\eee^{-(y_2 - y_1)} \simeq \big\la \partial_x R_1, \big(f'(R_1 + \sigma R_2)-f'(R_1)\big)\varepsilon\big\ra.
\end{equation}
Since $f'$ is locally Lipschitz, we have $|f'(R_1 + \sigma R_2)-f'(R_1)|
\lesssim R_2 \lesssim \eee^{-|\cdot - y_2|}$. We also have $|\partial_x R_1| \lesssim \eee^{-|\cdot - y_1|}$. Thus, by the Cauchy-Schwarz inequality,
\begin{equation}
\begin{aligned}
\big\la \partial_x R_1, \big(f'(R_1 + \sigma R_2)-f'(R_1)\big)\varepsilon\big\ra^2
&\lesssim \int_\bR \eee^{-2|x - y_1|}\eee^{-2|x - y_2|}\ud x\int_\bR \varepsilon^2\ud x \\
&\ll \big(\eee^{-(y_2 - y_1)} + \| \varepsilon\|_{H^1}^2\big)^2,
\end{aligned}
\end{equation}
where in the last step we use the fact that
\begin{equation}
\label{eq:prod-exp-2}
\int_\bR \eee^{-2|x - y_1|}\eee^{-2|x - y_2|}\ud x \lesssim (y_2 - y_1)\eee^{-2(y_2 - y_1)} \ll \eee^{-(y_2 - y_1)},
\end{equation}
see \eqref{eq:prod-exp}.
This finishes the proof of \eqref{eq:dtp1}, provided that we can show that
\eqref{eq:virial-III-1}, \eqref{eq:virial-III-2} and \eqref{eq:virial-III-3} hold.

Bound \eqref{eq:virial-III-1} follows from
\begin{equation}
\int_\bR \phi_1\,\varepsilon\,\partial_x \varepsilon\ud x =
\frac 12 \int_\bR \phi_1\,\partial_x( \varepsilon^2)\ud x =
-\frac 12 \int_\bR \partial_x \phi_1\, \varepsilon^2 \ud x,
\end{equation}
because $\|\partial_x \phi_1\|_{L^\infty} \lesssim (y_2 - y_1)^{-1} \ll 1$.
The proof of \eqref{eq:virial-III-2} is similar, but we need to integrate by parts
many times:
\begin{equation}
\begin{aligned}
\int_\bR \phi_1\,\varepsilon\,\partial_x^3 \varepsilon\ud x &=
-\int_\bR \partial_x\phi_1\, \varepsilon\,\partial_x^2 \varepsilon\ud x
- \int_\bR \phi_1\,\partial_x \varepsilon\,\partial_x^2 \varepsilon\ud x \\
&= \int_\bR \partial_x^2 \phi_1\, \varepsilon\,\partial_x \varepsilon\ud x
+ \int_\bR \partial_x \phi_1 (\partial_x \varepsilon)^2 \ud x
-\frac 12 \int_\bR \phi_1 \partial_x \big((\partial_x \varepsilon)^2\big)\ud x \\
&= -\frac 12 \int_\bR \partial_x^3 \phi_1\, \varepsilon^2 \ud x
+ \frac 32 \int_\bR\partial_x \phi_1\,(\partial_x \varepsilon)^2\ud x,
\end{aligned}
\end{equation}
and we see that both terms are negligible.

In order to prove \eqref{eq:virial-III-3}, it suffices to check that
\begin{equation}
\label{eq:virial-final}
\|f(R_1 + \sigma R_2) - f(R_1) - \sigma f(R_2)\|_{L^2} \ll \eee^{-\frac 12(y_2 - y_1)}
\end{equation}
and integrate by parts.
From \eqref{eq:f-cross} we have
\begin{equation}
|f(R_1 + \sigma R_2) - f(R_1) - \sigma f(R_2)| \lesssim R_1 R_2 \lesssim \eee^{-|\cdot - y_1|}\eee^{-|\cdot - y_2|},
\end{equation}
and \eqref{eq:virial-final} follows from \eqref{eq:prod-exp}.
\end{proof}
\subsection{Stable and unstable directions}
\label{ssec:unstable}
We also need to control the linear stable and unstable directions. We define
\begin{equation}
\label{eq:defal12}
\alpha_k^-(t, x) := \alpha^-(x - y_k(t)), \qquad \alpha_k^+(t, x) := \alpha^+(x - y_k(t)),
\qquad k \in \{1, 2\}
\end{equation}
(see Proposition~\ref{prop:alpha} for the definition of $ \alpha^-$ and $ \alpha^+$) and
\begin{equation}
\label{eq:defapm}
a_k^-(t) := \la \alpha_k^-(t), \varepsilon(t)\ra, \qquad
a_k^+(t) := \la \alpha_k^+(t), \varepsilon(t)\ra, \qquad k \in \{1, 2\}.
\end{equation}

\begin{lemma}
\label{lem:mod-stable}
For any $c > 0$ there exists $\delta > 0$ such that if
\begin{equation}
\label{eq:norm-small-stable}
\eee^{-(y_2(t) - y_1(t))} + \|\varepsilon(t)\|_{H^1}^2 \leq \delta,
\end{equation}
then
\begin{align}
\Big|\dd t a_k^-(t) + \nu a_k^-(t)\Big| &\leq c\sqrt{\eee^{-(y_2(t) - y_1(t))} + \|\varepsilon(t)\|_{H^1}^2}, \qquad k \in \{1, 2\}, \label{eq:dtakm} \\
\Big|\dd t a_k^+(t) - \nu a_k^+(t)\Big| &\leq c\sqrt{\eee^{-(y_2(t) - y_1(t))} + \|\varepsilon(t)\|_{H^1}^2}, \qquad k \in \{1, 2\}, \label{eq:dtakp}
\end{align}
where $\nu$ is defined in Proposition~\ref{prop:eigen}.
\end{lemma}
\begin{proof}
We will only prove \eqref{eq:dtakm} for $k = 1$, because the computation for \eqref{eq:dtakp}
or for $k = 2$ is almost the same. In this proof, we say that a real number is ``negligible''
if its absolute value is much smaller than $\sqrt{\eee^{-(y_2(t) - y_1(t))} + \|\varepsilon(t)\|_{H^1}^2}$ for $\delta$ sufficiently small.
We have
\begin{equation}
\label{eq:dta1m-1}
\dd t a_1^- = {-}y_1'\la \partial_x \alpha_1^-, \varepsilon\ra + \la \alpha_1^-, \partial_t \varepsilon\ra.
\end{equation}
Bound \eqref{eq:dtyk} implies that the first term is negligible and we can forget about it.
Like in the proof of Lemma~\ref{lem:mod},
we compute the second term using \eqref{eq:dteps}:
\begin{equation}
\la \alpha_1^-, \partial_t \varepsilon\ra =
\big\la \alpha_1^-, y_1'\partial_x R_1 + y_2' \partial_x R_2
+\partial_x\big({-}\partial_x^2 \varepsilon - f(R_1 + \sigma R_2 + \varepsilon)
+ f(R_1) + \sigma f(R_2)+ \varepsilon\big)\big\ra.
\end{equation}
We have $\la \alpha_1^-, \partial_x R_1\ra = \la \alpha^-, \partial_xQ\ra = 0$.
Moreover, the exponential decay of $\alpha^-$ and \eqref{eq:dtyk}
yield $|y_2'\la \alpha_1^-, \partial_x R_2\ra| = |y_2'|\,|\la \alpha^-, \partial_xQ(\cdot - (y_2 - y_1))\ra|
\ll \sqrt{\eee^{-(y_2(t) - y_1(t))} + \|\varepsilon(t)\|_{H^1}^2}$.
Hence, up to negligible terms, we have
\begin{equation}
\begin{aligned}
\dd t a_1^- &\simeq \big\la \alpha_1^-, \partial_x\big({-}\partial_x^2 \varepsilon - f(R_1 + \sigma R_2 + \varepsilon)
+ f(R_1) + \sigma f(R_2)+ \varepsilon\big)\big\ra \\
& \simeq {-}\big\la \partial_x \alpha_1^-, L_1 \varepsilon - f(R_1 + \sigma R_2 + \varepsilon)
+ f(R_1) + f'(R_1) \varepsilon + \sigma f(R_2)\big\ra,
\end{aligned}
\end{equation}
where $L_1 := -\partial_x^2 - f'(R_1) + 1$. We have $L_1(\partial_x \alpha_1^-) = \nu \alpha_1^-$,
see \eqref{eq:Ldxal}. Thus
\begin{equation}
-\la \partial_x \alpha_1^-, L_1\varepsilon\ra = -\la L_1(\partial_x \alpha_1^-), \varepsilon\ra
= -\nu\la \alpha_1^-, \varepsilon\ra = -\nu a_1^-,
\end{equation}
and we only need to check that
\begin{equation}
\label{eq:dta1m-2}
\big|\big\la \partial_x \alpha_1^-, f(R_1 + \sigma R_2 + \varepsilon)
- f(R_1) - f'(R_1) \varepsilon - \sigma f(R_2)\big\ra\big| \leq c \sqrt{\eee^{-(y_2(t) - y_1(t))} + \|\varepsilon(t)\|_{H^1}^2}.
\end{equation}
The triangle inequality and Lemma~\ref{lem:Taylor} yield
\begin{equation}
\begin{aligned}
&|f(R_1 + \sigma R_2 + \varepsilon)
- f(R_1) - f'(R_1) \varepsilon - \sigma f(R_2)| \\ 
&\quad\leq |f(R_1 + \sigma R_2) - f(R_1) - \sigma f(R_2)| \\
&\quad+ |f(R_1 + \sigma R_2 + \varepsilon) - f(R_1 + \sigma R_2) - f'(R_1 + \sigma R_2) \varepsilon| \\
&\quad+ |(f'(R_1 + \sigma R_2) - f'(R_1)) \varepsilon| \\
&\quad\lesssim R_1 R_2 + \varepsilon^2 + R_2| \varepsilon|,
\end{aligned}
\end{equation}
so we obtain, by the Cauchy-Schwarz inequality,
\begin{equation}
\begin{aligned}
\big|\big\la \partial_x \alpha_1^-, f(R_1 + \sigma R_2 + \varepsilon)
- f(R_1) - f'(R_1) \varepsilon - \sigma f(R_2)\big\ra\big| \\
\lesssim \|R_1 R_2\|_{L^2} + \| \varepsilon\|_{L^2}^2 + \| \alpha_1^- R_2\|_{L^2} \| \varepsilon\|_{L^2}
\end{aligned}
\end{equation}
(we have used the fact that $\alpha^- \in L^2 \cap L^\infty$).
The first term is negligible, see the proof of \eqref{eq:virial-final}.
The second term is clearly negligible.
The third term is negligible because
$$\|\alpha_1^- R_2\|_{L^2} = \|\alpha^- Q(\cdot - (y_2 - y_1))\|_{L^2} \leq c$$
if $y_2 - y_1$ is large enough (both $\alpha^-$ and $Q$ are exponentially decaying functions).
\end{proof}

\section{Dynamics of the reduced system}
\label{sec:dynamics}
In this section we complete the proof of Theorem~\ref{thm:classif}. We always assume that
$u: [T_0, \infty) \to H^1$ is a solution of \eqref{eq:kdv}
satisfying \eqref{eq:u-conv} and \eqref{eq:q-conv} with $v^\infty = 1$ (for the sake of simplicity).
Thus, for $t \geq T_0$, we have well-defined modulation parameters $y_k(t)$
and the error term $\varepsilon(t)$ such that $\lim_{t \to \infty} y_2(t) - y_1(t) = \infty$
and $\lim_{t \to \infty} \| \varepsilon(t)\|_{H^1} = 0$.

Given $\rho \geq \rho_0 \gg 1$, we define $q_k(t)$ by \eqref{eq:qk-def}.
Then $\|\varepsilon(t)\|_{H^1} \to 0$ implies
\begin{equation}
\lim_{t \to \infty}\big((q_2(t) - q_1(t)) - (y_2(t) - y_1(t))\big) = 0.
\end{equation}
In particular, the estimates from Section~\ref{sec:modulation-two} remain true
if $\eee^{-(y_2(t) - y_1(t))}$ is replaced by $\eee^{-(q_2(t) - q_1(t))}$.
\begin{proposition}
\label{prop:eps-bound}
The sign $\sigma$ equals $1$ and there exist $C_0 > 0$ (independent of $\rho$)
and $t_0 \geq T_0$ (which might depend on $\rho$) such that for all $t \geq t_0$
\begin{equation}
\label{eq:eps-bound}
\| \varepsilon(t)\|_{H^1}^2 \leq C_0 \sup_{\tau \geq t}\eee^{-(q_2(\tau) - q_1(\tau))}.
\end{equation}
\end{proposition}
Eventually, we will prove that
$q_2 - q_1$ is an increasing function, so $\sup$ in \eqref{eq:eps-bound}
is not really necessary.
We need two lemmas.
\begin{lemma}
\label{lem:a+small}
For any $c > 0$ and $t_0 \geq T_0$ there exists $t_1 \geq t_0$ such that
\begin{equation}
\label{eq:a+small}
a_1^-(t_1)^2 + a_2^-(t_1)^2 \leq c\big(\eee^{-(q_2(t_1) - q_1(t_1))} + \| \varepsilon(t_1)\|_{H^1}^2\big).
\end{equation}
\end{lemma}
\begin{proof}
Suppose the conclusion is false. Then for all $t \geq t_0$ we have
\begin{equation}
\label{eq:a+small-1}
N_1(t) := a_1^-(t)^2 + a_2^-(t)^2 \geq c\big(\eee^{-(q_2(t) - q_1(t))} + \| \varepsilon(t)\|_{H^1}^2\big).
\end{equation}
By Lemma~\ref{lem:mod-stable}, if we take $T_0$ large enough, then
\begin{equation}
\big|N_1'(t) + 2\nu N_1(t)\big| \leq c\nu\big(\eee^{-(q_2(t) - q_1(t))} + \| \varepsilon(t)\|_{H^1}^2\big) \leq \nu N_1(t),
\end{equation}
where in the last step we use \eqref{eq:a+small-1}. In particular, $N_1'(t) \leq -\nu N_1(t)$
for all $t \geq t_0$, which implies
\begin{equation}
N_1(t) \leq \eee^{-\nu(t - t_0)}N_1(t_0),\qquad \text{for all }t \geq t_0.
\end{equation}
Applying again \eqref{eq:a+small-1}, we deduce that $q_2(t) - q_1(t) \gtrsim t$ as $t \to \infty$,
which is impossible because $|q_1'(t)| + |q_2'(t)| \to 0$ as $t \to \infty$.
\end{proof}
\begin{lemma}
\label{lem:eps-bound-weak}
There exists $C_0 > 0$ (independent of $\rho$) with the following property.
For any $c_0 > 0$ there is $t_0 \geq T_0$ such that for all $t \geq t_0$
\begin{align}
\label{eq:a-small}
a_1^+(t)^2 + a_2^+(t)^2 &\leq c_0\,\sup_{\tau \geq t}\big(\eee^{-(q_2(\tau) - q_1(\tau))}
+ a_1^-(\tau)^2 + a_2^-(\tau)^2\big), \\
\label{eq:eps-bound-weak}
\| \varepsilon(t)\|_{H^1}^2 &\leq C_0\,\sup_{\tau \geq t}\big(\eee^{-(q_2(\tau) - q_1(\tau))}
+ a_1^-(\tau)^2 + a_2^-(\tau)^2\big).
\end{align}
\end{lemma}
\begin{proof}
Let $t \geq t_0$ and let $t_1 \geq t$ be such that
\begin{equation}
\label{eq:eps-bound-weak-3}
\eee^{-(q_2(t_1) - q_1(t_1))} + \| \varepsilon(t_1)\|_{H^1}^2 = \sup_{\tau \geq t}\big(
\eee^{-(q_2(\tau) - q_1(\tau))} + \| \varepsilon(\tau)\|_{H^1}^2\big).
\end{equation}
We first prove that for any $c > 0$, if $t_0$ is chosen large enough, then
\begin{equation}
\label{eq:a-small-1}
a_1^+(t_1)^2 + a_2^+(t_1)^2 \leq c\big(\eee^{-(q_2(t_1) - q_1(t_1))} + \| \varepsilon(t_1)\|_{H^1}^2\big).
\end{equation}
For $t \geq T_0$, denote $N_2(t) := a_1^+(t)^2 + a_2^+(t)^2$.
Suppose that \eqref{eq:a-small-1} does not hold and let $t_2 := \max\{\tau: N_2(\tau) \geq N_2(t_1)\}$.
Note that $t_2 \in [t_1, \infty)$, because $\lim_{\tau \to \infty} N_2(\tau) = 0$
and $N_2(t_1) > 0$.
We have $N_2'(t_2) \leq 0$ and
\begin{equation}
N_2(t_2) \geq N_2(t_1) \geq c\big(\eee^{-(q_2(t_1) - q_1(t_1))} + \| \varepsilon(t_1)\|_{H^1}^2\big) \geq c\big(\eee^{-(q_2(t_2) - q_1(t_2))} + \| \varepsilon(t_2)\|_{H^1}^2\big),
\end{equation}
where the last inequality follows from \eqref{eq:eps-bound-weak-3}.
This implies
\begin{equation}
-N_2'(t_2) + 2\nu N_2(t_2) \geq 2\nu c\big(\eee^{-(q_2(t_2) - q_1(t_2))} + \| \varepsilon(t_2)\|_{H^1}^2\big),
\end{equation}
which contradicts Lemma~\ref{lem:mod-stable} if $t_0$ is large enough.

From Proposition~\ref{prop:coer} we know that
\begin{equation}
\eee^{-(q_2(t_1) - q_1(t_1))} + \| \varepsilon(t_1)\|_{H^1}^2 \leq \frac{C_0}{2} \big(
\eee^{-(q_2(t_1) - q_1(t_1))}+ a_1^+(t_1)^2 + a_2^+(t_1)^2+a_1^-(t_1)^2 + a_2^-(t_1)^2\big).
\end{equation}
Setting $c = \frac{1}{C_0}$ in \eqref{eq:a-small-1}, we obtain
\begin{gather}
\label{eq:epsilon-t1-est}
\eee^{-(q_2(t_1) - q_1(t_1))} + \| \varepsilon(t_1)\|_{H^1}^2 \leq C_0 \big(
\eee^{-(q_2(t_1) - q_1(t_1))}+ a_1^-(t_1)^2 + a_2^-(t_1)^2\big),
\end{gather}
which implies \eqref{eq:eps-bound-weak}.

We prove \eqref{eq:a-small} by contradiction.
%
Set
\begin{equation}
t_3 := \sup\big\{\tau > t: N_2(\tau) > c_0 \big(
\eee^{-(q_2(t_1) - q_1(t_1))}+ a_1^-(t_1)^2 + a_2^-(t_1)^2\big)\big\}.
\end{equation}
Since $N_2(\tau) \to 0$ as $\tau \to \infty$, $t_3$ is well-defined and $t_3 \in (t, \infty)$.
Moreover, we would have $N_2'(t_3) \leq 0$ and,
using  the definition of $t_1$ and \eqref{eq:epsilon-t1-est},
\begin{equation}
\begin{aligned}
\eee^{-(q_2(t_3) - q_1(t_3))} + \|\varepsilon(t_3)\|_{H^1}^2 &\leq \eee^{-(q_2(t_1) - q_1(t_1))} + \|\varepsilon(t_1)\|_{H^1}^2  \\
&\leq C_0 \big(
\eee^{-(q_2(t_1) - q_1(t_1))}+ a_1^-(t_1)^2 + a_2^-(t_1)^2\big).
\end{aligned}
\end{equation}
Therefore,
\begin{equation}
\begin{aligned}
N_2(t_3) = c_0 \big(\eee^{-(q_2(t_1) - q_1(t_1))}+ a_1^-(t_1)^2 + a_2^-(t_1)^2\big)
\geq \frac{c_0}{C_0}\big(\eee^{-(q_2(t_3) - q_1(t_3))}+\|\varepsilon(t_3)\|_{H^1}^2\big).
\end{aligned}
\end{equation}
Hence, we get
\begin{equation}
-N_2'(t_3) + 2\nu N_2(t_3) \geq 2\nu \frac{c_0}{C_0}\big(\eee^{-(q_2(t_3) - q_1(t_3))} + \| \varepsilon(t_3)\|_{H^1}^2\big),
\end{equation}
which contradicts Lemma~\ref{lem:mod-stable} if $t_0$ is large enough.
\end{proof}
\begin{proof}[Proof of Proposition~\ref{prop:eps-bound}]
Let $c_0 > 0$. In Lemma~\ref{lem:a+small}, let $c = \frac{c_0}{2(C_0 + 1)}$, where $C_0$
is the constant from Lemma~\ref{lem:eps-bound-weak}. We obtain that there exists $t_1$
arbitrarily large such that
\begin{equation}
\label{eq:eps-bound-1}
N_1(t_1) \leq \frac{c_0}{2(C_0+1)}\big((C_0 +1)\sup_{\tau \geq t_1}\eee^{-(q_2(\tau) - q_1(\tau))} + C_0 \sup_{\tau \geq t_1}N_1(\tau)\big)
\end{equation}
(the meaning of $N_1(t)$ is the same as in the proof of Lemma~\ref{lem:a+small}).
Let
\begin{equation}
\label{eq:wtN3-def}
N_3(t) := \eee^{-(q_2(t) - q_1(t))}, \qquad \wt N_3(t) := \sup_{\tau \geq t}N_3(\tau) = \sup_{\tau \geq t}\eee^{-(q_2(\tau) - q_1(\tau))}.
\end{equation}
We will show that for all $t \geq t_1$ we have
\begin{equation}
\label{eq:N1-bound}
\wt N_1(t) := \sup_{\tau \geq t} N_1(\tau) \leq c_0\wt N_3(t).
\end{equation}
In view of Lemma~\ref{lem:eps-bound-weak}, this will finish the proof of \eqref{eq:eps-bound}.

By the rising sun lemma, see \cite[Lemma 1.6.17]{Tao-Measure},
for all $t$ except for a countable set, the function $\wt N_3$ is either constant, or equal to $N_3$ in a neighborhood of $t$.
In particular, for all $t$ except a countable set, $\wt N_3$ is differentiable and
\begin{equation}
\label{eq:derivN3}
\begin{aligned}
|\wt N_3'(t)| &\leq (|q_1'(t)| + |q_2'(t)|)\eee^{-(q_2(t) - q_1(t))} \\
&\ll \eee^{-(q_2(t) - q_1(t))} = N_3(t) \leq \wt N_3(t)\qquad \text{as }t \to \infty.
\end{aligned}
\end{equation}
We claim that if $t$ is sufficiently large and $N_1(t) = \wt N_1(t)$, then
\begin{equation}
\label{eq:eps-bound-5}
N_1(t) \geq c_0\wt N_3(t) \quad \Rightarrow \quad N_1'(t) \leq -\nu N_1(t).
\end{equation}
Indeed, Lemma~\ref{lem:eps-bound-weak} yields
\begin{equation}
\| \varepsilon(t)\|_{H^1}^2 \leq C_0\big(\wt N_1(t) + \wt N_3(t)\big) \leq C_0(1 + c_0^{-1})N_1(t).
\end{equation}
By Lemma~\ref{lem:mod-stable}, for any $c > 0$ and $t$ sufficiently large we have
\begin{equation}
|N_1'(t) + 2\nu N_1(t)| \leq c\big(\| \varepsilon(t)\|_{H^1}^2 + N_3(t)\big)
\leq c(C_0(1 + c_0^{-1}) + c_0^{-1}) N_1(t),
\end{equation}
and it suffices to take $c \leq \frac{\nu}{C_0(1+c_0^{-1}) + c_0^{-1}}$.

Suppose that \eqref{eq:N1-bound} does not hold, and let $t_2 > t_1$ be such that
$\wt N_1(t_2) > c_0\wt N_3(t_2)$. Without loss of generality,
we can assume that $\wt N_1(t_2) = N_1(t_2)$ (it suffices to replace $t_2$
by $\sup\big\{\tau \geq t_2: N_1(\tau) = N_1(t_2)\big\}$).
Let
\begin{equation}
t_3 := \min \Big\{t \in [t_1, t_2]: N_1'(\tau) \leq -\frac{\nu}{2}N_1(\tau)\text{ for all }\tau \in [t, t_2]\Big\}.
\end{equation}
By \eqref{eq:eps-bound-5} and continuity, $t_3 < t_2$. Suppose that $t_3 > t_1$.
By \eqref{eq:derivN3}, we can assume that $\wt N_3'(t) \geq -\frac{\nu}{4} \wt N_3(t)$
for almost all $t \in [t_3, t_2]$ (provided that $t_1$ was chosen sufficiently large).
Since $N_1(t_2) > c_0\wt N_3(t_2)$, this implies $N_1(t_3) > c_0\wt N_3(t_3)$.
The function $N_1(t)$ is decreasing for $t\in [t_3, t_2]$, so $N_1(t_3) = \wt N_1(t_3)$.
Thus \eqref{eq:eps-bound-5} yields $N_1'(t_3) \leq -\nu N_1(t_3)$,
which is in contradiction with the definition of $t_3$. This proves that $t_3 = t_1$.

In particular, we have shown that $N_1(t_1) = \wt N_1(t_1)$ and $N_1(t_1) > c_0\wt N_3(t_1)$.
This contradicts \eqref{eq:eps-bound-1}, so \eqref{eq:N1-bound} has to hold.

It remains to prove that $\sigma = 1$.
Suppose that $\sigma = -1$. Proposition~\ref{prop:coer} yields
\begin{equation}
\label{eq:coer-666}
\eee^{-(q_2(t) - q_1(t))} \lesssim
a_1^+(t)^2 + a_2^+(t)^2 + a_1^-(t)^2 + a_2^-(t)^2, \qquad \text{for all }t \geq T_0.
\end{equation}
Take $t_1$ sufficiently large such that
\begin{equation}
\label{eq:eps-bound-weak-666}
\eee^{-(q_2(t_1) - q_1(t_1))}= \sup_{\tau \geq t_1}\big(
\eee^{-(q_2(\tau) - q_1(\tau))}\big).
\end{equation}
Then \eqref{eq:a-small} and \eqref{eq:coer-666} yield
\begin{equation}
\eee^{-(q_2(t_1) - q_1(t_1))} \lesssim
a_1^-(t_1)^2 + a_2^-(t_1)^2,
\end{equation}
which contradicts \eqref{eq:N1-bound} if $c_0$ is small enough.
\end{proof}
\begin{lemma}
\label{lem:q-mono}
There exists $t_0 \geq T_0$ such that $q_2(t) - q_1(t)$
is increasing for $t \geq t_0$.
\end{lemma}
\begin{proof}
Set $q(t) := q_2(t) - q_1(t)$. Let $t_1 \geq t_0$, where $t_0$ is large
(chosen later in the proof).
We need to show that for all $t > t_1$ we have $q(t) > q(t_1)$.
Suppose this is not the case, and let
\begin{equation}
t_2 := \sup\big\{t: q(t) = \inf_{\tau \geq t_1}q(\tau)\big\}.
\end{equation}
Then $t_2 > t_1$, $q(t_2) = \inf_{\tau \leq t_2}q(\tau)$ and $q'(t_2) = 0$.

Let $p(t) := p_2(t) - p_1(t)$, $q_0 := q(t_2)$, $t_3 := \inf\{t \geq t_2: q(t) = q_0 + 1\}$.
Since $\lim_{t \to \infty}q(t) = +\infty$, $t_3$ is finite.
We will show that the modulation equations imply
\begin{equation}
\label{eq:q-mono-0}
z(t_3) \leq q_0 + \frac 12,
\end{equation}
which is a contradiction.

Let $t \in [t_2, t_3]$.
By Proposition~\ref{prop:eps-bound} we have $\| \varepsilon(t)\|_{H^1}^2 \lesssim \eee^{-q_0}$,
thus \eqref{eq:dtp1} and \eqref{eq:dtp2} yield, for some $C > 0$,
\begin{equation}
\label{eq:q-mono-1}
\begin{aligned}
p'(t) &\geq (4k_0^2 - 2c) \eee^{-q(t)} - 2cC\eee^{-q_0} \\
&\geq (4k_0^2-2c) \eee^{-(q_0+1)} - 2cC\eee^{-q_0} \geq k_0^2\eee^{-q_0},
\qquad \text{for all }t \in [t_2, t_3].
\end{aligned}
\end{equation}
Since $q'(t_2) = 0$, \eqref{eq:dtqk} yields $p(t_2) \geq -c\,\eee^{-\frac{q_0}{2}}$.
Integrating \eqref{eq:q-mono-1}, we get
\begin{equation}
\label{eq:q-mono-2}
p(t) \geq -c\,\eee^{-\frac{q_0}{2}} + k_0^2(t - t_2)\eee^{-q_0}, \qquad \text{for all }t\in [t_2, t_3].
\end{equation}
Using \eqref{eq:dtqk} again we obtain
\begin{equation}
-\la Q, \wt Q\ra q'(t) \leq -k_0^2(t - t_2)\eee^{-q_0} + c\,\eee^{-\frac{q_0}{2}}, \qquad\text{for all }t\in [t_2, t_3].
\end{equation}
We now integrate for $t$ between $t_2$ and $t_3$:
\begin{equation}
\begin{aligned}
{-}\la Q, \wt Q\ra\big(q(t_3) - q(t_2)\big)
&\leq  \int_{t_2}^{t_3}\big({-}k_0^2(t - t_2)\eee^{-q_0} + c\,\eee^{-\frac{q_0}{2}}\big)\ud t \\
&=-\frac 12 k_0^2\eee^{-q_0}(t_3 - t_2)^2 + c\,\eee^{-\frac{q_0}{2}}(t_3 - t_2) \leq \frac{c^2}{2k_0^2},
\end{aligned}
\end{equation}
so that \eqref{eq:q-mono-0} follows if we take $c$ small enough.
\end{proof}
We have the following immediate consequence of Proposition~\ref{prop:eps-bound}
and Lemma~\ref{lem:q-mono}.
\begin{corollary}
\label{cor:eps-bound}
There exist $C_0 > 0$ and $t_0 \geq T_0$ such that for all $t \geq t_0$
\begin{equation}
\label{eq:eps-bound-strong}
\| \varepsilon(t)\|_{H^1}^2 \leq C_0 \eee^{-(q_2(t) - q_1(t))}.
\end{equation}
\end{corollary}
\begin{proposition}
\label{prop:reduced}
Let $c > 0$. There exist $\rho > 0$ and $t_0 \geq T_0$ with the following property.
Let $q(t) := q_2(t) - q_1(t)$ and $p(t) := p_2(t) - p_1(t)$, where
$q_k(t)$ and $p_k(t)$ are the modulation parameters defined in Section~\ref{sec:modulation-two}.
Then for all $t \geq t_0$ the following inequalities are true:
\begin{align}
\label{eq:effective-dtq}
\big|q'(t) - \la Q, \wt Q\ra^{-1}p(t)\big| &\leq c\,\eee^{-\frac{q(t)}{2}}, \\
\label{eq:effective-dtp}
\big|p'(t) - 4k_0^2\eee^{-q(t)}\big| &\leq c\,\eee^{-q(t)}.
\end{align}
\end{proposition}
\begin{proof}
Subtracting \eqref{eq:dtqk} for $k = 1$ and $k = 2$ we get
\begin{equation}
\big|q'(t) - \la Q, \wt Q\ra^{-1}p(t)\big| \leq c\sqrt{\eee^{-q(t)} + \| \varepsilon\|_{H^1}^2}
\leq c\,\eee^{-\frac{q(t)}{2}},
\end{equation}
where the last inequality follows from Corollary~\ref{cor:eps-bound}.

We already know from Proposition~\ref{prop:eps-bound} that $\sigma = 1$.
Subtracting \eqref{eq:dtp1} and \eqref{eq:dtp2}
and using Corollary~\ref{cor:eps-bound}, we obtain \eqref{eq:effective-dtp}.
\end{proof}
\begin{remark}
An important feature of the system of differential inequalities \eqref{eq:effective-dtq},
\eqref{eq:effective-dtp} is that it does not involve the error term $ \varepsilon(t)$.
Thus the study of the dynamical behavior of the solution $u(t)$ to \eqref{eq:kdv}
is reduced to the study of a two-dimensional system of differential inequalities.
As we will see below, these inequalities determine the dynamics of the parameters $q(t)$
and $p(t)$, at least at the main order.
\end{remark}
\begin{proof}[Proof of Theorem~\ref{thm:classif}]
Let $r(t) := p(t) - 2\kappa\la Q, \wt Q\ra\eee^{-\frac{q(t)}{2}}$
for $t \geq T_0$, where $\kappa$ is defined in Remark~\ref{rem:kappa}.
Since $p_1(t)\to 0$, $p_2(t) \to 0$ and $q(t)\to \infty$ as $t \to \infty$, we first note that $r(t)\to 0$ as $t \to \infty$.

By \eqref{eq:effective-dtq} and \eqref{eq:effective-dtp}, we have
\begin{equation}
\label{eq:dtr}
\begin{aligned}
r'(t) &= p'(t) +\kappa\la Q, \wt Q\ra q'(t)\eee^{-\frac{q(t)}{2}}
= 4k_0^2 \eee^{-q(t)} +\kappa p(t)\eee^{-\frac{q(t)}{2}} + O(c\,\eee^{-q(t)}) \\
&= \kappa\eee^{-\frac{q(t)}{2}}\big(p(t) -2\kappa\la Q, \wt Q\ra\eee^{-\frac{q(t)}{2}}\big) + O(c\,\eee^{-q(t)}) = \kappa\eee^{-\frac{q(t)}{2}}r(t) + O(c\,\eee^{-q(t)}),
\end{aligned}
\end{equation}
where we have used the fact that $-2\kappa^2 \la Q, \wt Q\ra = 4k_0^2$ in order to pass
from the first to the second line.
We now show that for any $c_0$ there exists $t_0 \geq T_0$ such that
\begin{equation}
\label{eq:r-small}
|r(t)| \leq c_0 \eee^{-\frac{q(t)}{2}}, \qquad \text{for all }t\geq t_0.
\end{equation}
Suppose \eqref{eq:r-small} fails, so there exists $t_1$ arbitrarily large such that
$|r(t_1)| > c_0 \eee^{-\frac{q(t_1)}{2}}$. Assume $r(t_1) > 0$
(the case $r(t_1) < 0$ is similar).
Let $t_2 := \sup\big\{t: r(t) = c_0 \eee^{-\frac{q(t_1)}{2}}\big\}$.
We have $t_2 \in (t_1, \infty)$ and $r'(t_2) \leq 0$.
Since $q(t)$ is increasing, we have $r(t_2) \geq c_0 \eee^{-\frac{q(t_2)}{2}}$.
Thus \eqref{eq:dtr} yields $r'(t_2) > 0$, a contradiction.

From \eqref{eq:r-small} and \eqref{eq:effective-dtq} we deduce that for any $c_0 > 0$
and $t_0$ large enough we have
\begin{equation}
\big|q'(t) - 2\kappa \eee^{-\frac{q(t)}{2}}\big| \leq \frac{c_0}{2} \eee^{-\frac{q(t)}{2}}
\ \Leftrightarrow\ \big|\big(\eee^{\frac{q(t)}{2}}\big)' - \kappa\big| \leq \frac{c_0}{4},
\end{equation}
which implies, after integrating,
\begin{equation}
(\kappa - c_0)t \leq \eee^{\frac{q(t)}{2}} \leq (\kappa + c_0)t\ \Leftrightarrow
\ 2\log t + 2\log(\kappa - c_0) \leq q(t) \leq 2\log t + 2\log(\kappa + c_0),
\end{equation}
for $t$ large enough.
Since $c_0$ is arbitrary, this proves \eqref{eq:classif}.
\end{proof}

\appendix
\section{Reduced dynamics}
\label{sec:reduced}
In this section we prove some facts about the reduced equation for modulation parameters \eqref{eq:constrained}.
Of course $E(\bs x, \bs v)$ is a conservation law for this system.
Maybe not suprisingly,
\begin{equation}
M(\bs x, \bs v) := M\Big(\sum_{k=1}^K \sigma_k Q_{v_k}(\cdot - x_k)\Big)
\end{equation}
is also a conservation law.
Indeed, the Hamiltonian vector field corresponding to $M$ is the generator of space translations,
which leave $E(\bs x, \bs v)$ invariant.
\subsection{General lemmas}
We recall some facts from the theory of ordinary differential equations.

Given $t_0$, a Euclidean space $E$ and $\beta > 0$,
we denote $N_\beta(t_0; E)$ the space of continuous functions $f: [t_0, \infty) \to E$ such that
\begin{equation}
\|f\|_{N_\beta(t_0; E)} := \sup_{t \geq t_0} \eee^{\beta t}|f(t)|_E < +\infty.
\end{equation}
If $t_0$ and $E$ are known from the context, we write $N_\beta$ instead of $N_\beta(t_0; E)$.

Given $x_0 \in E$ and $\rho > 0$, we denote $B_E(x_0; \rho) := \{x \in E: |x - x_0|_E \leq \rho\}$ and $B_E(\rho) := B_E(0; \rho)$.
If $E$ is known from the context, we skip the subscript.
\begin{lemma}
\label{lem:lin-syst}
If $T \in \bR^{d\times d}$ is a matrix having no eigenvalues whose real part is smaller than $-\lambda \in \bR$, then for any $\beta > \lambda$ and $t_0 \in \bR$ the system
\begin{equation}
\label{eq:lin-syst}
\bs x'(t) = T\bs x(t) + \bs f(t)
\end{equation}
defines a bounded linear operator $S: N_\beta(t_0; \bR^d) \to N_\beta(t_0; \bR^d)$, $\bs f \mapsto \bs x = S\bs f$,
whose norm depends on $T$ and $\beta$.
\end{lemma}
\begin{proof}
Without loss of generality, we can assume $\lambda = 0$. Indeed, it suffices to set $\wt T := T + \lambda$, $\wt{\bs x}(t) := \eee^{\lambda t}\bs x(t)$
and $\wt{\bs f}(t) := \eee^{\lambda t}\bs f(t)$, so that $\|\wt{\bs x}\|_{N_{\beta - \lambda}} = \|\bs x\|_{N_{\beta}}$
and $\|\wt{\bs f}\|_{N_{\beta - \lambda}} = \|\bs f\|_{N_{\beta}}$.

The solutions of \eqref{eq:lin-syst} are given by the Duhamel formula:
\begin{equation}
\label{eq:lin-syst-1}
\bs x(t) = \eee^{(t-t_0)T}\bs x(t_0) + \int_{t_0}^t \eee^{(t-s)T}\bs f(s)\ud s, \qquad\text{for all }t \geq t_0.
\end{equation}
Applying $\eee^{-(t-t_0)T}$ to both sides, we get
\begin{equation}
\label{eq:lin-syst-2}
\bs x(t_0) = \eee^{-(t-t_0)T}\bs x(t) - \int_{t_0}^t \eee^{-(s - t_0)T}\bs f(s)\ud s, \qquad\text{for all }t \geq t_0.
\end{equation}
By the Jordan decomposition, there exists $\wt C > 0$, depending on $T$, such that
\begin{equation}
\label{eq:exp-norm}
\| \eee^{-tT} \|_{\bR^d \to \bR^d} \leq \wt C(1 + t^{d-1}),\qquad\text{for all }t \geq 0.
\end{equation}
Hence, if $\bs f \in N_\beta(t_0; \bR^d)$ and $\bs x \in N_\beta(t_0; \bR^d)$, we can pass to the limit $t \to \infty$
in \eqref{eq:lin-syst-2} and obtain
\begin{equation}
\bs x(t_0) = -\int_{t_0}^\infty \eee^{-(s - t_0)T}\bs f(s)\ud s.
\end{equation}
Plugging this into \eqref{eq:lin-syst-1}, we have
\begin{equation}
\bs x(t) = -\int_{t_0}^\infty \eee^{-(s-t)T}\bs f(s)\ud s + \int_{t_0}^t\eee^{(t-s)T}\bs f(s)\ud s = -\int_t^\infty \eee^{-(s-t)T}\bs f(s)\ud s,
\end{equation}
which is a continuous function. Using again \eqref{eq:exp-norm}, we compute
\begin{equation}
\begin{aligned}
|\bs x(t)| &\leq \wt C\int_t^\infty \big(1 + (s-t)^{d-1}\big)|\bs f(s)|\ud s \leq \wt C\|\bs f\|_{N_\beta}\int_t^\infty \big(1 + (s-t)^{d-1}\big)\eee^{-\beta s}\ud s \\
&= \wt C\eee^{-\beta t}\|\bs f\|_{N_\beta}\int_0^\infty \big(1+s^{d-1}\big)\eee^{-\beta s}\ud s \leq C(\beta^{-1} + \beta^{-d})\eee^{-\beta t}\|\bs f\|_{N_\beta},
\end{aligned}
\end{equation}
with $C$ depending only on $T$.
\end{proof}
\begin{proposition}
\label{prop:no-negative}
Let $T \in \bR^{d\times d}$ be a matrix having no eigenvalues with a negative real part.
For any $\beta > 0$ there exists $c_0 > 0$ such that the following is true.
Let $t_0 \in \bR$, $\eta > 0$ and $\bs f: [t_0, \infty) \times B_{\bR^d}(\eta) \to \bR^d$ be a~continuous function satisfying
\begin{align}
\label{eq:no-negative-ass1} |\bs f(t, 0)| &\leq \eee^{-\beta t},\qquad\qquad\text{for all }t \geq t_0, \\
\label{eq:no-negative-ass2} |\bs f(t, \sh{\bs x}) - f(t, \bs x)| &\leq c_0|\sh{\bs x} - \bs x|,\quad\ \,\text{for all }t \geq t_0\text{ and }|\bs x|, |\sh{\bs x}| \leq \eta.
\end{align}
There exist $t_1 \geq t_0$ and $\bs x_s \in N_\beta(t_1; \bR^d)$ such that
\begin{equation}
\label{eq:no-negative}
\bs x_s'(t) = T\bs x_s(t) + \bs f(t, \bs x_s(t)),\qquad\text{for all }t \geq t_1.
\end{equation}
If $t_2 \in \bR$ and $\sh {\bs x}_s \in N_\beta(t_2; \bR^d)$ solves \eqref{eq:no-negative} for all $t \geq t_2$,
then $\sh {\bs x}_s(t) = \bs x_s(t)$ for all $t \geq \max(t_1, t_2)$.
\end{proposition}
\begin{proof}
Given $\bs x : [t_1, \infty) \to \bR^d$ continuous and such that $\|\bs x\|_{L^\infty} \leq \eta$,
we denote $\bs f(\bs x)$ the function $t \mapsto \bs f(t, \bs x(t))$.
The system \eqref{eq:no-negative} is equivalent to the fixed point problem $\bs x_s = S\bs f(\bs x_s)$,
where $S$ is the operator from Lemma~\ref{lem:lin-syst}.
We now check that for any $\rho$ large enough and $t_1\in \bR$ large enough (depending on $\rho$), $\bs x \mapsto S\bs f(\bs x)$
is a~contraction on the ball in $N_\beta(t_1; \bR^d)$ of center $0$ and radius $\rho$,
which by the Contraction Principle will finish the proof.

Fix $\rho > 0$ and let $\|\bs x\|_{N_\beta(t_1)} \leq \rho$. If $t_1$ is large enough,
then the last bound implies in particular $\|\bs x\|_{L^\infty} \leq \eta$, hence $\bs f(\bs x)$ is a well-defined
continuous function. Using \eqref{eq:no-negative-ass1} and \eqref{eq:no-negative-ass2}, we have
\begin{equation}
|\bs f(t, \bs x(t))| \leq \eee^{-\beta t}(1 + c_0\rho) \quad\Rightarrow\quad \|\bs f(\bs x)\|_{N_\beta(t_1)} \leq 1 + c_0\rho.
\end{equation}
By Lemma~\ref{lem:lin-syst}, if $c_0$ is small enough and $\rho$ large enough, then $\|S\bs f(\bs x)\|_{N_\beta} \leq \rho$.

Now, let $\sh{\bs x}$ be another function satisfying $\|\sh{\bs x}\|_{N_\beta(t_1)} \leq \rho$.
Using \eqref{eq:no-negative-ass2}, we have
\begin{equation}
|\bs f(t, \sh{\bs x}(t)) - \bs f(t, \bs x(t))| \leq c_0|\sh{\bs x}(t) - \bs x(t)| \quad\Rightarrow\quad \|\bs f(\sh{\bs x}) - \bs f(\bs x)\|_{N_\beta(t_1)} \leq c_0\|\sh{\bs x} - \bs x\|_{N_\beta(t_1)}.
\end{equation}
By Lemma~\ref{lem:lin-syst}, if $c_0$ is small enough, then $\|S\bs f(\sh{\bs x}) - S\bs f(\bs x)\|_{N_\beta} \leq \frac 12 \|\sh{\bs x} - \bs x\|_{N_\beta}$.
\end{proof}
\begin{lemma}
\label{lem:one-negative}
Let $T \in \bR^{d\times d}$ be a matrix with exactly one negative eigenvalue $-\lambda$ and
all the other eigenvalues having a non-negative real part. Denote $\cY_s$ an eigenvector of $T$ corresponding to the eigenvalue $-\lambda$ and $\alpha_s$ the eigenvector of $T^*$, the transpose of $T$, corresponding to the eigenvalue $-\lambda$,
normalised so that $\alpha_s \cdot \cY_s = 1$.
Let $t_0 \in \bR$ and $\beta \in (0, \lambda)$.
For any $\bs f \in N_\beta(t_0; \bR^d)$ the system
\begin{equation}
\label{eq:lin-syst-neg}
\bs x'(t) = T\bs x(t) + \bs f(t)
\end{equation}
has a unique solution $\bs x \in N_\beta(t_0; \bR^d)$ such that $\alpha_s \cdot \bs x(t_0) = 0$.
The mapping $\bs f \mapsto \bs x = S_p\bs f$ is a bounded linear operator $N_\beta(t_0; \bR^d) \to N_\beta(t_0; \bR^d)$
whose norm depends on $T$ and $\beta$.

If $\bs x, \sh{\bs x} \in N_\beta(t_0; \bR^d)$ are two solutions of \eqref{eq:lin-syst-neg},
then there exists $a \in \bR$ such that $\sh{\bs x}(t) - \bs x(t) = a\eee^{-\lambda t}\cY_s$ for all $t \geq t_0$.
\end{lemma}
\begin{remark}
We denote the solution operator $S_p$ to indicate that it yields a \emph{particular} solution of \eqref{eq:lin-syst-neg},
chosen in an arbitrary and non-canonical way.
\end{remark}
\begin{proof}[Proof of Lemma~\ref{lem:one-negative}]
The solutions of \eqref{eq:lin-syst-neg} are given by the Duhamel formula:
\begin{equation}
\label{eq:lin-syst-neg-1}
\bs x(t) = \eee^{(t-t_0)T}\bs x(t_0) + \int_{t_0}^t \eee^{(t-s)T}\bs f(s)\ud s, \qquad\text{for all }t \geq t_0.
\end{equation}

Let $X_u := \{\bs x \in \bR^d: \alpha_s \cdot \bs x = 0\}$ and let $T_u: X_u \to X_u$ be the restriction of $T$ to $X_u$.
Let $\bs x(t) = a(t)\cY_s + \bs x_u(t)$ with $\bs x_u(t) \in X_u$ and $\bs f(t) = b(t)\cY_s + \bs f_u(t)$ with $\bs f_u(t) \in X_u$.
Taking the inner product of \eqref{eq:lin-syst-neg-1} with $\alpha_s$, we obtain
\begin{equation}
\label{eq:lin-syst-neg-2}
a(t) = \eee^{-\lambda(t-t_0)}a(t_0) + \int_{t_0}^t \eee^{-\lambda(t - s)}b(s)\ud s = \int_{t_0}^t \eee^{-\lambda(t - s)}b(s)\ud s,
\end{equation}
where the last equality follows since we require $\alpha_s \cdot \bs x(t_0) = 0$.
Taking the difference of \eqref{eq:lin-syst-neg-1} and \eqref{eq:lin-syst-neg-2} multiplied by $\cY_s$, we obtain
\begin{equation}
\label{eq:lin-syst-neg-3}
\bs x_u(t) = \eee^{(t-t_0)T_u}\bs x_u(t_0) + \int_{t_0}^t \eee^{(t-s)T_u}\bs f_u(s)\ud s, \qquad\text{for all }t \geq t_0.
\end{equation}
Applying $\eee^{-(t-t_0)T_u}$ to both sides, we get
\begin{equation}
\label{eq:lin-syst-neg-4}
\bs x_u(t_0) = \eee^{-(t-t_0)T_u}\bs x_u(t) - \int_{t_0}^t \eee^{-(s - t_0)T_u}\bs f_u(s)\ud s, \qquad\text{for all }t \geq t_0.
\end{equation}
By the Jordan decomposition, there exists $C_1 > 0$, depending on $T$, such that
\begin{equation}
\label{eq:exp-norm-neg}
\| \eee^{-tT_u} \|_{X_u \to X_u} \leq C_1(1 + t^{d-2}),\qquad\text{for all }t \geq 0.
\end{equation}
Hence, if $\bs f \in N_\beta(t_0; \bR^d)$ and $\bs x \in N_\beta(t_0; \bR^d)$, we can pass to the limit $t \to \infty$
in \eqref{eq:lin-syst-neg-4} and obtain
\begin{equation}
\bs x_u(t_0) = -\int_{t_0}^\infty \eee^{-(s - t_0)T_u}\bs f_u(s)\ud s.
\end{equation}
Plugging this into \eqref{eq:lin-syst-neg-3}, we have
\begin{equation}
\bs x_u(t) = -\int_{t_0}^\infty \eee^{-(s-t)T_u}\bs f_u(s)\ud s + \int_{t_0}^t\eee^{(t-s)T_u}\bs f_u(s)\ud s = -\int_t^\infty \eee^{-(s-t)T_u}\bs f_u(s)\ud s,
\end{equation}
which, together with \eqref{eq:lin-syst-neg-2}, uniquely determines a continuous function $\bs x(t)$. We have
\begin{equation}
\begin{aligned}
|a(t)| &\leq C_2\|\bs f\|_{N_\beta}\int_{t_0}^t \eee^{-\lambda(t - s)}\eee^{-\beta s}\ud s = C_2\eee^{-\beta t}\|\bs f\|_{N_\beta}\int_{t_0}^t \eee^{-(\lambda - \beta)(t-s)}\ud s \\
&\leq C_2 \eee^{-\beta t}\|\bs f\|_{N_\beta}\int_{0}^\infty \eee^{-(\lambda - \beta)s}\ud s \leq C_3 \eee^{-\beta t}\|\bs f\|_{N_\beta},
\end{aligned}
\end{equation}
where $C_2$ depends on $T$, and $C_3$ depends on $T$ and $\beta$. Finally,
\begin{equation}
\begin{aligned}
|\bs x_u(t)| &\leq C_4\int_t^\infty \big(1 + (s-t)^{d-2}\big)|\bs f(s)|\ud s
\leq C_4 \|\bs f\|_{N_\beta}\int_t^\infty \big(1 + (s-t)^{d-2}\big)\eee^{-\beta s}\ud s \\
&= C_4 \eee^{-\beta t}\|\bs f\|_{N_\beta}\int_0^\infty \big(1+s^{d-2}\big)\eee^{-\beta s}\ud s
\leq C_5 \eee^{-\beta t}\|\bs f\|_{N_\beta},
\end{aligned}
\end{equation}
where $C_4$ depends on $T$, and $C_5$ depends on $T$ and $\beta$.

If $\sh{\bs x}$ and $\bs x$ both belong to $N_\beta(t_0; \bR^d)$ and solve \eqref{eq:lin-syst-neg},
then $\bs z := \sh{\bs x} - \bs x \in N_\beta(t_0; \bR^d)$ solves $\bs z'(t) = T\bs z(t)$ for all $t \geq t_0$,
hence the Jordan decomposition yields $\bs z(t) = a\eee^{-\lambda t}\cY_s$ for some $a \in \bR$.
\end{proof}
The following version of the Stable Manifold Theorem for non-autonomous systems will be useful.
\begin{proposition}
\label{prop:one-negative}
Let $T \in \bR^{d\times d}$ be a matrix with exactly one negative eigenvalue $-\lambda$ and
all the other eigenvalues having a non-negative real part. Denote $\cY_s$ an eigenvector of $T$ corresponding to the eigenvalue $-\lambda$.
Let $t_0 \in \bR$, $\beta, \gamma, \eta > 0$
and $\bs f: [t_0, \infty) \times B_{\bR^d}(\eta) \to \bR^d$ be a~continuous function satisfying
\begin{gather}
\label{eq:one-negative-ass1} |\bs f(t, 0)| \leq \eee^{-\beta t},\qquad\text{for all }t \geq t_0, \\
\label{eq:one-negative-ass2} |\bs f(t, \sh{\bs x}) - f(t, \bs x)| \leq \eee^{-\gamma t}|\sh{\bs x} - \bs x|,\quad\text{for all }t \geq t_0\text{ and }|\bs x|, |\sh{\bs x}| \leq \eta.
\end{gather}

There exists a family $\{\bs x_a \in N_\beta(\tau_a; \bR^d): a \in \bR\}$
having the following properties:
\begin{itemize}
\item for all $a \in \bR$, $\bs x = \bs x_a$ solves for all $t \geq \tau_a$ the equation
\begin{equation}
\label{eq:one-negative}
\bs x'(t) = T\bs x(t) + \bs f(t, \bs x(t)),
\end{equation}
\item for all $a, \sh a \in \bR$ and $\nu < \lambda + \gamma$, $\bs x_{\sh a} - \bs x_a - (\sh a - a)\exp(-\lambda\,\cdot)\cY_s \in N_{\nu}(\max(\tau_a, \tau_{\sh a}); \bR^d)$,
\item if $t_1 \in \bR$ and $\bs x \in N_\beta(t_1; \bR^d)$ solves \eqref{eq:one-negative} for all $t \geq t_1$,
then there exists $a \in \bR$ such that $\bs x(t) = \bs x_a(t)$ for all $t \geq \max(t_1, \tau_a)$.
\end{itemize}
\end{proposition}
\begin{proof}
By exactly the same computation as in the proof of Proposition~\ref{prop:no-negative},
we obtain that for any $\rho$ large enough and $\tau_0\in \bR$ large enough (depending on $\rho$), $\bs x \mapsto S_p\bs f(\bs x)$,
where $S_p$ is the operator defined in Lemma~\ref{lem:one-negative},
is a~contraction on the ball in $N_\beta(\tau_0; \bR^d)$ of center $0$ and radius $\rho$.
We denote $\bs x_0 \in N_\beta(\tau_0; \bR^d)$ its unique fixed point.

Setting $\bs x = \bs x_0 + \bs y$, we rewrite \eqref{eq:one-negative} as
\begin{equation}
\label{eq:one-negative-1}
\bs y'(t) = T\bs y(t) + \big[ \bs f(t, \bs x_0(t) + \bs y(t)) - \bs f(t, \bs x_0(t)) \big].
\end{equation}
Let $\nu \in (\lambda, \lambda + \gamma)$.
We first show that if $\bs y \in N_\beta(t_1; \bR^d)$ solves \eqref{eq:one-negative-1} for all $t \geq t_1$,
then there exists $a \in \bR$ such that $\bs y - a\exp(-\lambda\,\cdot)\cY_s \in N_\nu(t_1; \bR^d)$.

Indeed, if $\bs y \in N_\beta(t_1; \bR^d)$, then \eqref{eq:one-negative-ass2} implies
$\bs f(\bs x_0 + \bs y) - \bs f(\bs x_0) \in N_{\beta + \gamma}(t_1; \bR^d)$,
thus Lemma~\ref{lem:one-negative} yields $\bs y \in N_{\wt \beta}(t_1; \bR^d)$ for any $\wt \beta < \min(\beta + \gamma, \lambda)$. Repeating a finite number of times, we obtain $\bs y \in N_{\nu - \gamma}(t_1; \bR^d)$,
which, again using \eqref{eq:one-negative-ass2}, implies
$\bs f(\bs x_0 + \bs y) - \bs f(\bs x_0) \in N_{\nu}(t_1; \bR^d)$.
Hence,
\begin{equation}
S(\bs f(\bs x_0 + \bs y) - \bs f(\bs x_0)) \in N_{\nu}(t_1; \bR^d),
\end{equation}
where $S$ is the operator defined in Lemma~\ref{lem:lin-syst}.
But $\bs z := \bs y - S(\bs f(\bs x_0 + \bs y) - \bs f(\bs x_0))$ satisfies the homogeneous equation $\bs z'(t) = T\bs z(t)$,
yielding $\bs z(t) = a\eee^{-\lambda t}\cY_s$ for some $a \in \bR$.

Now, we prove that for every $a \in \bR$ there exists a unique solution $\bs y_a$ of \eqref{eq:one-negative-1}
such that $\bs y_a - a\exp(-\lambda\,\cdot)\cY_s \in N_\nu(t_1; \bR^d)$ for some $t_1 \in \bR$.
Writing $\bs y_a(t) = a\exp(-\lambda t)\cY_s + \bs z(t)$, \eqref{eq:one-negative-1} becomes
\begin{equation}
\bs z'(t) = T\bs z(t) + \big[ \bs f(t, \bs x_0(t) + a\exp(-\lambda t)\cY_s + \bs z(t)) - \bs f(t, \bs x_0(t)) \big].
\end{equation}
Since $\bs f(\bs x_0 + a\exp(-\lambda\,\cdot)\cY_s + \bs z) - \bs f(\bs x_0) \in N_\nu(t_1; \bR^d)$ and $\nu > \lambda$,
the last equation is equivalent to
\begin{equation}
\bs z = S\big( \bs f(\bs x_0 + a\exp(-\lambda\,\cdot)\cY_s + \bs z) - \bs f(\bs x_0)\big).
\end{equation}
By a similar computation as in the proof of Proposition~\ref{prop:no-negative},
if $\rho$ is large enough (depending on $a$), then the right hand side defines a contraction on the ball in $N_\nu(t_1; \bR^d)$ of center 0 and radius $\rho$ for $t_1$ large enough.

This proves existence and uniqueness of $\bs z$, and thus of $\bs y_a$. We set $\bs x_a := \bs x_0 + \bs y_a$.
\end{proof}
\subsection{Distinct limit speeds}
\label{ssec:constrained-distinct}
For given $\bs x(t) = (x_1(t), \ldots, x_K(t))$ we denote
\begin{equation}
L(t) := \min_{1\leq k < K}(x_{k+1}(t) - x_k(t)).
\end{equation}
We also denote
\begin{equation}
\begin{pmatrix} X(\bs x, \bs v) \\ V(\bs x, \bs v)\end{pmatrix}
:= \begin{pmatrix}A(\bs x, \bs v) & C(\bs x, \bs v) \\
-C(\bs x, \bs v) & B(\bs x, \bs v)\end{pmatrix}^{-1}\begin{pmatrix}\partial_{\bs x}E(\bs x, \bs v) \\ \partial_{\bs v}E(\bs x, \bs v)\end{pmatrix}
\end{equation}
the right hand side of \eqref{eq:constrained}.
\begin{proposition}
\label{prop:constrained-distinct}
Let $\bs v^\infty = (v_1^\infty, \ldots, v_K^\infty) \in ((0, v_*) \setminus V_\tx{crit})^K$
be such that $v_1^\infty < \ldots < v_K^\infty$  and let $(\bs x(t), \bs v(t))$ be a global solution
to \eqref{eq:constrained} such that $\lim_{t \to \infty}v_k(t) = v_k^\infty$ and $\lim_{t \to \infty}L(t) = \infty$.
Then there exist $x_1^\infty, \ldots, x_K^\infty \in \bR$ and $\beta > 0$ such that,
for $k \in \{1, \ldots, K\}$,
\begin{equation}
\label{eq:conv-exp}
|x_k(t) - (v_k^\infty t + x_k^\infty)| \leq \eee^{-\beta t},\qquad |v_k(t) - v_k^\infty| \leq \eee^{-\beta t},\qquad\text{for }t \text{ sufficiently large}.
\end{equation}
Moreover, for any $x_1^\infty, \ldots, x_K^\infty \in \bR$ there exists a unique solution to \eqref{eq:constrained}
satisfying \eqref{eq:conv-exp}.
\end{proposition}
In this section and the next one, $\beta$ denotes a positive number which can change a finite number of times
in the course of the proof. This is why we do not distinguish for example between $\lesssim \eee^{-\beta t}$ and $\leq \eee^{-\beta t}$.

We need the following bounds on $(X, V)$.
\begin{lemma}
\label{lem:red-ham-distinct}
Let $I$ be a compact interval $\subset (0, v_*) \setminus V_\tx{crit}$.
There exist $\beta, L_0 > 0$ such that for all $(\bs x, \bs v)$ and $(\bs y, \bs w)$
with $\bs v, \bs w \in I^K$ and $L:= \min_{1\leq k < K}(x_{k+1} - x_k) \geq L_0$ the following bounds are true:
\begin{equation}
\label{eq:red-ham-distinct-1}
\big|X(\bs x, \bs v) - \bs v\big| + \big|V(\bs x, \bs v)\big|\leq \eee^{-\beta L},
\end{equation}
\begin{equation}
\label{eq:red-ham-distinct-2}
\begin{aligned}
|X(\bs x, \bs v)
- X(\bs y, \bs w)-(\bs v - \bs w)| + |V(\bs x, \bs v)
- V(\bs y, \bs w)| \leq \eee^{-\beta L}\big(|\bs x - \bs y| + |\bs v - \bs w|\big).
\end{aligned}
\end{equation}
\end{lemma}
\begin{proof}
Using the Chain Rule and Lemma~\ref{lem:Taylor}, one obtains
\begin{equation}
\begin{aligned}
\Big|E(\bs x, \bs v) - \sum_{k=1}^K E\big(Q_{v_k}\big)\Big| &+\Big|\grad\Big(E(\bs x, \bs v) - \sum_{k=1}^K E\big(Q_{v_k}\big)\Big)\Big| \\
&+\Big|\grad^2\Big(E(\bs x, \bs v) - \sum_{k=1}^K E\big(Q_{v_k}\big)\Big)\Big| \leq \eee^{-\beta L}
\end{aligned}
\end{equation}
for some $\beta > 0$ and $L$ large enough, where ``$\grad$'' is the gradient in $\bR^{2K}$.
Alternatively, we can write:
\begin{equation}
\begin{aligned}
&|\partial_{\bs x} E(\bs x, \bs v)| + |\partial_{\bs v} E(\bs x, \bs v) + D(\bs x, \bs v)\bs v| \\
&\ +|\partial_{\bs x}^2 E(\bs x, \bs v)| + |\partial_{\bs x}\partial_{\bs v} E(\bs x, \bs v)| + |\partial_{\bs v}^2 E(\bs x, \bs v) + \partial_{\bs v}(D(\bs x, \bs v)\bs v)| \leq \eee^{-\beta L},
\end{aligned}
\end{equation}
where $D(\bs x, \bs v)$ is the diagonal $K\times K$ matrix with entries $\la Q_{v_k}, \wt Q_{v_k}\ra$
(we are using here the fact that $\vD E(Q_v) = -v Q_v$).
Similarly, we have
\begin{equation}
|A(\bs x, \bs v)| + |\grad A(\bs x, \bs v)| + |C(\bs x, \bs v)-D(\bs x, \bs v)| + |\grad(C(\bs x, \bs v)-D(\bs x, \bs v))| \leq \eee^{-\beta L}.
\end{equation}
From the equality
\begin{equation}
\label{eq:XV-def-2}
\begin{pmatrix}A(\bs x, \bs v) & C(\bs x, \bs v) \\
-C(\bs x, \bs v) & B(\bs x, \bs v)\end{pmatrix}\begin{pmatrix} X(\bs x, \bs v) \\ V(\bs x, \bs v)\end{pmatrix}
= \begin{pmatrix}\partial_{\bs x}E(\bs x, \bs v) \\ \partial_{\bs v}E(\bs x, \bs v)\end{pmatrix}
\end{equation}
and the estimates above we obtain
\begin{equation}
\begin{pmatrix}0 & D(\bs x, \bs v) \\
-D(\bs x, \bs v) & B(\bs x, \bs v)\end{pmatrix}\begin{pmatrix} X(\bs x, \bs v) \\ V(\bs x, \bs v)\end{pmatrix}
= \begin{pmatrix}0 \\ -D(\bs x, \bs v)\bs v\end{pmatrix} + O(\eee^{-\beta L}),
\end{equation}
which easily implies \eqref{eq:red-ham-distinct-1}.

Similarly, differentiating \eqref{eq:XV-def-2} with respect to $\bs x$ and using $\partial_{\bs x}D(\bs x, \bs v) = 0$
we get
\begin{equation}
\begin{pmatrix}0 & 0 \\
0 & \partial_{\bs x}B(\bs x, \bs v)\end{pmatrix}\begin{pmatrix} X(\bs x, \bs v) \\ V(\bs x, \bs v)\end{pmatrix}
+ \begin{pmatrix}0 & D(\bs x, \bs v) \\
-D(\bs x, \bs v) & B(\bs x, \bs v)\end{pmatrix}\begin{pmatrix} \partial_{\bs x}X(\bs x, \bs v) \\ \partial_{\bs x}V(\bs x, \bs v)\end{pmatrix}
= O(\eee^{-\beta L}),
\end{equation}
which yields
\begin{equation}
\label{eq:partial-x-small}
|\partial_{\bs x}X(\bs x, \bs v)| + |\partial_{\bs x}V(\bs x, \bs v)| \lesssim \eee^{-\beta L}.
\end{equation}
Differentiating \eqref{eq:XV-def-2} with respect to $\bs v$ we get
\begin{equation}
\begin{pmatrix}0 & \partial_{\bs v}D \\
-\partial_{\bs v}D & \partial_{\bs v}B\end{pmatrix}\begin{pmatrix} X \\ V\end{pmatrix}
+ \begin{pmatrix}0 & D \\
-D & B\end{pmatrix}\begin{pmatrix} \partial_{\bs v}X\\ \partial_{\bs v}V\end{pmatrix}
= \begin{pmatrix}0 \\ -\partial_{\bs v}(D\bs v)\end{pmatrix}+O(\eee^{-\beta L}),
\end{equation}
which yields
\begin{equation}
\label{eq:partial-v-small}
|\partial_{\bs v}X - \tx{Id}| + |\partial_{\bs v}V| \lesssim \eee^{-\beta L}.
\end{equation}
Now \eqref{eq:red-ham-distinct-2} follows from \eqref{eq:partial-x-small} and \eqref{eq:partial-v-small},
perhaps with smaller $\beta$.
\end{proof}
\begin{proof}[Proof of Proposition~\ref{prop:constrained-distinct}]
First we prove \eqref{eq:conv-exp}, then we will prove the uniqueness part.
Let $\delta := \min_{1 \leq k < K}(v_{k+1}^\infty - v_k^\infty) > 0$.
Since we assume $L(t) \to \infty$, Lemma~\ref{lem:red-ham-distinct}
implies $|x_k'(t) - v_k(t)| \lesssim \frac \delta 3$ for $t$ large enough.
We also assume $v_k(t) \to v_k^\infty$, thus $x_{k+1}'(t) - x_k'(t) \geq \frac \delta 2$ for $t$ large,
which implies $L(t) \geq \frac \delta 3 t$ for $t$ large.
Applying again Lemma~\ref{lem:red-ham-distinct}, we obtain that there exists $\beta > 0$ such that
\begin{equation}
|x_k'(t) - v_k(t)| \leq \eee^{-2\beta t},\quad |v_k'| \leq \eee^{-2\beta t},\quad\text{for }t\text{ large},
\end{equation}
which yields \eqref{eq:conv-exp}.

Given $\bs x^\infty$, we define $(\wt {\bs x}(t), \wt {\bs v}(t))$ by
\begin{equation}
\begin{aligned}
\wt{\bs x}(t) &:= \bs x(t) - \big(\bs v^\infty t + \bs x^\infty \big), \\
\wt{\bs v}(t) &:= \bs v(t) - \bs v^\infty.
\end{aligned}
\end{equation}
Then $(\bs x(t), \bs v(t))$ solves \eqref{eq:constrained} if and only if $(\wt{\bs x}(t), \wt {\bs v}(t))$
solves
\begin{equation}
\begin{pmatrix} \wt{\bs x}'(t) \\ \wt{\bs v}'(t)\end{pmatrix}
= T\begin{pmatrix} \wt{\bs x}(t) \\ \wt{\bs v}(t)\end{pmatrix}
+ \begin{pmatrix} \wt X(t, \wt{\bs x}(t), \wt{\bs v}(t)) \\ \wt V(t, \wt{\bs x}(t), \wt{\bs v}(t))\end{pmatrix},
\end{equation}
where
\begin{equation}
\begin{aligned}
T &:= \begin{pmatrix} 0 & \tx{Id} \\ 0 & 0\end{pmatrix}, \\
\wt X(t, \wt{\bs x}, \wt{\bs v}) &:= X(\bs v^\infty t + \bs x^\infty + \wt{\bs x}, \bs v^\infty+\wt{\bs v})-(\bs v^\infty + \wt{\bs v}), \\
\wt V(t, \wt{\bs x}, \wt{\bs v}) &:= V(\bs v^\infty t + \bs x^\infty + \wt{\bs x}, \bs v^\infty+\wt{\bs v}).\end{aligned}
\end{equation}
From Lemma~\ref{lem:red-ham-distinct} we obtain that there exists $\beta > 0$
such that for $t$ large enough and $|\wt{\bs x}| + |\wt{\bs v}| +|\wt{\bs y}| +|\wt{\bs w}| \leq 1$ we have
\begin{equation}
\begin{aligned}
|\wt X(t, \wt{\bs x}, \wt{\bs v})| + |\wt V(t, \wt{\bs x}, \wt{\bs v})| &\leq \eee^{-\beta t}, \\
|\wt X(t, \wt{\bs x}, \wt{\bs v}) - \wt X(t, \wt{\bs y}, \wt{\bs w})| +
|\wt V(t, \wt{\bs x}, \wt{\bs v}) - \wt V(t, \wt{\bs y}, \wt{\bs w})| &\leq \eee^{-\beta t}(|\wt{\bs x} - \wt{\bs y}| + |\wt{\bs v} - \wt{\bs w}|).
\end{aligned}
\end{equation}
From \eqref{eq:conv-exp}, we obtain that if $(\bs x(t), \bs v(t))$
solves \eqref{eq:constrained}, then $|\wt{\bs x}(t)| + |\wt{\bs v}(t)| \leq \eee^{-\beta t}$ for $t$ large.
Since $T$ has no eigenvalues with a negative real part, invoking Proposition~\ref{prop:no-negative}
finishes the proof.
\end{proof}

\subsection{Two-solitons with the same limit speed}
\label{ssec:constrained-same}
We proceed to the proof of Proposition~\ref{prop:constrained}.
By rescaling the space variable we can assume $v^\infty = 1$.
The functional $H(\bs x, \bs v) := E(\bs x, \bs v) + \frac 12 M(\bs x, \bs v)$ plays
an important role in our analysis. It is a conservation law for \eqref{eq:constrained}.

Until the end of this section, we denote $L := x_2 - x_1$ and $L(t) := x_2(t) - x_1(t)$.
We abbreviate $\bs 1 := (1, 1)$ and $\bs\iota := (-1, 1)$.
\begin{lemma}
\label{lem:v-bound-by-H}
There exist numbers $\delta, L_0, C_0 > 0$ such that if $|\bs v - \bs 1| \leq \delta$
and $L \geq L_0$, then
\begin{equation}
|\bs v - \bs 1|^2 \leq C_0(|H(\bs x, \bs v) - 2H(Q)| + \eee^{-L}).
\end{equation}
\end{lemma}
\begin{remark}
In this section, the conservation law $H$ is used only to obtain the bound \eqref{eq:bounds-v-equal-0} in the proof of Proposition~\ref{prop:constrained} below.
In the proof of Theorem~\ref{thm:classif}, the functional $H$ is crucially used
to obtain bounds on the error term $\varepsilon(t)$. The reason why we were unable to treat
the case $\int_\bR Q(x)\wt Q(x)\ud x > 0$ in Theorem~\ref{thm:classif} is that in this situation
the terms $\|\varepsilon\|_{H^1}^2$ and $|\bs v - \bs 1|^2$ come with opposite signs
in the expansion of $H$, so we could not obtain any useful estimate for either of them.
\end{remark}
\begin{proof}[Proof of Lemma~\ref{lem:v-bound-by-H}]
As in the proof of Proposition~\ref{prop:coer}, denote $R_1 := Q(\cdot - x_1)$, $R_2 := Q(\cdot - x_2)$, $U := \sigma_1 R_1 + \sigma_2 R_2$.
Let $\varepsilon := \sigma_1 Q_{v_1}(\cdot - x_1) + \sigma_2 Q_{v_2}(\cdot - x_2) - U$.
Then
\begin{equation}
\|\varepsilon -( \sigma_1(v_1 - 1)\wt Q(\cdot - x_1) + \sigma_2(v_2 - 1)\wt Q(\cdot - x_2))\|_{H^1} \lesssim |\bs v - \bs 1|^2.
\end{equation}
We have the crucial non-degeneracy
\begin{equation}
\la \wt Q, \vD^2 H(Q)\wt Q\ra = \la \wt Q, L\wt Q\ra = \la \wt Q, -Q\ra \neq 0.
\end{equation}
Combining this with similar estimates as in the proof of Lemma~\ref{lem:D2H}, we obtain
\begin{equation}
|\la \varepsilon, \vD^2 H(U)\varepsilon\ra| \gtrsim |\bs v - \bs 1|^2.
\end{equation}
We also have $|H(U) - 2H(Q)| \lesssim \eee^{-L}$ and $|\la \vD H(U), \varepsilon\ra| \ll \eee^{-L} + |\bs v - \bs 1|^2$,
see \eqref{eq:H-deux-solit} and \eqref{eq:small-tension-0}.
The conclusion follows from the Taylor expansion \eqref{eq:H-Tay}.
\end{proof}

We need a more precise estimate of the right hand side of \eqref{eq:constrained}
than the one provided by Lemma~\ref{lem:red-ham-distinct}.
\begin{lemma}
\label{lem:red-ham-same}
Assume $1 \in (0, v_*) \setminus V_\tx{crit}$. There exist $\beta, \delta > 0$ such that for all $(\bs x, \bs v)$ and $(\bs y, \bs w)$
with $|\bs v - \bs 1| + |\bs w - \bs 1| + |\bs x - \bs y| \leq \delta$ the following bounds hold for $L := x_2 - x_1$ large enough:
\begin{gather}
\big|X(\bs x, \bs v) - \bs v\big| \leq \eee^{-(\frac 12 + \beta)L},\\
\big| V(\bs x, \bs v) - \bs \phi(\bs x) \big| \leq \eee^{-(1+\beta)L} + |\bs v -\bs 1|\eee^{-(\frac 12 + \beta)L}, \\
\big|X(\bs x, \bs v) - X(\bs y, \bs w) - (\bs v - \bs w)\big| \leq \eee^{-(\frac 12 + \beta)L}|\bs x - \bs y| + \eee^{-\beta L}|\bs v - \bs w|, \\
\begin{aligned}
&\big| V(\bs x, \bs v) - V(\bs y, \bs w) - (\bs \phi(\bs x)-\bs\phi(\bs y)) \big| \\
&\qquad\qquad\leq \big(\eee^{-(1+\beta)L} + |\bs v -\bs 1|\eee^{-(\frac 12 + \beta)L}\big)|\bs x - \bs y| + \eee^{-(\frac 12 + \beta)L}|\bs v - \bs w|,
\end{aligned}
\end{gather}
with $\bs \phi(\bs x) := \big(\sigma\kappa^2\eee^{-L}, {-}\sigma\kappa^2\eee^{-L}
\big)$, where $\sigma := {-}\tx{sgn}\big(\la Q, \wt Q\ra\sigma_1\sigma_2\big)$ and $\kappa$ is a~constant.
\end{lemma}
\begin{proof}
Denote $\bs \psi(\bs x) := \bs\iota\sigma_1\sigma_2\int_\bR (\partial_x Q(\cdot - x_1))Q(\cdot - x_2)\ud x$.
We claim that
\begin{align}
\label{eq:E-deriv-x-same}
|\partial_x E(\bs x, \bs v) -\la Q, \wt Q\ra\bs \phi(\bs x) +\bs\psi(\bs x) | &\leq \eee^{-(1+\beta)L}+|\bs v - \bs 1|\eee^{-(\frac 12 + \beta)L}, \\
|\partial_v E(\bs x, \bs v) + D(\bs x, \bs v)\bs v| &\leq \eee^{-(\frac 12 + \beta)L}, \label{eq:E-deriv-v-same} \\
|A(\bs x, \bs v)| + |C(\bs x, \bs v) - D(\bs x, \bs v)| &\leq \eee^{-(\frac 12 + \beta)L}, \label{eq:AC-same} \\
|A(\bs x, \bs v)\bs v + \bs \psi(\bs x)| &\leq \eee^{-(1+\beta)L}+|\bs v - \bs 1|\eee^{-(\frac 12 + \beta)L}. \label{eq:Av-psi}
\end{align}
Bounds \eqref{eq:E-deriv-v-same}, \eqref{eq:AC-same} and \eqref{eq:Av-psi} follow from the fact that $|\bs v - \bs 1|$ is small
($\beta$ could be any number $< \frac 12 - \delta$, see Remark~\ref{rem:unif-bound}).
The proof of the first bound, the details of which we skip, is similar to the proof of Lemma~\ref{lem:inter-asym}.

Consider the matrix
\begin{equation}
\begin{aligned}
\begin{pmatrix} G(\bs x, \bs v) & K(\bs x, \bs v) \\ -K(\bs x, \bs v) & J(\bs x, \bs v)\end{pmatrix} &:= \begin{pmatrix} A(\bs x, \bs v) & C(\bs x, \bs v) \\ -C(\bs x, \bs v) & B(\bs x, \bs v)\end{pmatrix}^{-1} \\
&= \left(\begin{pmatrix} 0 & D \\ -D & B\end{pmatrix}+\begin{pmatrix} A & C-D \\ -C+D & 0\end{pmatrix}\right)^{-1}.
\end{aligned}
\end{equation}
By the standard asymptotic expansion for the matrix inverse we have
\begin{multline}
\begin{pmatrix} G & K \\ -K & J\end{pmatrix}
= \begin{pmatrix} 0 & D \\ -D & B\end{pmatrix}^{-1}
- \begin{pmatrix} 0 & D \\ -D & B\end{pmatrix}^{-1}
\begin{pmatrix} A & C-D \\ -C+D & 0\end{pmatrix}
\begin{pmatrix} 0 & D \\ -D & B\end{pmatrix}^{-1}\\
+ O(\eee^{-(1+\beta)L}).
\end{multline}
Since $\begin{pmatrix} 0 & D \\ -D & B\end{pmatrix}^{-1} = \begin{pmatrix} D^{-1}BD^{-1} & -D^{-1} \\ D^{-1} & 0\end{pmatrix}
= O(1)$,
we obtain
\begin{equation}
\begin{aligned}
\label{eq:inv-matrix-same}
K &= -D^{-1} + O(\eee^{-(\frac 12 + \beta)L}), \qquad J = D^{-1}AD^{-1} + O(\eee^{-(1+\beta)L}).
\end{aligned}
\end{equation}
We have $X = G\partial_x E + K \partial_v E$, so \eqref{eq:E-deriv-x-same}, \eqref{eq:E-deriv-v-same} and \eqref{eq:inv-matrix-same} yield $|X - \bs v| \leq \eee^{-(\frac 12 + \beta)L}$.

Next, we have
\begin{equation}
\begin{aligned}
V &= -K\partial_x E + J\partial_v E \\
&= -D^{-1}({-}\la Q, \wt Q\ra\bs \phi +\bs \psi) -D^{-1}AD^{-1}D\bs v + O(\eee^{-(1+\beta)L} + |\bs v - \bs 1|\eee^{-(\frac 12 + \beta)L}).
\end{aligned}
\end{equation}
Since $|D^{-1}\la Q, \wt Q\ra - \tx{Id}| \lesssim |\bs v - \bs 1| + \eee^{-(\frac 12 + \beta)L}$,
it follows from \eqref{eq:Av-psi} that
\begin{equation}
|V - \bs \phi| \lesssim \eee^{-(1+\beta)L} + |\bs v - \bs 1|\eee^{-(\frac 12 + \beta)L}.
\end{equation}

In order to bound the derivatives of $X$ and $V$, we need the following estimates:
\begin{align}
|\partial_{\bs x}^2 E(\bs x, \bs v) - \la Q, \wt Q\ra\partial_{\bs x}\bs \phi(\bs x) + \partial_{\bs x}\bs \psi(\bs x)| &\leq \eee^{-(1+\beta)L}+|\bs v - \bs 1|\eee^{-(\frac 12 + \beta)L}, \label{eq:deriv-Exx} \\
|\partial_{\bs x}\partial_{\bs v} E(\bs x, \bs v)| &\leq \eee^{-(\frac 12+\beta)L}, \label{eq:deriv-Exv} \\
|\partial_{\bs x}A(\bs x, \bs v)\bs v + \partial_{\bs x}\bs \psi(\bs x)| &\leq \eee^{-(1+\beta)L}+|\bs v - \bs 1|\eee^{-(\frac 12 + \beta)L}, \label{eq:deriv-Ax} \\
|\partial_{\bs v}A(\bs x, \bs v)| &\leq \eee^{-(\frac 12+\beta)L}, \label{eq:deriv-Av} \\
|\partial_{\bs x}C(\bs x, \bs v)| &\leq \eee^{-(\frac 12+\beta)L}. \label{eq:deriv-Cx}
\end{align}
\noeqref{eq:deriv-Exx,eq:deriv-Exv,eq:deriv-Ax,eq:deriv-Av,eq:deriv-Cx}We skip the proof. Note that we are simply claiming
that bounds \eqref{eq:E-deriv-x-same}, \eqref{eq:E-deriv-v-same}, \eqref{eq:AC-same} and \eqref{eq:Av-psi} still hold after differentiating the terms on the left hand side.
Hence the proof amounts to repeating the computations for \eqref{eq:E-deriv-x-same}, \eqref{eq:E-deriv-v-same}, \eqref{eq:AC-same} and \eqref{eq:Av-psi},
and checking that each discarded term has negligible partial derivatives.

By the Chain Rule, we have
\begin{equation}
\label{eq:deriv-x-same}
\begin{pmatrix} \partial_{\bs x}X \\ \partial_{\bs x}V\end{pmatrix}
= \begin{pmatrix}
G & K \\ -K & J
\end{pmatrix}
\begin{pmatrix}
\partial_{\bs x}^2 E \\ \partial_{\bs x}\partial_{\bs v}E
\end{pmatrix} -
\begin{pmatrix}
G & K \\ -K & J
\end{pmatrix}
\begin{pmatrix}
\partial_{\bs x}A & \partial_{\bs x}C \\ -\partial_{\bs x}C & \partial_{\bs x}B
\end{pmatrix}
\begin{pmatrix}
X \\ V
\end{pmatrix}.
\end{equation}
We know already that $|V| \lesssim \eee^{-(\frac 12 + \beta)L}$, which,
together with \eqref{eq:inv-matrix-same}--\eqref{eq:deriv-x-same}, leads to
$|\partial_{\bs x}X| \lesssim \eee^{-(\frac 12 + \beta)L}$.

Consider now $\partial_{\bs x} V$.
From \eqref{eq:deriv-x-same} we have
\begin{equation}
\partial_x V = -K \partial_{\bs x}^2 E + K(\partial_{\bs x}A)X + O(\eee^{-(1+\beta)L}+|\bs v - \bs 1|\eee^{-(\frac 12 + \beta)L})
\end{equation}
(all the other terms resulting from the matrix multiplication are negligible).
Using again \eqref{eq:inv-matrix-same}, \eqref{eq:deriv-Exx}, \eqref{eq:deriv-Ax}
and $|X - \bs v| \lesssim \eee^{-(\frac 12 + \beta)L}$, this yields
$|\partial_x V - \partial_x \bs \phi| \leq \eee^{-(1+\beta)L}+|\bs v - \bs 1|\eee^{-(\frac 12 + \beta)L}$.

Similarly, we have
\begin{equation}
\label{eq:deriv-v-same}
\begin{pmatrix} \partial_{\bs v}X \\ \partial_{\bs v}V\end{pmatrix}
= \begin{pmatrix}
G & K \\ -K & J
\end{pmatrix}
\begin{pmatrix}
\partial_{\bs x}\partial_{\bs v} E \\ \partial_{\bs v}^2 E
\end{pmatrix} -
\begin{pmatrix}
G & K \\ -K & J
\end{pmatrix}
\begin{pmatrix}
\partial_{\bs v}A & \partial_{\bs v}C \\ -\partial_{\bs v}C & \partial_{\bs v}B
\end{pmatrix}
\begin{pmatrix}
X \\ V
\end{pmatrix}.
\end{equation}
As for $\partial_{\bs v}X$, \eqref{eq:partial-v-small} is already enough.
Regarding $\partial_{\bs v}V$, we check all the terms resulting from the matrix multiplication and we see that they are all of size $\lesssim \eee^{-(\frac 12 + \beta)L}$.
We skip this routine computation.
\end{proof}
\begin{proof}[Proof of Proposition~\ref{prop:constrained}]
Let $(\bs x(t), \bs v(t))$ be a solution of \eqref{eq:constrained} such that \eqref{eq:conv-vitesse}
holds, with $v^\infty = 1$. Since $H$ is a conservation law, we have
$H(\bs x(t), \bs v(t)) = 2H(Q)$ for all $t$. By Lemma~\ref{lem:v-bound-by-H} we obtain
\begin{equation}
\label{eq:bounds-v-equal-0}
|\bs v(t) - \bs 1| \lesssim \eee^{-\frac 12 L(t)},\qquad\text{for all }t\text{ large enough.}
\end{equation}

Denote $v(t) := v_2(t) - v_1(t)$. Directly from Lemma~\ref{lem:red-ham-same} and \eqref{eq:bounds-v-equal-0} we obtain the differential inequalities
\begin{equation}
\label{eq:bounds-v-equal-1}
|L'(t) - v(t)| \leq \eee^{-(\frac 12+\beta)L(t)},\qquad \Big|v'(t) + 2\sigma{\kappa^2}\eee^{-L(t)}\Big| \leq \eee^{-(1+\beta)L(t)}.
\end{equation}
Suppose that $\sigma = -1$. Since we assume $\lim_{t \to \infty} v(t) = 0$, the second inequality above yields
\begin{equation}
\label{eq:bounds-v-equal-2}
v(t) \leq -c_0\int_t^\infty \eee^{-L(s)}\ud s,\qquad\text{for some }c_0 > 0
\text{ and }t\text{ large enough.}
\end{equation}
Let $l := \limsup_{t \to \infty}\frac{L(t)}{\log t}$.
Note that \eqref{eq:bounds-v-equal-2} implies $l \geq 1$.

From \eqref{eq:bounds-v-equal-2} we have $v(t) < 0$ for all $t$ large enough,
hence \eqref{eq:bounds-v-equal-1} yields $L'(t) < \eee^{-(\frac 12 + \beta)L(t)}$
for all $t$ large enough. We can integrate this inequality and obtain in particular
\begin{equation}
\label{eq:bounds-v-equal-4}
l \leq \frac{1}{\frac 12 + \beta}.
\end{equation}

Let $\varepsilon \in (0, l)$. We claim that there exists a sequence $t_n \to \infty$
such that
\begin{equation}
\label{eq:bounds-v-equal-5}
L(t_n) \geq (l - \varepsilon)\log(t_n)\quad\text{and}\quad L'(t_n) \geq 0.
\end{equation}
Indeed, by the definition of $\limsup$ there exists a sequence $\tau_n \to \infty$
such that $L(\tau_n) \geq (l - \varepsilon)\log(\tau_n)$.
It suffices to take $t_n$ the smallest time such that $L(t_n) \geq (l-\varepsilon)\log(\tau_n)$.

Using \eqref{eq:bounds-v-equal-1}, \eqref{eq:bounds-v-equal-2} and \eqref{eq:bounds-v-equal-5} we get
\begin{equation}
\label{eq:bounds-v-equal-6}
c_0 \int_{t_n}^\infty \eee^{-L(s)}\ud s \leq \eee^{-(\frac 12 + \beta)L(t_n)} \leq t_n^{-(\frac 12 + \beta)(l-\varepsilon)}.
\end{equation}
If $t_n$ is large enough, then $L(s) \leq (l + \varepsilon)\log s$ for $s \geq t_n$.
This implies
\begin{equation}
\label{eq:bounds-v-equal-7}
c_0 \int_{t_n}^\infty \eee^{-L(s)}\ud s \geq c_0 \int_{t_n}^\infty s^{-l-\varepsilon}\ud s
= \frac{c_0}{l + \varepsilon - 1}t_n^{-(l + \varepsilon - 1)}.
\end{equation}
Making $t_n \to \infty$ in \eqref{eq:bounds-v-equal-6} and \eqref{eq:bounds-v-equal-7},
we obtain $l + \varepsilon - 1 \geq (\frac 12 + \beta)(l - \varepsilon)$,
thus $l > \frac{1}{\frac 12 - \beta} - \frac{\frac 32 + \beta}{\frac 12 - \beta}\varepsilon$,
which contradicts \eqref{eq:bounds-v-equal-4} if $\varepsilon$ is sufficiently small.
This shows that $\sigma = 1$.

Our next objective is to prove that there exists $\beta > 0$ such that if $(\bs x(t), \bs v(t))$ is a solution of \eqref{eq:constrained}, then
\begin{equation}
\label{eq:bounds-v-equal-3}
\begin{aligned}
|\bs x(t) - (t\bs 1 + x^\infty\bs 1 + \log (\kappa t)\bs \iota)| &\lesssim t^{-\beta}, \\
|\bs v(t) - (\bs 1+ t^{-1}\bs\iota)| &\lesssim t^{-(1+\beta)},
\end{aligned}
\end{equation}
for some $x^\infty \in \bR$.

We claim that $L(t)$ is increasing for $t$ large enough.
Let $t_0$ be large (chosen later) and $t_1 \geq t_0$. We need to show that for all $t > t_1$
we have $L(t) > L(t_1)$.
Suppose this is not the case, and let
\begin{equation}
t_2 := \sup\big\{t: L(t) = \inf_{\tau \geq t_1}L(\tau)\big\}.
\end{equation}
Then $t_2 > t_1$, $L(t_2) = \inf_{\tau \geq t_2}L(\tau)$ and $L'(t_2) = 0$.

Let $L_0 := L(t_2)$, $t_3 := \inf\{t \geq t_2: L(t) = L_0 + 1\}$.
Since $\lim_{t \to \infty}L(t) = \infty$, $t_3$ is finite.
We will show that \eqref{eq:bounds-v-equal-1} implies
\begin{equation}
\label{eq:q-mono-red-0}
L(t_3) \leq L_0 + \frac 12,
\end{equation}
which is a contradiction.

We have, for $t_0$ and thus $L_0$ large enough, and for all $t \in [t_2, t_3]$,
\begin{equation}
\label{eq:q-mono-red-1}
v'(t) \leq ({-}2\kappa^2 + \eee^{-\beta L_0}) \eee^{-L(t)} \leq {-}\frac 32 \kappa^2 \eee^{-L(t)}
\leq -\frac 32 \kappa^2 \eee^{-L_0 - 1} \leq -\frac{\kappa^2}{2}\eee^{-L_0}.
\end{equation}
Since $L'(t_2) = 0$, \eqref{eq:bounds-v-equal-1} yields $v(t_2) \leq \eee^{-(\frac 12 + \beta)L_0}$.
Integrating \eqref{eq:q-mono-red-1}, we get
\begin{equation}
\label{eq:q-mono-red-2}
v(t) \leq \eee^{-(\frac 12 + \beta)L_0} -\frac{\kappa^2}{2}\eee^{-L_0}(t - t_2), \qquad\text{for all }t\in [t_2, t_3].
\end{equation}
Using \eqref{eq:bounds-v-equal-1} again we obtain
\begin{equation}
L'(t) \leq 2\eee^{-(\frac 12 + \beta)L_0} - \frac{\kappa^2}{2}\eee^{-L_0}(t-t_2), \qquad\text{for all }t\in [t_2, t_3].
\end{equation}
We now integrate for $t$ between $t_2$ and $t_3$:
\begin{equation}
\begin{aligned}
L(t_3) - L(t_2)
&\leq  \int_{t_2}^{t_3}\Big(2\eee^{-(\frac 12 + \beta)L_0} - \frac{\kappa^2}{2}\eee^{-L_0}(t-t_2)\Big)\ud t \\
&=2\eee^{-(\frac 12 + \beta)L_0}(t_3 - t_2) - \frac{\kappa^2}{4}\eee^{-L_0}(t_3 - t_2)^2 \leq\frac{4}{\kappa^2}\eee^{-2\beta L_0},
\end{aligned}
\end{equation}
so that \eqref{eq:q-mono-red-0} follows if $L_0$ is large enough.
This shows that $L(t)$ is increasing for $t$ large enough.

Let $r(t) := v(t) - 2\kappa \eee^{-L(t)/2}$. Note that $\lim_{t\to\infty}r(t) = 0$. Using \eqref{eq:bounds-v-equal-1},
we obtain
\begin{equation}
\label{eq:q-mono-red-3}
r'(t) = \kappa \eee^{-L(t)/2}r(t) + O(\eee^{-(1+\beta)L(t)}).
\end{equation}
This implies
\begin{equation}
\label{eq:q-mono-red-3p}
|r(t)| \lesssim \eee^{-(\frac 12 + \beta)L(t)},\qquad \text{for }t\text{ large enough.}
\end{equation}
Indeed, suppose $r(\tau_n) \geq C_0 \eee^{-(\frac 12 + \beta)L(\tau_n)}$
for any large constant $C_0 > 0$ and a sequence $\tau_n \to \infty$.
Let $t_n$ be the largest time such that $r(t_n) = C_0 \eee^{-(\frac 12 + \beta)L(\tau_n)}$.
Since $L(t)$ is increasing for large $t$, we have $r(t_n) \geq C_0 \eee^{-(\frac 12 + \beta)L(t_n)}$.
But also $r'(t_n) \leq 0$. This contradicts \eqref{eq:q-mono-red-3} if $C_0$ is large enough.
The case $r(\tau_n) \leq -C_0 \eee^{-(\frac 12 + \beta)L(\tau_n)}$ is similar.

From \eqref{eq:bounds-v-equal-1} and \eqref{eq:q-mono-red-3p} we deduce
\begin{equation}
\label{eq:q-mono-red-4}
|L'(t) - 2\kappa\eee^{-L(t)/2}| \lesssim \eee^{-(\frac 12 + \beta)L(t)}\ \Leftrightarrow\ \big|(\eee^{L(t)/2})' - \kappa\big| \lesssim \eee^{-\beta L(t)} \ll 1.
\end{equation}
Integrating, we obtain $L(t) = 2\log t + O(1)$. We can reinsert this to \eqref{eq:q-mono-red-4},
integrate, and obtain the more precise bound
\begin{equation}
\label{eq:q-mono-red-5}
|L(t) - 2\log(\kappa t)| \lesssim t^{-2\beta}.
\end{equation}
This implies $|\kappa^2 \eee^{-L(t)} - t^{-2}| \lesssim t^{-2 - 2\beta}$.
Thus Lemma~\ref{lem:red-ham-same} and \eqref{eq:bounds-v-equal-0} yield
\begin{equation}
|v_1'(t) - t^{-2}| \lesssim t^{-2-2\beta}, \qquad |x_1'(t) - v_1| \lesssim t^{-1-2\beta}.
\end{equation}
Since we assume $\lim_{t\to\infty}v_1(t) = 1$, the first inequality above yields $|v_1(t) - (1 - 1/t)| \lesssim t^{-(1+2\beta)}$.
Now the second inequality yields $|x_1(t) - (t + x_1^\infty - \log(\kappa t))| \lesssim t^{-2\beta}$
for some $x_1^\infty \in \bR$. Analogously, we obtain
$|v_2(t) - (1 + 1/t)| \lesssim t^{-(1+2\beta)}$ and
$|x_2(t) - (t + x_2^\infty + \log(\kappa t))| \lesssim t^{-2\beta}$.
Since $L(t) = x_2(t) - x_1(t)$, \eqref{eq:q-mono-red-5} yields $x_1^\infty = x_2^\infty = x^\infty$.
This finishes the proof of \eqref{eq:bounds-v-equal-3}.

We proceed to the proof of existence and uniqueness.
We introduce the new ``time'' variable $s$ by $t = {\kappa}^{-1}\eee^{s}$.
Given $x^\infty \in \bR$, we define $(\wt {\bs x}(s), \wt {\bs v}(s))$ by
\begin{equation}
\label{eq:corresp}
\begin{aligned}
\wt{\bs x}(s) &:= \bs x(\kappa^{-1}\eee^s) - \big(\kappa^{-1}\eee^s\bs 1 + x^\infty\bs 1 + s\bs \iota\big), \\
\wt{\bs v}(s) &:= \kappa^{-1}\eee^s\big(\bs v(\kappa^{-1}\eee^s) - (\bs 1 + \kappa \eee^{-s}\bs\iota)\big).
\end{aligned}
\end{equation}
We have
\begin{equation}
\begin{aligned}
\wt{\bs x}'(s) &= \kappa^{-1}\eee^s\bs x'(\kappa^{-1}\eee^s) - \big(\kappa^{-1}\eee^s\bs 1 +\bs \iota\big), \\
\wt{\bs v}'(s) &= \kappa^{-2}\eee^{2s}\bs v'(\kappa^{-1}\eee^s) + \kappa^{-1}\eee^s \bs v(\kappa^{-1}\eee^s) - \kappa^{-1}\eee^s \bs 1 \\
&= \kappa^{-2}\eee^{2s}\bs v'(\kappa^{-1}\eee^s) + \wt{\bs v}(s) + \bs\iota.
\end{aligned}
\end{equation}
Set
\begin{equation}
\begin{aligned}
\wt X(s, \wt{\bs x}, \wt{\bs v}) &:= {\kappa}^{-1}\eee^{s}X(
{\kappa}^{-1}\eee^{s}\bs 1 + x^\infty\bs 1 + s\bs \iota + \wt{\bs x},
\bs 1 + {\kappa}\eee^{-s}\bs \iota + \kappa\eee^{-s}\wt{\bs v}) - (\kappa^{-1}\eee^s\bs 1 + \bs \iota) - \wt{\bs v}, \\
\wt V(s, \wt{\bs x}, \wt{\bs v}) &:= \kappa^{-2}\eee^{2s} V({\kappa}^{-1}\eee^{s}\bs 1 + x^\infty\bs 1 + s\bs \iota + \wt{\bs x},
\bs 1 + {\kappa}\eee^{-s}\bs \iota + \kappa\eee^{-s}\wt{\bs v}) + \bs\iota - (\wt x_2 - \wt x_1)\bs \iota,
\end{aligned}
\end{equation}
so that
\begin{equation}
\begin{aligned}
\wt{X}(s, \wt{\bs x}(s), \wt{\bs v}(s)) &= \kappa^{-1}\eee^s X(\bs x(\kappa^{-1} \eee^s), \bs v(\kappa^{-1} \eee^s))
- (\kappa^{-1}\eee^s\bs 1 + \bs \iota) - \wt{\bs v}(s), \\
\wt{V}(s, \wt{\bs x}(s), \wt{\bs v}(s)) &= \kappa^{-2}\eee^{2s} V(\bs x(\kappa^{-1} \eee^s), \bs v(\kappa^{-1} \eee^s))
+\bs \iota - (\wt x_2(s) - \wt x_1(s))\bs\iota.
\end{aligned}
\end{equation}
We obtain that $(\bs x(t), \bs v(t))$ solves \eqref{eq:constrained} if and only if $(\wt{\bs x}(s), \wt {\bs v}(s))$
solves
\begin{equation}
\begin{pmatrix} \wt{\bs x}'(s) \\ \wt{\bs v}'(s)\end{pmatrix}
= T\begin{pmatrix} \wt{\bs x}(s) \\ \wt{\bs v}(s)\end{pmatrix}
+ \begin{pmatrix} \wt X(s, \wt{\bs x}(s), \wt{\bs v}(s)) \\ \wt V(s, \wt{\bs x}(s), \wt{\bs v}(s))\end{pmatrix},
\end{equation}
where
\begin{equation}
\begin{aligned}
T\begin{pmatrix} \wt{\bs x} \\ \wt{\bs v}\end{pmatrix} := \begin{pmatrix} \wt{\bs v} \\ \wt{\bs v} + (\wt x_2 - \wt x_1)\bs\iota\end{pmatrix}\ \Leftrightarrow \  T &:= \begin{pmatrix} 0 & 0 & 1 & 0 \\ 0 & 0 & 0 & 1 \\ 1 & -1 & 1 & 0 \\ -1 & 1 & 0 & 1\end{pmatrix}.
\end{aligned}
\end{equation}
Note that the lower left quarter of the matrix $T$ is related to the linearisation of $\bs\phi(\bs x)$.
The matrix $T$ has eigenvalues $-1, 0, 1, 2$, and $\cY_s := (-1,1,1,-1) = (\bs\iota, -\bs \iota)$ satisfies $T\cY_s = -\cY_s$.

Observe that $L(t) \sim 2s$ from \eqref{eq:q-mono-red-5}. Thus, from Lemma~\ref{lem:red-ham-same} we obtain that there exists $\beta > 0$
such that for $s$ large enough and $|\wt{\bs x}| + |\wt{\bs v}| +|\wt{\bs y}| +|\wt{\bs w}| \leq 1$ we have
\begin{gather}
|\wt X(s, \wt{\bs x}, \wt{\bs v})| + |\wt V(s, \wt{\bs x}, \wt{\bs v})| \lesssim \eee^{-\beta s}
+ |\wt{\bs x}|^2, \\
\begin{aligned}
&|\wt X(s, \wt{\bs x}, \wt{\bs v}) - \wt X(s, \wt{\bs y}, \wt{\bs w})| +
|\wt V(s, \wt{\bs x}, \wt{\bs v}) - \wt V(s, \wt{\bs y}, \wt{\bs w})| \\
&\qquad\qquad\qquad\lesssim (\eee^{-\beta s}
+ |\wt{\bs x}| + |\wt{\bs y}|)(|\wt{\bs x} - \wt{\bs y}| + |\wt{\bs v} - \wt{\bs w}|).
\end{aligned}
\end{gather}
From \eqref{eq:bounds-v-equal-3}, we obtain that if $(\bs x(t), \bs v(t))$
solves \eqref{eq:constrained}, then $|\wt{\bs x}(s)| + |\wt{\bs v}(s)| \leq \eee^{-\beta s}$ for $s$ large.
Let $\{(\wt {\bs x}_a, \wt {\bs v}_a): a \in\bR\}$ be the family of exponentially decaying solutions given by Proposition~\ref{prop:one-negative}. Let $(\bs x, \bs v)$ be the solution of \eqref{eq:constrained}
corresponding to $(\wt {\bs x}_0, \wt {\bs v}_0)$ through \eqref{eq:corresp}, in other words
\begin{equation}
\begin{aligned}
\bs x(t) &= \wt{\bs x}_0\big(\log(\kappa t)\big) + t\bs 1 + x^\infty \bs 1 + \log(\kappa t)\bs\iota, \\
\bs v(t) &= t^{-1}\wt{\bs v}_0\big(\log(\kappa t)\big) + \bs 1 + t^{-1}\bs\iota.
\end{aligned}
\end{equation}
Let $t^\infty \in \bR$ and define $(\sh{\bs x}, \sh{\bs v})$ by
\begin{equation}
(\sh{\bs x}(t), \sh{\bs v}(t)) := (\bs x(t - t^\infty) + t^\infty\bs 1, \bs v(t - t^\infty)),
\end{equation}
which is also a solution of \eqref{eq:constrained} for $t$ large enough. Let $(\sh{\wt{\bs x}}, \sh{\wt{\bs v}})$
be defined by \eqref{eq:corresp} with $(\bs x, \bs v)$ replaced by $(\sh{\bs x}, \sh{\bs v})$.
We thus have
\begin{equation}
\begin{aligned}
\sh{\wt{\bs x}}(s) &= \sh{\bs x}(\kappa^{-1}\eee^s) - (\kappa^{-1}\eee^s \bs 1 + x^\infty \bs 1 + s\bs \iota) \\
&= \bs x(\kappa^{-1}\eee^s - t^\infty) - \big((\kappa^{-1}\eee^s - t^\infty)\bs 1 + x^\infty \bs 1 + s\bs\iota\big) \\
&= \wt{\bs x}_0\big(\log(\eee^s - \kappa t^\infty)\big) + \big(\log(\eee^s - \kappa t^\infty) - s\big)\bs\iota \\
&= \wt{\bs x}_0\big(s + \log(1 - \kappa t^\infty \eee^{-s})\big) + \log(1 - \kappa t^\infty\eee^{-s})\bs\iota
\end{aligned}
\end{equation}
and
\begin{equation}
\begin{aligned}
\sh{\wt{\bs v}}(s) &= \kappa^{-1}\eee^s\big(\sh{\bs v}(\kappa^{-1}\eee^s) - (\bs 1 + \kappa \eee^{-s}\bs\iota)\big)
= \kappa^{-1}\eee^s\big({\bs v}(\kappa^{-1}\eee^s - t^\infty) - (\bs 1 + \kappa \eee^{-s}\bs\iota)\big) \\
&= \kappa^{-1}\eee^s\big((\kappa^{-1}\eee^s - t^\infty)^{-1}\wt{\bs v}_0\big(\log(\eee^s - \kappa t^\infty)\big)
+ \big((\kappa^{-1}\eee^s - t^\infty)^{-1} - \kappa\eee^{-s}\big)\bs\iota\big) \\
&=(1 - \kappa t^\infty \eee^{-s})^{-1}\wt{\bs v}_0\big(s + \log(1 - \kappa t^\infty \eee^{-s})\big)
+ \big((1 - \kappa t^\infty \eee^{-s})^{-1} - 1\big)\bs\iota.
\end{aligned}
\end{equation}
Since $|\wt{\bs x}_0'(s)| \lesssim \eee^{-\beta s}$ for all $s$ large enough, we have
$
\big|\wt{\bs x}_0\big(s + \log(1 - \kappa t^\infty \eee^{-s})\big) - \wt{\bs x}_0(s)\big| \lesssim \eee^{-(1+\beta)s}
$, thus
\begin{equation}
\sh{\wt{\bs x}}(s) - \wt{\bs x}_0(s) = -\kappa t^\infty \eee^{-s}\bs\iota + O(\eee^{-(1+\beta)s}).
\end{equation}
Similarly,
\begin{equation}
\sh{\wt{\bs v}}(s) - \wt{\bs v}_0(s) = \kappa t^\infty \eee^{-s}\bs\iota + O(\eee^{-(1+\beta)s}).
\end{equation}
Proposition~\ref{prop:one-negative} yields $(\sh{\wt{\bs x}}, \sh{\wt{\bs v}}) = (\wt{\bs x}_a, \wt{\bs v}_a)$
with $a := -\kappa t^\infty$, in particular there is a one-to-one correspondence between $t^\infty$ and $a$.
\end{proof}

\providecommand{\noopsort}[1]{}

\end{document}